\documentclass[11pt]{article}
\usepackage[margin=1.1in]{geometry}
\usepackage{amsmath,amssymb,amsthm,mathtools}
\usepackage{enumitem}
\usepackage{xcolor}
\usepackage[colorlinks=true,linkcolor=blue,citecolor=blue]{hyperref}
\usepackage{graphicx}
\usepackage{bbm}
\usepackage{extarrows}
\usepackage{tikz}
\usetikzlibrary{positioning,calc,decorations.pathmorphing}

\newtheorem{theorem}{Theorem}[section]
\newtheorem{proposition}[theorem]{Proposition}
\newtheorem{lemma}[theorem]{Lemma}
\newtheorem{corollary}[theorem]{Corollary}
\theoremstyle{definition}

\newtheorem{namedassumptioninner}{Assumption}

\theoremstyle{remark}
\newtheorem{remark}[theorem]{Remark}

\newcommand{\R}{\mathbb{R}}
\newcommand{\Z}{\mathbb{Z}}
\newcommand{\E}{\mathbb{E}}
\newcommand{\bP}{\mathbb{P}}
\newcommand{\bN}{{\mathbb{N}_0}}                 
\newcommand{\cM}{{\mathcal{M}_F(\R^d)}}          
\newcommand{\cMM}{{\mathcal{M}_F(\mathcal{M}_F(\R^d))}}
\newcommand{\Cc}{C_c^+(\R^d)}
\newcommand{\Ccg}{C_c(\R^d)}                   
\newcommand{\T}{\mathbb{T}}                    
\newcommand{\supp}{\operatorname{supp}}
\newcommand{\dist}{\operatorname{dist}}
\newcommand{\Ran}{\mathcal{R}}                 
\newcommand{\eps}{\varepsilon}
\newcommand{\nuh}{\hat{\nu}}
\newcommand{\Nh}{\hat{\mathbb{N}}_0}
\newcommand{\fl}[1]{\lfloor #1\rfloor}

\title{Range convergence and sharp one-arm asymptotics\\for the critical percolation cluster in high dimensions}
\author{%
  Manuel Cabezas\thanks{Facultad de Matem\'aticas, Pontificia Universidad Cat\'olica de Chile, Santiago, Chile.\ \texttt{mncabeza@uc.cl}}
  \and David A.\ Croydon\thanks{Research Institute for Mathematical Sciences, Kyoto University, Kyoto, Japan.\ \texttt{croydon@kurims.kyoto-u.ac.jp}}
  \and Alexander Fribergh\thanks{{D\'epartement de math\'ematiques et de statistique, Universit\'e de Montr\'eal, Montr\'eal, Canada}.\ \texttt{alexander.fribergh@umontreal.ca}}
  \and Noe Kawamoto\thanks{Department of Mathematics, Graduate School of Science, Kyoto University, Kyoto, Japan.\ \texttt{kawamoto.noe.4u@kyoto-u.ac.jp}}
}
\date{\today}

\numberwithin{equation}{section}

\definecolor{OldOrange}{rgb}{0.85,0.45,0}

\definecolor{ReviewPink}{rgb}{0.90,0.10,0.55}

\theoremstyle{plain}
\newtheorem{thm}[theorem]{Theorem}
\newtheorem{lem}[theorem]{Lemma}

\theoremstyle{definition}
\newtheorem{defn}[theorem]{Definition}
\theoremstyle{remark}

\theoremstyle{plain}
\newcommand{\ind}[1]{\mathbbm{1}\big\{#1\big\}}
\newcommand{\tnorm}[1]{\vert\!\vert\!\vert #1 \vert\!\vert\!\vert}

\begin{document}
\maketitle

\begin{abstract}
For critical Bernoulli bond percolation on $\Z^d$ in the high-dimensional
regime, we prove that the rescaled empirical measure of the cluster of the
origin converges, in a suitable $\sigma$-finite sense, to the total occupation measure of super-Brownian motion. Combined with
a uniform lower mass bound for the critical cluster with respect to the extrinsic (Euclidean) metric, which we also prove and which is of independent interest, the measure
convergence further yields the convergence of the rescaled cluster as a compact set,
in the Hausdorff metric. As a consequence, we are able to obtain the sharp one-arm asymptotics
$r^2\,\bP(0\leftrightarrow \partial B_r)\to \theta_1\in(0,\infty)$, hence refining a result of Kozma and Nachmias.
\end{abstract}

\tableofcontents

\section{Introduction}

Consider critical Bernoulli bond percolation on the integer lattice $\Z^d$ in
high dimensions, and denote by $\mathcal C$ the open cluster of the origin. A large critical cluster is an intricate random fractal, and a central conjecture in the field --- going back to Aizenman \cite{Aiz97} and to Hara and Slade \cite{HS00a,HS00b} --- is that, after rescaling, it is described by the total occupation
measure of a super-Brownian motion (SBM).

In a recent article, Chatterjee, Chinmay, Hanson and Sosoe \cite{CCHS26}
proved that the rescaled $k$-point functions of the critical cluster converge.
Their limits are, after a suitable normalization, precisely the moment
densities of the total occupation measure
$\mu:=\int_0^\infty X_t\,dt$ of a super-Brownian motion $(X_t)_{t\geq 0}$ under
its canonical measure $\bN$ (both recalled in Section~\ref{sec:SBM}). Our first
main theorem upgrades this moment-level information to convergence of the law
of the cluster itself. Writing
\begin{equation}\label{eq:empirical}
\mu_n \;:=\; \frac{1}{\mathfrak a\, n^4}\sum_{x\in \mathcal C}\delta_{x/n},
\qquad
\nu_n \;:=\; n^2\,\bP\bigl(\mu_n\in\cdot\,\bigr),
\end{equation}
for the rescaled empirical measure and its intensity (where $\mathfrak a$ is the
two-point constant appearing in \eqref{eq:guiding-kp} below), we show that $\nu_n$ converges
to $\bN$ in a suitable $\sigma$-finite sense.

Our second result is a uniform \emph{lower mass bound} for the critical
cluster with respect to the extrinsic (Euclidean) metric: no macroscopic region that the cluster reaches can carry
asymptotically negligible volume. It is the analogue, for critical
percolation, of the lower mass bounds that underlie Gromov--Hausdorff--Prohorov
convergence of random metric measure spaces \cite{ALW16}. It is proved by
percolation methods independent of the rest of the paper and is of interest in
its own right.

Combining these two results, we show that the cluster of the origin --- conditioned
to be large and suitably rescaled --- converges, as a compact set, to the
support $\Ran=\supp\mu$ of the occupation measure. As a consequence, we obtain
the exact asymptotics of the one-arm probability, refining a result of Kozma and Nachmias \cite{KN11}.

\subsection{Main results}\label{sec:mainresults}

We now state the four main results.

Since the limiting measure $\bN$ is $\sigma$-finite with infinite total mass, the convergence in our first theorem is formulated restricted to the events $\{\mu(\R^d)>\eps\}$, in the same spirit as Holmes and Perkins \cite{HP07} (to which we refer for this mode of convergence towards $\sigma$-finite limits). Towards stating the result, we define $\mathcal{M}_F(X)$ to be the space of finite measures on a space $X$.

\begin{theorem}[Measure convergence]\label{thm:mainA}
Consider critical nearest-neighbour bond percolation on $\Z^d$ with $d\geq 11$. For $\gamma>0$, let $\mathbb{N}_{0,\gamma}$ denote the canonical measure of the $(\tfrac12\Delta,\gamma)$-super-Brownian motion (defined in Section~\ref{sec:SBM}). Then, for every
$\eps>0$,
\[
\nu_n\bigl(\,\cdot\,;\ \mu(\R^d)>\eps\bigr)
\;\xRightarrow[n\to\infty]{}\;
\mathbb{N}_{0,\lambda/\mathfrak a}\bigl(\,\cdot\,;\ \mu(\R^d)>\eps\bigr)
\]
weakly in $\cMM$, where $\mathfrak a,\lambda>0$ are the constants 
of \eqref{eq:guiding-kp} below.

\end{theorem}

We write
$Q_{z,r}:=\{y\in\Z^d:\lVert y-z\rVert_\infty\le r\}$ for the $\ell_\infty$ lattice
box of radius $r$ centred at $z$.

\begin{theorem}[Uniform lower mass bound]\label{thm:mainLMB}
Consider critical nearest-neighbour bond percolation on $\Z^d$ with $d\geq 11$. For every $\delta>0$,
\[
\lim_{\eta\rightarrow 0}\;\limsup_{R\rightarrow\infty}\;
\bP\left(\inf_{x\in \mathcal C}\bigl|\mathcal C\cap Q_{x,\delta R}\bigr|
\leq \eta R^4 \:\Big\vert\:0\leftrightarrow (Q_{0,R})^c\right)=0,
\]
and the same holds when the conditioning event is replaced by
$\{|\mathcal C|\geq R^4\}$.
\end{theorem}

The remaining two results are geometric consequences; for the relevant
literature see Section~\ref{sec:related}. Combining the measure
convergence (Theorem~\ref{thm:mainA}) with a variant of the
lower mass bound (Lemma~\ref{lem:nothin}), we obtain convergence of the cluster as a compact
set.

\begin{theorem}[Range convergence]\label{thm:mainD}
Consider critical nearest-neighbour bond percolation on $\Z^d$ with $d\geq 11$. For every $\eps>0$, conditionally on the event $\{\mu_n(\R^d)>\eps\}$, the pair
$\bigl(\mu_n,\ n^{-1}\mathcal C\bigr)$ converges in distribution to $(\mu,\supp\mu)$ under the
conditioned canonical measure $\mathbb{N}_{0,\lambda/\mathfrak a}(\,\cdot\mid\mu(\R^d)>\eps)$, with $\mathfrak a,\lambda$ as in \eqref{eq:guiding-kp}), jointly in $\cM\times\mathcal K(\R^d)$. In particular the rescaled cluster $n^{-1}\mathcal C$
converges in distribution to the range of super-Brownian motion in the Hausdorff metric on the
space $\mathcal K(\R^d)$ of compact sets.
\end{theorem}

A variant of the range convergence also yields precise asymptotics for the one-arm probability.

\begin{theorem}[Sharp one-arm asymptotics]\label{thm:mainC}
Consider critical nearest-neighbour bond percolation on $\Z^d$ with $d\geq 11$. Then
\[
\lim_{r\to\infty}\ r^2\,\bP\bigl(0\leftrightarrow\partial B_r\bigr)
\;=\;\theta_1\;\in(0,\infty),
\]
where, with $\mathfrak a,\lambda$ as in \eqref{eq:guiding-kp}
\[
\theta_1\;=\;\mathbb{N}_{0,\lambda/\mathfrak a}\bigl(\Ran\not\subset\overline B_1\bigr)\;=\;\frac{\mathfrak a}{\lambda}\,\mathbb{N}_{0,1}\bigl(\Ran\not\subset\overline B_1\bigr),
\]
and $\mathbb{N}_{0,1}$ is the standard canonical measure of super-Brownian motion (i.e.\ the $(\tfrac12\Delta,1)$-super-Brownian motion); thus $\theta_1$ equals $\mathfrak a/\lambda$ times the probability that the range of standard super-Brownian motion is not contained in the unit ball.
\end{theorem}

This sharpens the Kozma--Nachmias bound
$\bP(0\leftrightarrow\partial B_r)\asymp r^{-2}$ \cite{KN11} to an exact
asymptotic and identifies its constant.

\begin{remark}
For brevity, in the remainder of the paper we write $\bN$ for $\mathbb{N}_{0,\lambda/\mathfrak a}$.
\end{remark}

\subsection{Discussion on the dimension}

Our results rest on two properties of critical percolation in high dimensions,
whose history is reviewed in Section~\ref{sec:assumptions}. For
$x_1,\dots,x_k\in\Z^d$, let $\tau_k(x_1,\dots,x_k)$ denote the probability
that the vertices $x_1,\dots,x_k$ all belong to a common open cluster; in
particular
$\tau_2(0,x)$ is the probability that $x$ is connected to the origin. Let us
introduce the two key properties.
\begin{itemize}
\item \emph{Uniform two-point function bound.} There is a constant $K_1<\infty$ such that
\begin{equation}\label{eq:2pt}
K_1^{-1}\bigl(1\vee|x|\bigr)^{-(d-2)}
\;\le\;\tau_2(0,x)\;\le\;
K_1\bigl(1\vee|x|\bigr)^{-(d-2)}
\qquad\text{for all }x\in\Z^d,
\end{equation}
where $|x|$ denotes the Euclidean norm and $1\vee|x|=\max(1,|x|)$: the
two-point function decays at the same rate as the Green function of simple
random walk.
\item \emph{Convergence of $k$-point functions.} For every $m\ge1$ and all pairwise distinct
$y_1,\dots,y_m\in\R^d\setminus\{0\}$,
\begin{equation}\label{eq:guiding-kp}
n^{(d-4)m+2}\,\tau_{m+1}\bigl(0,\fl{ny_1},\dots,\fl{ny_m}\bigr)
\;\xrightarrow[n\to\infty]{}\;
\mathfrak a\,\lambda^{\,m-1}\sum_{T\in\T_m}I_T(y_1,\dots,y_m),
\end{equation}
where $\fl{ny}\in\Z^d$ is the componentwise integer part of $ny$, the sum runs
over the set $\T_m$ of binary trees with $m+1$ labelled leaves $\{0,1,\dots,m\}$, the leaf $0$ being pinned at the origin, $I_T$ is the
\emph{tree integral} obtained by placing a Green kernel on every edge of $T$
and integrating out the internal vertices (Section~\ref{sec:treedef}), and $\mathfrak a,\lambda>0$ are explicit positive constants determined by the model and given in \cite{CCHS26} (their decomposition is recorded in Appendix~\ref{app:dictionary}). The convergence \eqref{eq:guiding-kp} is the main theorem of
Chatterjee, Chinmay, Hanson and Sosoe \cite{CCHS26}.
\end{itemize}

We point out that \eqref{eq:guiding-kp} does not imply \eqref{eq:2pt}: the
former is a scaling limit at macroscopic distances $|x|\asymp n$, whereas the
latter is a uniform two-sided bound valid at all scales, and it is in this
uniform form that the two-point function enters several of our arguments.

In the body of the paper the results are proved with \eqref{eq:2pt} and
\eqref{eq:guiding-kp} as hypotheses; the condition $d\ge 11$ serves only to
guarantee them. More precisely:
\begin{itemize}
\item Theorem~\ref{thm:mainA} (measure convergence) requires $d>6$ together
with \eqref{eq:2pt} and \eqref{eq:guiding-kp};
\item Theorem~\ref{thm:mainLMB} (lower mass bound) requires $d>6$ together
with \eqref{eq:2pt} only;
\item Theorems~\ref{thm:mainD} and~\ref{thm:mainC} (range convergence and
one-arm asymptotics) require $d>6$, \eqref{eq:2pt} and \eqref{eq:guiding-kp},
since they combine the two previous theorems.
\end{itemize}

Both properties \eqref{eq:2pt} and \eqref{eq:guiding-kp} are known to hold for the \emph{nearest-neighbour} model in
dimension $d\ge 11$ (the history and references are reviewed in
Section~\ref{sec:assumptions}).
The two further inputs we
use are the Aizenman--Newman tree-graph inequality \cite{AN84}, which holds
unconditionally, and the Kozma--Nachmias one-arm bound \cite{KN11}, which
holds whenever $d>6$ and \eqref{eq:2pt} are verified. For this reason we state
the four main theorems for the nearest-neighbour model with $d\ge 11$.

For sufficiently spread-out models, \eqref{eq:2pt} is known in all $d>6$
(Hara, van der Hofstad and Slade \cite{HHS03}), and the authors of \cite{CCHS26}
expect \eqref{eq:guiding-kp} to extend to that setting. It
is natural to expect that our results extend as well, but we have not verified
that our arguments apply verbatim to spread-out models, and we do not pursue
this here.

\subsection{Related work}\label{sec:related}

The results connect two lines of work.

\emph{Super-Brownian motion as a universal scaling limit.} That large critical
structures in high dimensions are governed by super-Brownian motion is a theme
with a long history, surveyed in \cite{Sla99} and in the monograph of
Heydenreich and van der Hofstad \cite{HvdH17}; its analytic
backbone is the lace expansion, developed for percolation by Hara and Slade
\cite{HS90,HS00a,HS00b}, with the sharpest information on the two- and
higher-point functions now available through
\cite{CCHS26,FH17,Har08,HHS03}.

The first rigorous
scaling limits concerned models conditioned by their size, for which the limit
is the integrated super-Brownian excursion (ISE) of Aldous \cite{AldISE}:
Derbez and Slade proved that sufficiently spread-out lattice trees in
dimensions $d>8$, conditioned to contain exactly $n$ vertices, converge to ISE
\cite{DS98}, and Hara and Slade proved partial results in that direction for the
critical percolation cluster in high dimensions \cite{HS00a,HS00b}. In the
unconditioned setting, van der Hofstad and Slade proved convergence of the
moment measures for critical oriented percolation above $4+1$ dimensions
\cite{vdHS03}, and Bramson, Cox and Le Gall proved convergence in distribution
of rescaled voter-model clusters to super-Brownian motion \cite{BCL01}.
For unoriented percolation, a definitive connection with
super-Brownian motion was established only recently, in a breakthrough of
Chatterjee, Chinmay, Hanson and Sosoe \cite{CCHS26} proving the convergence
\eqref{eq:guiding-kp} of the $k$-point functions. We also note here that Hutchcroft has recently derived a super-Brownian motion limit similar to our Theorem \ref{thm:mainA} for a long-range model in high dimensions \cite{hutchcroft}.

Holmes and Perkins \cite{HP} studied a problem closely
related to ours: they isolated general conditions on a lattice model under
which the rescaled range converges to the range of super-Brownian motion, and
verified them for the voter model ($d\ge2$), for sufficiently spread-out
critical oriented percolation and contact processes ($d>4$), and for
sufficiently spread-out critical lattice trees ($d>8$). The present paper adds
unoriented Bernoulli percolation to this picture.

\emph{Random walk on critical clusters, and the origin of the lower mass
bound.} Our initial motivation for proving the lower mass bound came not from percolation itself but from the study of random walk on critical structures --- the ``ant in the labyrinth'', a phrase coined by de Gennes \cite{deGennes76}. In the program of Ben Arous, Cabezas and Fribergh \cite{BACF1} on the scaling limit of the random walk on large critical clusters in high dimensions (which includes deriving convergence for the random walk on a critical branching random walk \cite{BACF2}), one of the sufficient conditions for convergence asks, intuitively, that no part of the graph that is large (in $\Z^d$ or in intrinsic distance) may carry negligible volume. Such volume growth conditions are standard in the theory of scaling limits of random walks via resistance forms, where they underlie Gromov--Hausdorff-type convergence; see Athreya, L\"ohr and Winter \cite{ALW16} for general topological background, and also Croydon, Hambly and Kumagai \cite{CHK}, Croydon \cite{Cres} and Noda \cite{Noda} for general resistance form convergence results (also \cite{Cint} for a specific conjecture about the random walk on a critical percolation cluster in high-dimensions. (As for applications of this general theory, see \cite{Cer} for work on the critical Erd\H{o}s-R\'enyi graph, \cite{ACHS, BCK17} for work on the uniform spanning tree, \cite{AC} for work on critical percolation clusters on hyperbolic random half-planar triangulations, and the recent remarkable work of Dankovic et al.\ on the random walk on critical percolation cluster on the two-dimensional triangular lattice \cite{Miller}.) We also note that understanding volume growth has been similarly important in deriving heat kernel estimates on critical structures (see the surveys of \cite{Kumanom,Kumstf}, as well as \cite{BJKS,BK,HHH,KM08,KNAO} for important examples, including the work of Kozma and Nachmias in which the spectral dimension of the incipient infinite cluster was derived). These considerations motivated us to derive the lower mass bound for the critical cluster. Combined with the $k$-point convergence of \cite{CCHS26}, it also yields the one-arm constant, by the following mechanism. The measure convergence determines the asymptotic probability that the cluster deposits macroscopic mass at distance $r$, while the lower mass bound shows that a cluster cannot \emph{reach} distance $r$ without doing so; reaching and depositing mass are therefore asymptotically the same event, and the exit probability inherits the exact super-Brownian asymptotics.

\emph{Independent parallel work of Blanc-Renaudie and Hutchcroft.} As this paper was being finalized, the preprint of Blanc-Renaudie and Hutchcroft \cite{BRH26} appeared. There, the first-order asymptotics of the $k$-point function --- the conjecture of Aizenman and Newman \cite{AN84} in its general form, with the mutual distances between the $k$ points tending to infinity at arbitrary relative rates --- are established under the perturbative hypotheses of the lace expansion (nearest-neighbour in sufficiently high dimensions, or sufficiently spread-out when $d>6$), and are combined with a detailed analysis of the intrinsic geometry of the cluster to obtain a scaling limit in topologies considerably stronger than the ones considered here: the cluster converges as an embedded metric measure space to the continuum random tree with its Brownian embedding, jointly for the chemical, pivotal and resistance distances, and first-order one-arm asymptotics with identified constants follow, in the Euclidean as well as the chemical metric. The two works were carried out independently, and the approaches are complementary. In \cite{BRH26} the lace expansion is confined to a single step, namely the proof of a ``two-blob'' mixing estimate stating that, conditionally on a connection between two distant vertices, the configurations near the two endpoints decouple into independent copies of the incipient infinite cluster; beyond this estimate and the first-order asymptotics of the two-point function, no further lace expansion input is required there (see \cite[Theorems~1.11 and~1.12]{BRH26}). Taking the $k$-point convergence \eqref{eq:guiding-kp} of \cite{CCHS26} as an input, the route followed here is short and self-contained, makes no use of the lace expansion, and applies in particular to the nearest-neighbour model in dimensions $d\ge11$, which is not covered by the unconditional results of \cite{BRH26}: in that setting the two-point asymptotics are known \cite{FH17}, but the two-blob estimate has at present not been verified (see \cite[Remark~1.3]{BRH26}, where it is suggested that such an estimate could be derived from the construction of the incipient infinite cluster of \cite{CCHSiic}, though this is not pursued there). Finally, no extrinsic lower mass bound in the form of Theorem~\ref{thm:mainLMB} is stated in \cite{BRH26}: the tightness analysis there proceeds via a lower volume bound for chemical balls together with an equicontinuity estimate for the Euclidean embedding, conditionally on the total volume of the cluster and under the sharp two-point asymptotics (see \cite[Lemmas~6.5 and~6.6]{BRH26}). While an extrinsic statement could be assembled from those ingredients in that setting, Theorem~\ref{thm:mainLMB} is proved by a direct percolation argument, requires only the up-to-constants two-point bound \eqref{eq:2pt}, and is stated conditionally on the one-arm event, which is the form required by the mechanism described above.

\subsection{The canonical measure of super-Brownian motion}\label{sec:SBM}

Let $\gamma>0$. For a finite measure $m$ on $\R^d$, let $\bP_m$ denote the law of the
$(\tfrac12\Delta,\gamma)$-super-Brownian motion $(X_t)_{t\ge0}$ started from
$m$: that is, the $\cM$-valued continuous strong Markov process characterized by
\begin{equation}\label{eq:laplace}
\E_m\bigl[e^{-\langle\phi,X_t\rangle}\bigr]=e^{-\langle u_t,m\rangle},
\qquad
\partial_t u=\tfrac12\Delta u-\tfrac{\gamma}{2}u^2,\quad u_0=\phi,
\end{equation}
for $\phi\in\Cc$; see \cite{Eth00,LG99,Per02}. The \emph{canonical measure}
$\mathbb{N}_x$ is the $\sigma$-finite measure on continuous $\cM$-valued paths
defined by $\mathbb{N}_x:=\lim_{\eps\downarrow0}\eps^{-1}\bP_{\eps\delta_x}$
\cite[Exercice II.7.2]{Per02};
equivalently, under $\bP_{\eps\delta_x}$ the process is the superposition of a
Poisson collection of excursions with intensity $\eps\,\mathbb{N}_x$
\cite[Theorem~IV.4]{LG99}. Under $\bN$, the canonical measure at the origin, the
extinction time $\sigma:=\inf\{t:X_t=0\}$ is finite a.e., with the classical
survival formula
\begin{equation}\label{eq:survival}
\bN\bigl(\sigma>t\bigr)=\bN\bigl(X_t\neq0\bigr)=\frac{2}{\gamma t},
\qquad t>0,
\end{equation}
and the total occupation measure $\mu=\int_0^\infty X_t\,dt$, introduced in Section~\ref{sec:mainresults}, is a.e.\ a finite
measure on $\R^d$ with compact support; $\Ran:=\supp\mu$ is its closed support.
In particular the total mass is not integrable under $\bN$: $\bN[\mu(\R^d)]=\infty$
(Proposition~\ref{prop:totalmass}). Further properties of $\bN$ are recalled in
Section~\ref{sec:sbmstructure} as needed.

\subsection{The key formulas linking percolation to the SBM universality class}\label{sec:guiding}

The argument rests on an exact correspondence between the moments of the
cluster and those of super-Brownian motion. On the super-Brownian side, the moments of the
occupation measure $\mu$ under $\bN$ are given by the following identity: for
$\phi_1,\dots,\phi_k\in\Ccg$,
\begin{equation}\label{eq:guiding-adler}
\bN\Bigl[\prod_{i=1}^k\mu(\phi_i)\Bigr]
=\gamma^{\,k-1}\sum_{T\in\T_k}\int_{(\R^d)^k}
I_T(z_1,\dots,z_k)\prod_{i=1}^k\phi_i(z_i)\,dz_i ,
\end{equation}
where the sum runs over the binary trees with $k+1$ labelled leaves, one pinned at the origin, and $I_T$ is
the tree integral already encountered in \eqref{eq:guiding-kp}. The
identity \eqref{eq:guiding-adler} is proved in Proposition~\ref{prop:occmoments}: it is
a consequence, obtained by integrating out the times, of the classical
fixed-time moment formulas for superprocesses that go back to Dynkin
\cite{Dyn88} and, in the graphical form used here, to Adler \cite{Adl93}. On
the percolation side, the main theorem of \cite{CCHS26} states that the
rescaled $(m{+}1)$-point functions converge to \emph{precisely these tree
integrals}: this is the convergence \eqref{eq:guiding-kp}, whose left-hand
side, summed against test functions, is exactly the $m$-th moment of the
measures $\nu_n$ of \eqref{eq:empirical}. The two displays together say that
\emph{every moment of $\nu_n$ converges to the corresponding moment of $\bN$},
with $\gamma=\lambda/\mathfrak a$ (the constants $\mathfrak a,\lambda$ of \eqref{eq:guiding-kp}); in particular $\mathbb{N}_{0,\gamma}=\bN$.
Section~\ref{sec:measureconv} is devoted to upgrading this convergence
of moments to convergence of the measures themselves.

The method of moments, however, does not apply directly: the obstructions, and
their resolution by size-biasing, are sketched at the beginning of
Section~\ref{sec:measureconv}, and the passage from convergence of measures to
convergence of supports at the beginning of Section~\ref{sec:onearm}. The
paper is organized as follows: Section~\ref{sec:lmb} proves the lower mass
bound, Section~\ref{sec:measureconv} establishes the measure convergence
(Theorem~\ref{thm:mainA}), and Section~\ref{sec:onearm} derives the geometric
consequences (Theorems~\ref{thm:mainD} and~\ref{thm:mainC}). Each of
Sections~\ref{sec:lmb}--\ref{sec:onearm} opens with a brief sketch of its
proofs.

\subsection{Notation}\label{sec:notation}

We write $B(x,r)$ for the \emph{open} Euclidean ball of radius $r$ centred at
$x$, and $\overline B(x,r)$ for the closed ball; $B_r:=B(0,r)$ and
$\overline B_r:=\overline B(0,r)$. Lattice $\ell_\infty$ balls (boxes) are denoted
$Q_{z,r}:=\{y\in\Z^d:\lVert y-z\rVert_\infty\le r\}$. We write $\{0\leftrightarrow\partial B_r\}:=\{\mathcal C\not\subset\overline B_r\}$, and
$\fl{ny}\in\Z^d$ for the componentwise floor of $ny$.

As already introduced, $\cM$ is the space of finite Borel measures on $\R^d$ with the weak topology
(i.e.\ test functions $C_b(\R^d)$), which is a Polish space; $\cMM$ is the space of finite
Borel measures on $\cM$, again with the weak topology. We write
$\mu(\phi)=\int\phi\,d\mu$. Test functions against which moments are computed
lie in $\Ccg=C_c(\R^d)$, while the size-biasing weight $\varphi_0$ will always be
taken in $\Cc=C_c^+(\R^d)$, the nonnegative compactly supported continuous
functions. The Green
function of Brownian motion is $G(x)=c_d|x|^{-(d-2)}$ with
$c_d=\Gamma(d/2-1)/(2\pi^{d/2})$, and $p_s(x)=(2\pi s)^{-d/2}e^{-|x|^2/2s}$ is
the corresponding heat kernel, so that $G(x)=\int_0^\infty p_s(x)\,ds$ for $x\neq0$.

Generic constants $c,C$ change from line to line, with dependence on
parameters indicated when it matters. We write $A\lesssim B$ if $A\le CB$ for such a constant $C$, $A\gtrsim B$ if $B\lesssim A$, and $A\asymp B$ if both $A\lesssim B$ and $B\lesssim A$; a subscript, as in $A\lesssim_\Lambda B$, emphasizes that the implicit constant may depend on the parameter~$\Lambda$. A handful of symbols are local to a
single section and defined there at first use (for instance the density
parameter $\alpha$ and the regularity scale $K$ of Section~\ref{sec:lmb});
outside such a declaration, the notation is global.

For convenience we collect the most frequently used notation.
\begin{center}
{\small
\begin{tabular}{@{}lll@{}}
\multicolumn{3}{@{}l}{\textbf{Frequently used notation}}\\[3pt]
$\tau_k(x_1,\dots,x_k)$ & $k$-point connection function & \S\ref{sec:guiding}\\
$\mathcal C$ & open cluster of the origin & \\
$B(x,r)$, $B_r$; $\overline B(x,r)$ & open \emph{Euclidean} ball (centre $x$ resp.\ $0$); its closure & \S\ref{sec:notation}\\
$Q_{z,r}$ & lattice $\ell_\infty$ box $\{y:\lVert y-z\rVert_\infty\le r\}$ & \S\ref{sec:notation}\\
$\mu_n$; $\nu_n=n^2\bP(\mu_n\in\cdot)$ & rescaled empirical measure; its intensity & \eqref{eq:empirical}\\
$\varphi_0$; $\nuh_n$, $\Nh$ & size-biasing weight; size-biased measures & \eqref{eq:sizebias}\\
$\bN$; $\mu=\int_0^\infty X_t\,dt$; $\Ran=\supp\mu$ & canonical measure; occupation measure; range & \S\ref{sec:SBM}\\
$\cM$; $\cMM$ & finite measures on $\R^d$; on $\cM$ & \S\ref{sec:notation}\\
$\Cc$; $\Ccg$ & nonnegative resp.\ general $C_c$ test functions & \S\ref{sec:notation}\\
$\T_k$; $I_T$ & binary trees with $k+1$ labelled leaves; tree integrals & \S\ref{sec:treedef}\\
$G(x)$; $p_s(x)$ & Green function; heat kernel & \S\ref{sec:notation}\\
$\mathfrak a,\lambda$; $\gamma=\lambda/\mathfrak a$ & constants of the $k$-point limit; SBM branching parameter & \eqref{eq:guiding-kp}\\
$\theta_K=\theta_1K^{-2}$ & exit mass of the range & Prop.~\ref{prop:range}\\
\end{tabular}}
\end{center}

\section{Extrinsic lower mass bound}
\label{sec:lmb}

This section, which is self-contained given the classical inputs recalled below at \eqref{onearm}, \eqref{taubound} and \eqref{voltail}, proves the lower mass bound of Theorem~\ref{thm:mainLMB}. Providing a key input for the proof of this result is the following estimate, which establishes that if $\mathcal{C}$ hits a box of size $r$, then there will be of order $r^4$ points of $\mathcal{C}$ in close proximity to the box in question (with high probability). This is proved in Section \ref{proofsec}. Furthermore, in Section \ref{proofsec2}, we explain how it leads to Theorem \ref{thm:mainLMB}.

\begin{thm}\label{thm:ulmb} Assume $d>6$ and the two-point bound \eqref{eq:2pt} (in particular, this holds for nearest-neighbour percolation with $d\ge11$). Let $3< \Lambda <\infty$. There exists a constant $c<\infty$ such that, for every $\eta\in(0,1)$,
\[\limsup_{r \rightarrow \infty}\sup_{z\in Q_{0,\Lambda r}\backslash Q_{0,3r}}\bP\left(\left|\mathcal C\cap Q_{z,2r}\right|\le \eta r^4 \mid 0 \leftrightarrow Q_{z,r}\right)\leq c\eta^{1/4}.\]
\end{thm}

The inspiration for the proof of the above result comes from \cite[Theorem 2]{KN11}, which provides a lower bound on the number of points of $\mathcal{C}$ within an annulus. An important point to note about the statement of Theorem \ref{thm:ulmb} is that the upper bound can be made arbitrarily small by adjusting $\eta$. By contrast, in \cite{KN11}, the corresponding upper bound simply gives that the probability is bounded away from 1, uniformly over the various scales. We highlight a simpler adaptation of the part of the argument from \cite{KN11} that we follow can be found in \cite{ASS}, and a bound related to Theorem \ref{thm:ulmb} for the mass at the origin of a large percolation cluster appears in \cite{HHH} (see Proposition 5.1 and Lemma 6.1).

To obtain our stronger conclusion, we iterate the approach of \cite{KN11} over a range of annuli. A schematic for our strategy is provided in Figure \ref{fig:proofidea}. In particular, if $\mathcal{C}$ hits a box $Q_{z,r}$ (which is assumed to be at a distance of order $r$ from the origin), we will show that the mass of $r^4$ points that we are looking for can be found in $Q_{z,2r}\backslash Q_{z,3r/2}$. To demonstrate this, we subdivide this annulus into finer annuli of width $\varepsilon r$. It is straightforward to show that the number of pioneer points on the outer boundary of each of these annuli (that is, the points of $\mathcal{C}$ that are connected to $0$ outside of the box in question) is at least of order $r^2$ (see Lemma \ref{pioneerpoint}; pioneer points are shown on Figure \ref{fig:proofidea} as solid black circles). Moreover, following \cite{KN11}, or more precisely the adaptation of \cite{KN11} provided by \cite{ASS}, we can show that most of these pioneer points are regular, in the sense that the cluster seen outside the current annuli is not too dense around them (see Lemmas \ref{Kregpoint} and \ref{KLgood}). Similarly to the second moment argument of \cite{KN11}, this allows us to suitably decouple the contributions to the volume from separate regular points and, as a consequence, deduce that there is a strictly positive probability that each annulus contains an order $r^4$ points (see Lemma \ref{Iterating}). Finally, through an iterative argument, which gives multiple chances to find a good annuli, we can make this probability arbitrarily close to 1 by adjusting $\varepsilon$ suitably, see Section \ref{proofsec} for details.

\begin{figure}
    \centering
    \includegraphics[width=0.8\textwidth]{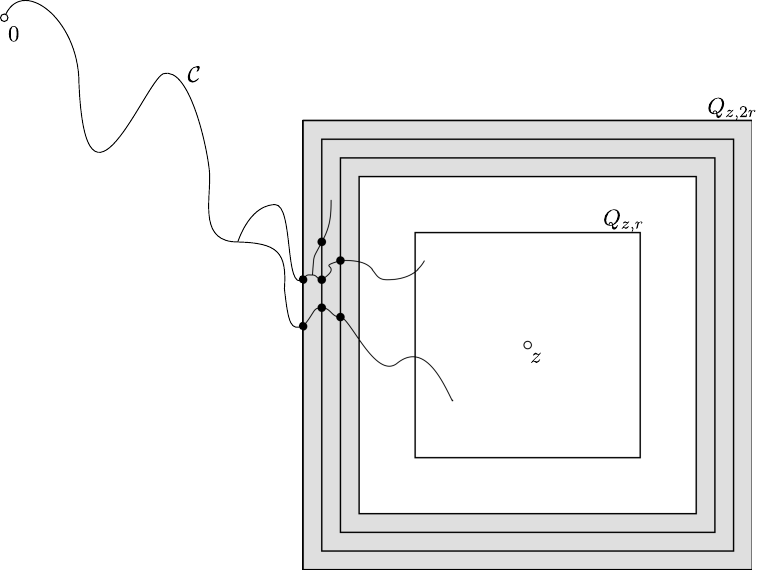}
    \caption{Schematic for the proof of Theorem \ref{thm:ulmb}. The solid black circles represent pioneer points.}
    \label{fig:proofidea}
\end{figure}

For the rest of this section we assume $d>6$ and the two-point bound \eqref{eq:2pt}. In particular, this allows us to use the cluster-tail asymptotics $\bP(|\mathcal C|\ge n)\asymp n^{-1/2}$ (see \eqref{voltail}), which hold under the so-called triangle condition of Barsky and Aizenman \cite{BA91} and Hara and Slade \cite{HS90}, and which are in turn readily implied by \eqref{eq:2pt}.

\subsection{Probability of connecting to a box}

Before proceeding with the program for proving Theorem \ref{thm:ulmb} outlined above, we give a bound on the probability that $0$ is connected to $Q_{z,r}$ for $z$ in an appropriate range. It is for this bound that we require $z$ is at a distance of order $r$
from the origin. Indeed, we believe that Theorem \ref{thm:ulmb} is true if we take the supremum over $z\not\in Q_{0,3r}$, but then we would have to handle more carefully the situation when $z$ is very far from the origin, as the probability of event we are conditioning upon would be much smaller than $r^{-2}$. We recall that we are assuming the one-arm estimate from \cite[Theorem 1]{KN11}:
\begin{equation}\label{onearm}
\bP[0\longleftrightarrow Q_{0,r}^c]\asymp r^{-2},\qquad r\geq 1,
\end{equation}
and the two-point function estimate
\begin{equation}\label{taubound}
\tau(x)\asymp \tnorm{x}^{2-d},\quad \forall x\in\Z^d,
\end{equation}
where we have written $\tnorm{x}=\|x\|\vee1$. Here $\tau(x)=\tau_2(0,x)$ is the two-point function of \eqref{eq:2pt}, and $\|\cdot\|$ denotes the Euclidean norm (any fixed norm gives the same estimates up to constants). In particular, the two-sided estimate \eqref{taubound} is the uniform two-point function bound~\eqref{eq:2pt}, and \eqref{onearm} is Proposition~\ref{prop:KN}; both hold unconditionally for the nearest-neighbour model in $d\ge11$ (see Section~\ref{sec:assumptions}).

\begin{lem}\label{lambda}
Let $\Lambda>3$. There exist positive constants $c_{\Lambda}$ and $c$, where $c_{\Lambda}$ depends only on $\Lambda$ (and $d$), such that
    \begin{align*}
        c_{\Lambda}r^{-2}\le \bP[0 \leftrightarrow Q_{z,r}]\le cr^{-2},\qquad \forall z\in Q_{0,\Lambda r}\setminus Q_{0,3r}.
    \end{align*}
\end{lem}
\begin{proof}
    For the upper bound, observe that if $z \in Q_{0,\Lambda r}\setminus Q_{0,3r}$ and the event $0 \leftrightarrow Q_{z,r}$ occurs, then there must exist an open path from $0$ to $\partial Q_{0,r}$. Hence we obtain that
    \begin{align*}
       \bP[0 \leftrightarrow Q_{z,r}]\le \bP[0 \leftrightarrow Q_{0,r}^c] \lesssim r^{-2},
           \end{align*}
The second inequality is provided by \eqref{onearm}.

    For the lower bound, by the Paley--Zygmund inequality (see \cite[Lemma 4.1]{Kall}, for example), we have
    \begin{align*}
      \bP[0 \leftrightarrow Q_{z,r}]= \bP[|\mathcal C\cap Q_{z,r}|\ge 1]\ge \frac{\E[|\mathcal C\cap Q_{z,r}|]^2}{\E[|\mathcal C\cap Q_{z,r}|^2]}.
    \end{align*}
    Applying \eqref{taubound}, the numerator is bounded below by
    \begin{align*}
       \E[|\mathcal C\cap Q_{z,r}|]=\sum_{x \in Q_{z,r}}\tau(x)\gtrsim \sum_{x \in Q_{z,r}}\tnorm{x}^{2-d}\ge ((\Lambda+1) r)^{2-d}|Q_{z,r}|\gtrsim (\Lambda+1)^{2-d}r^2.
    \end{align*}
For the upper bound, by applying the BK-inequality (see \cite[Theorem 2.12]{Grim}, for example), we obtain
    \begin{align}\label{2ndmmt}
       \E[|\mathcal C\cap Q_{z,r}|^2]&=\sum_{y,w \in Q_{z,r}}\bP[0 \leftrightarrow y, 0\leftrightarrow w]\nonumber\\
       &=\sum_{y \in Q_{z,r}}\tau(y)+\sum_{\substack{y,w \in Q_{z,r}\\ y\neq w}}\bP[0 \leftrightarrow y, 0\leftrightarrow w]\nonumber\\
       &\lesssim r^2+\sum_{\substack{y,w \in Q_{z,r}\\ y\neq w}}\sum_{u \in \Z^d}\bP[\{0 \leftrightarrow u\} \circ\{ u\leftrightarrow w\} \circ\{ u\leftrightarrow y\}]\nonumber\\
       &\le  r^2+\sum_{\substack{y,w \in Q_{z,r}\\ y\neq w}}\sum_{u \in \Z^d}\tau(u)\tau(w-u)\tau(y-u).
    \end{align}
    Here, in the first inequality, $A \circ B$ denotes the bond-disjoint occurrence of the events $A$ and $B$, while $u$ is chosen to be the last intersection point of a path from $0$ to $w$
 and a path from $0$ to $y$, where ``last'' is taken along the path from $0$ to $w$.

 We observe that for $0\le h<d$, uniformly over $u,x \in \Z^d$
 \begin{align}\label{bxbd}
     \sum_{y \in Q_{x,L}}\tnorm{u-y}^{h-d}=\sum_{\substack{y \in Q_{x,L}\\ \tnorm{u-y}\le L}}\tnorm{u-y}^{h-d}+\sum_{\substack{y \in Q_{x,L}\\ \tnorm{u-y}> L}}\tnorm{u-y}^{h-d}\lesssim L^h.
 \end{align}
 Furthermore, it is proved in \cite[Proposition~1.7]{HHS03} that for $a\ge b>0$, there exists a constant $C$ depending on $a,b,d$ such that
 \begin{align}\label{HHSbound}
     \sum_{y \in \Z^d}\tnorm{y}^{-a}\tnorm{x-y}^{-b}\le
     \begin{cases}
         C\tnorm{x}^{-b},& \mbox{ if }a>d\\
         C\tnorm{x}^{d-(a+b)},& \mbox{ if } a<d~\text{and}~a+b>d.
     \end{cases}
 \end{align}
By applying \eqref{bxbd}--\eqref{HHSbound}, the second term in \eqref{2ndmmt} is bounded by
    \begin{align*}
       \sum_{y,w \in Q_{z,r}}\sum_{u \in \Z^d}\tnorm{u}^{2-d}\tnorm{w-u}^{2-d}\tnorm{y-u}^{2-d}
       &=\sum_{w \in Q_{z,r}}\sum_{u \in \Z^d}\tnorm{u}^{2-d}\tnorm{w-u}^{2-d}\sum_{y\in Q_{z,r}}\tnorm{y-u}^{2-d}\\
       &\lesssim r^2 \sum_{w \in Q_{z,r}}\tnorm{w}^{4-d}\\
       &\lesssim r^6.
    \end{align*}
    Combining the above estimates yields the desired lower bound for the probability $\bP[0 \leftrightarrow Q_{z,r}]$.
\end{proof}

\subsection{Pioneers and regularity}

To describe the strategy presented in Figure \ref{fig:proofidea} above more precisely, we introduce a decomposition of the annulus $Q_{z,2r}\setminus Q_{z,r}$ into smaller annuli $A_{z,i}$ by setting, for a given $\varepsilon>0$ and $i=1,\dots,\varepsilon^{-1}$,
\begin{align*}
A_{z,i}=Q_{z, (i\varepsilon+1)r}\setminus Q_{z, ((i-1)\varepsilon+1)r}.
\end{align*}
We treat $\varepsilon$ as a continuous parameter and, with a slight abuse of notation, write $\varepsilon^{-1}$ for $\lfloor\varepsilon^{-1}\rfloor$ throughout this section. (Alternatively, it would be enough for our arguments to consider the case when $\varepsilon^{-1}\in\mathbb{N}$.) For simplicity, we will subsequently abbreviate $Q_{z, (i\varepsilon+1)r}$ to $\mathcal Q_{z,i}$, so that $A_{z,i}={\mathcal Q_{z,i}\setminus \mathcal Q_{z,i-1}}$.

As indicated above, a key part of our argument is a regularity analysis. To present this, we next set out some key definitions. (Our definitions are essentially those of \cite{ASS}, though with a modification to take into account the fact that we are approaching a box from outside, whereas they were studying the situation inside a box.) Firstly, we define the number of pioneer points on $\partial \mathcal Q_{z,i}$, where $\partial \mathcal Q_{z,i}$ denotes the subset of $\mathcal Q_{z,i}$ containing those vertices with at least one neighbor in $\mathcal Q_{z,i}^c$, by setting
\begin{align*}
    X_{r,i}=\left|\left\{x \in \partial \mathcal Q_{z,i}: 0 \overset{\mathcal Q_{z, i}^c}{\longleftrightarrow} x \right\}\right|.
\end{align*}
Here, for subsets $A,B, C \subseteq \Z^d$, the event $A \overset{C}{\longleftrightarrow} B$ means that there exists an open path connecting a vertex of $A$ to a vertex of $B$ whose elements lie entirely within $C$, apart from possibly the terminal vertices. For $x \in \partial \mathcal Q_{z, i}$ and $s>0$, we define an event concerning the global density of the cluster containing $x$ by setting
\begin{align*}
\mathcal T_{s}(i;x)=\left\{\left|\mathcal C(x; \mathcal Q_{z,i}^c)\cap Q_{x,s}\right|\le s^4(\log s)^7\right\}\cap \left\{\left|\mathcal C(x; \mathcal Q_{z,i}^c)\cap Q_{x,s}\cap \partial \mathcal Q_{z,i}\right|\le s^2(\log s)^7\right\},
\end{align*}
where, for a subset $A \subseteq \Z^d$, and a vertex $x \in A$, we define the restricted cluster in $A$ by setting
\begin{align*}
 \mathcal C(x; A)=\{y \in A: x \overset{A}{\longleftrightarrow} y \}.
\end{align*}
With these definitions in place, we are now able to state what it means for a point on the boundary of $\mathcal Q_{z,i}$ to be regular.

\begin{defn}\label{def:Kreg}
For $x \in \partial \mathcal Q_{z, i}$, we say that $x$ is $s$-good if $\mathcal T_{s}(i;x)$ holds. Otherwise, we say that $x$ is $s$-bad. For $x \in \partial \mathcal Q_{z, i}$, we call $x$ a $K$-regular point if the following two conditions hold:
\begin{enumerate}
\item $x$ is $s$-good for $s\ge K$;
\item $x$ is connected to $0$ in $\mathcal Q_{z, i}^c$.
\end{enumerate}
Similarly, if $x$ is connected to $0$ in $\mathcal Q_{z, i}^c$, but it is not $K$-regular, then we say it is $K$-irregular. We write $X_{r,i}^{K-\mathrm{reg}}$ (resp.\ $X_{r,i}^{K-\mathrm{irr}}$) for the number of $K$-regular (resp.\ irregular) points on $\partial \mathcal Q_{z, i}$.
\end{defn}

Following \cite{ASS}, we also introduce a special kind of regular point, called a `line good' point.

\begin{defn}\label{def:KLgood}
Consider the event $\{\mathcal C(0; \mathcal Q_{z, i}^c)=B\}$ for a given $B$.
\begin{itemize}
\item Denote by $\mathrm{Reg}_i(B)$ a maximal subset of $K$-regular points on $\partial \mathcal Q_{z, i}$ whose pairwise $\ell_\infty$-distances are at least $3K$. (If there is more than one possibility for this set, then we choose one arbitrarily according to some fixed procedure.)
\item For each $x \in  \mathrm{Reg}_i(B)$, we define $L_x$ to be a line segment of length $K$ extending from $x$ inside $\mathcal Q_{z, i}$, perpendicularly to the face upon which $x$ is sited. If $x$ lies within a distance $K$ from the intersection of at least two faces of $\mathcal Q_{z, i}$, then $L_x$ is defined to be the union of (at most) $d$ line segments inside $\mathcal Q_{z, i}$ of total length at most $dK$ that connect $x$ to a point that is at a distance $K$ from $\partial \mathcal Q_{z, i}$. Note that, since the vertices are separated by a distance at least $3K$, two distinct $L_x$ are separated by a distance at least $K$.
\item For each $x \in  \mathrm{Reg}_i(B)$, we consider the maximal line segment from $x$ that is contained in $L_x$ and formed by open edges. We write $L(B)$ for the union of all these line segments. (Note this is a random set, depending on $B$ and the configuration within $\mathcal Q_{z, i}$.)
\end{itemize}
A point $x\in \mathrm{Reg}_i(B)$ is said to be $K$-line good if all edges of $L_x$ are open. For such points, write $x'$ for the point of $L_x$ at maximal distance from $x$, and let $\mathrm{Reg}'_i(B)$ be the collection of such points. We write $X_{r,i}^{K-\mathrm{line-good}}$ for the number of $K$-line good points on $\partial Q_{z, (i\varepsilon+1)r}$. (See Figure \ref{fig:KLgooddef}.) Moreover, for the set $\mathcal Q_{z, i}^c$, we define the corresponding extended cluster by setting
\begin{align}
\mathcal C_e(0; \mathcal Q_{z, i}^c)=\mathcal C(0;\mathcal Q_{z, i}^c)\cup L(\mathcal C(0;\mathcal Q_{z, i}^c)). \label{extclust}
\end{align}
\end{defn}

\begin{figure}
    \centering
    \includegraphics[width=0.54\textwidth]{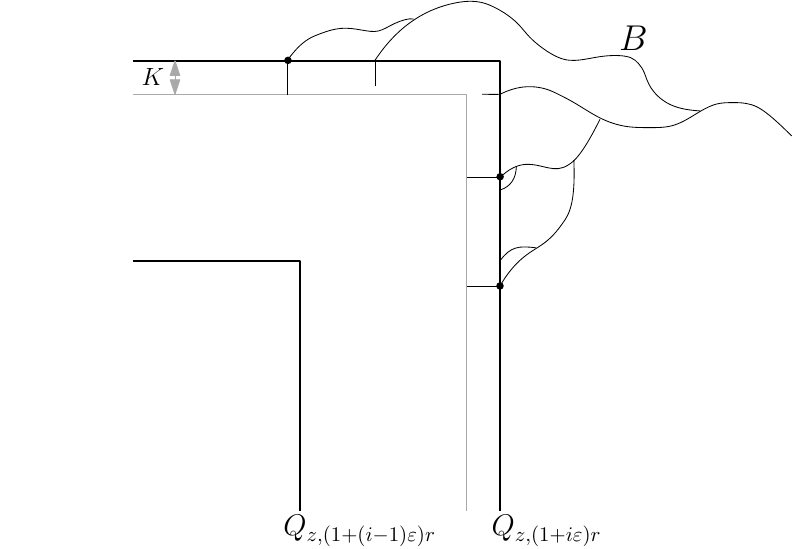}\qquad\includegraphics[width=0.36\textwidth]{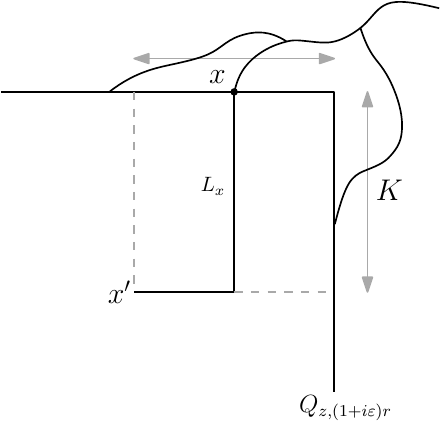}
    \caption{Left: Schematic for the definition of $K$-line good points. The line segments represent $L(B)$, and the solid black circles represent the $K$-line good points of $B$. Right: Schematic for the definition of $L_x$, for $x$ lying within distance $K$ of an edge of $\partial \mathcal Q_{z,i}$. The segment first extends orthogonally into $\mathcal Q_{z,i}$ (perpendicular to the face on which $x$ lies), then bends and extends orthogonally again -- this time parallel to the first face -- until it reaches distance $K$ from the second face. The resulting endpoint $x'$ is then at distance at least $K$ from all of $\partial \mathcal Q_{z,i}$. }
    \label{fig:KLgooddef}
\end{figure}

In the next three lemmas, we give estimates that allow us to understand the number of pioneer points, $K$-regular points and $K$-line good points occurring on the boundary of each annuli in our decomposition. Going forward, the key estimate of this section will be the bound given in Lemma \ref{KLgood}.

\begin{lem}\label{pioneerpoint}
Let $\Lambda>3$. There exists a $\Lambda$-dependent constant $C_{\Lambda}$ such that, for all $z \in Q_{0,\Lambda r}\setminus Q_{0,3r}$ and every $\varepsilon>0$, $i \in \{1, \cdots, \varepsilon^{-1}\}$ and $\alpha>0$,
\begin{align*}
       \bP[X_{r,i}\le \alpha r^2| 0 \leftrightarrow Q_{z,r}]\le C_{\Lambda} \frac{\alpha}{(i\varepsilon)^2}.
    \end{align*}
\end{lem}
\begin{proof}
We say that a cluster $A$ is $i$-admissible if $\bP[\mathcal C(0; \mathcal Q_{z, i}^c)=A]>0$ and $1\leq |A\cap \partial \mathcal{Q}_{z,i}|\leq  \alpha r^2$. We then have that
\begin{align*}
       \bP[X_{r,i}\le \alpha r^2,\: 0 \leftrightarrow Q_{z,r}]&=\sum_{\substack{A\:\text{admissible},\\ A \neq \varnothing }}\bP[\mathcal C(0; \mathcal Q_{z, i}^c)=A,\:0 \leftrightarrow Q_{z,r}]\\
       &=\sum_{\substack{A\:\text{admissible},\\ A \neq \varnothing }}\bP[\mathcal C(0; \mathcal Q_{z, i}^c)=A]\bP[0 \leftrightarrow Q_{z,r}|\mathcal C(0; \mathcal Q_{z, i}^c)=A].
    \end{align*}
Whenever the event $\{0 \leftrightarrow Q_{z,r}\}$ occurs, there exists a vertex  $y \in A \cap \partial \mathcal Q_{z,i}$ such that $y \overset{\mathrm{off }A}{\longleftrightarrow}Q_{z,r}$ occurs. Here, for subsets of vertices $A,B,C \subseteq \Z^d$, the event $A \overset{\mathrm{off}~C}{\longleftrightarrow}B$ means that there exists an open path connecting a vertex of $A$ to a vertex of $B$ that avoids all vertices of $C$, apart from possibly the terminal vertices. Thus, a union bound gives that
    \begin{align*}
      \bP[0 \leftrightarrow Q_{z,r}|\mathcal C(0; \mathcal Q_{z, i}^c)=A]
      &\le \sum_{y \in A \cap \partial \mathcal Q_{z,i}}\bP[y \overset{\mathrm{off}~A}{\longleftrightarrow}Q_{z,r}|\mathcal C(0; \mathcal Q_{z, i}^c)=A]  \\
      &\le\sum_{y \in A \cap \partial \mathcal Q_{z,i}}\bP[y \overset{\mathrm{off}~A}{\longleftrightarrow}Q_{z,r}] \\
      &\le | A \cap \partial \mathcal Q_{z,i}|\sup_{y \in A \cap \partial \mathcal Q_{z,i}}
      \bP[y \longleftrightarrow Q_{y,i\varepsilon r}^c].
    \end{align*}
Since $| A \cap \partial \mathcal Q_{z,i}|\le \alpha r^2$, we obtain that
    \begin{align*}
       \bP[X_{r,i}\le \alpha r^2| 0 \leftrightarrow Q_{z,r}]&\le \alpha r^2 \frac{\bP[0 \longleftrightarrow Q_{0,i\varepsilon r}^c] }{\bP[0\leftrightarrow Q_{z,r}]}\bP[1\le X_{r,i}\le \alpha r^2] \\
       &\le \alpha r^2 \frac{\bP[0 \longleftrightarrow Q_{0,i\varepsilon r}^c] }{\bP[0\leftrightarrow Q_{z,r}]}\bP[0\longleftrightarrow Q_{0,r}^c]\\
       &\lesssim_{\Lambda} \frac{\alpha}{(i\varepsilon)^2},
    \end{align*}
    where the final inequality follows from \eqref{onearm} and Lemma \ref{lambda}.
\end{proof}

In proving a good fraction of pioneer points are regular, as we do in the next lemma, we consider local versions of the quantities introduced in Definition \ref{def:Kreg}. In particular, we define a local density condition by setting, for $x \in \partial \mathcal Q_{z, i}$ and $s>0$,
\begin{align*}
\mathcal T_{s}^{\mathrm{loc}}(i;x)=&\bigcap_{y \in Q_{x,s}}\big\{ |\mathcal C(y; Q_{x,s^d}\cap \mathcal Q_{z,i}^c)\cap Q_{x,s}|\le s^4(\log s)^4\big\} \\
&\bigcap \bigcap_{y \in Q_{x,s} \cap \partial \mathcal Q_{z,i}}\big\{ |\mathcal C(y; Q_{x,s^d}\cap \mathcal Q_{z,i}^c)\cap Q_{x,s} \cap \partial \mathcal Q_{z,i}|\le s^2(\log s)^4\big\}\\
&\bigcap \big\{ \exists \:\text{at most}~(\log s)^3~\text{disjoint path from}~Q_{x,s}~\text{to}~Q_{x,s^d}~\text{in}~\mathcal Q_{z,i}^c \big\}.
\end{align*}
Again, this is based on the corresponding definition from \cite{ASS}. We note that for any $i \in \{1,\cdots, \varepsilon^{-1}\}$ and $x \in \partial \mathcal Q_{z, i}$, the event $\mathcal T_{s}^{\mathrm{loc}}(i;x)$ implies $\mathcal T_{s}(i;x)$. (Cf.\ \cite[Claim 5.4]{ASS} or \cite[Claim 4.1]{KN11}.) We also define $s$-locally good and $s$-locally bad points, similarly to $s$-good and $s$-bad points (resp.), replacing $\mathcal T_{s}(i;x)$ by $\mathcal T_{s}^{\mathrm{loc}}(i;x)$ in Definition \ref{def:Kreg}.

\begin{lem}\label{Kregpoint}
Let $\Lambda>3$. There exist positive constants $c$ and $C_{\Lambda}$ such that, for all sufficiently large $K$, if $z \in Q_{0,\Lambda r}\setminus Q_{0,3r}$, $\varepsilon>0$,  $i \in \{1, \cdots, \varepsilon^{-1}\}$, $r\geq1$, $\alpha>0$ and $\delta'>0$ satisfy $\alpha>2\delta'$, then
    \begin{align*}
        \bP[X_{r,i}\ge \alpha r^2,X_{r,i}^{K-\mathrm{reg}}\le \delta' r^2 | 0 \leftrightarrow Q_{z,r}]\le C_{\Lambda}r^{d+1}e^{-c(\log \alpha r^2)^4}.
    \end{align*}
\end{lem}
\begin{proof}Since $\mathcal T_{s}^{\mathrm{loc}}(i;x)\subseteq \mathcal T_{s}(i;x)$, it follows from Definition \ref{def:Kreg} that
\begin{align*}
    X_{r,i}^{K-\mathrm{irr}}\le \sum_{s\ge K}X_{r,i}^{s-\mathrm{loc-bad}},
\end{align*}
where $X_{r,i}^{s-\mathrm{loc-bad}}$ denotes the number of $s$-locally bad points. Consequently,
    \begin{align*}
        \bP[X_{r,i}\ge \alpha r^2,X_{r,i}^{K-\mathrm{reg}}\le \delta' r^2 | 0 \leftrightarrow Q_{z,r}]&\le \frac{1}{ \bP[0 \leftrightarrow Q_{z,r}]} \bP[X_{r,i}\ge \alpha r^2,X_{r,i}^{K-\mathrm{irr}}\ge (1-\delta'/\alpha) X_{r,i}]\\
        &\lesssim_{\Lambda} r^2\sum_{s\ge K}\bP[X_{r,i}\ge \alpha r^2,X_{r,i}^{s-\mathrm{loc-bad}}\ge \frac{X_{r,i}}{s^2}].
    \end{align*}
In the second inequality, we used Lemma~\ref{lambda}, and the fact that the event $\sum_{s\ge K}X_{r,i}^{s-\mathrm{loc-bad}}\ge (1-\delta'/\alpha) X_{r,i}\geq X_{r,i}/2$ implies that there exists $s \ge K$ such that
  \begin{align*}
     X_{r,i}^{s-\mathrm{loc-bad}}\ge \frac{X_{r,i}}{s^2}.
  \end{align*}
(We may assume $K$ is large enough so that $\sum_{s\geq K}s^{-2}\leq 1/2$.) Next, let $K_0=(\alpha r^2)^{1/8d^2}$. For the sum over $s\ge K_0$, we obtain
  \begin{align*}
    \sum_{s\ge K_0}\bP[X_{r,i}\ge \alpha r^2,X_{r,i}^{s-\mathrm{loc-bad}}\ge \frac{X_{r,i}}{s^2}]&\le \sum_{s\ge K_0}\bP[X_{r,i}^{s-\mathrm{loc-bad}}\ge1]\\
    &\le  \sum_{s\ge K_0}\sum_{ x \in \partial \mathcal Q_{z,i}}\bP[\mathcal T_{s}^{\mathrm{loc}}(i;x)^c]\\
    &\lesssim r^{d-1}\sum_{s\ge K_0}e^{-c(\log s)^4}\\
    &\lesssim r^{d-1}e^{-c(\log \alpha r^2)^4}
  \end{align*}
  where for the third inequality, we used the fact that for every $x \in \partial \mathcal Q_{z,i}$ and $s\geq 1$, there exists a constant $c$ such that
  \begin{align}\label{slocgoodbound}
    \bP[\mathcal T_{s}^{\mathrm{loc}}(i;x)]\ge 1-e^{-c(\log s)^4}.
  \end{align}
This estimate is an analogue of \cite[Claim~5.5]{ASS} for the exterior-of-a-box geometry used here (in contrast to the interior-of-a-box geometry of \cite{ASS}); the same argument remains valid when our definition of $\mathcal T_{s}^{\mathrm{loc}}(i;x)$ is used. We highlight that, in \cite{ASS}, the range of $s$ for which \eqref{slocgoodbound} applies is naturally restricted to $s\in (0,r)$, but this is unnecessary in our setting, since the exterior of the box is unbounded.

For the sum over all $K\le s\le K_0$, we use essentially the same exploration procedure as in \cite{ASS, KN11}, with minor modifications to accommodate our definition of regularity. In the remaining proof, we aim to show
\begin{align}\label{smalls}
     \sum_{K\le s\le K_0}\bP[X_{r,i}\ge \alpha r^2,X_{r,i}^{s-\mathrm{loc-bad}}\ge \frac{X_{r,i}}{s^2}]\lesssim e^{-\kappa\alpha^{1/2}r}
  \end{align}
  for some constant $\kappa>0$.

Let  $\mathcal U:=\{u \in \Z^d: u_i\in \{0,s^d\}, \forall i=1,\dots, d\}$ denote the collection of shifts. For $w\in \mathcal U$, define $\mathcal B_{z,i}(w):=\{\mathcal Q_{z,i}^c\cap B(v,2s^d): v\in w+4s^d\cdot \Z^d \}$. This yields a partition of $\mathcal Q_{z,i}^c$ into boxes. We now explain the exploration process with a slight modification of the one originally introduced in \cite{KN11}. First, we fix an arbitrary ordering of the elements of $\mathcal B_{z,i}(w)$. The exploration process consists of two sequences of subsets of $\mathcal B_{z,i}(w)$: the explored boxes $E_j$, and the active boxes $A_j$. To start the procedure, we define $E_1$ and $A_1$ as follows:
\[E_1=\{q \in \mathcal B_{z,i}(w): q \cap \partial \mathcal Q_{z,i}=\varnothing\},\]
\[A_1=\{q \in \mathcal B_{z,i}(w): \exists x \in \partial q,\: 0 \overset{\cup E_1}{\longleftrightarrow} x\},\]
where $\cup E_j:=\cup_{q \in E_j}q$. At step $j$, we choose a box $q_j \in A_{j-1}$ according to the prescribed ordering of the boxes in  $\mathcal B_{z,i}(w)$. We regard this as exploring $q_j$. The remaining boxes in $A_{j-1}$ not explored stay active. We then update the set of explored boxes by setting
\begin{align*}
    E_j=E_{j-1}\cup q_j.
\end{align*}
To obtain $A_j$, we add to $A_{j-1}$ all previously unexplored boxes in $\mathcal B_{z,i}(w)$ that are reached from $0$ by paths going only through boxes in $E_j$. Namely,
\begin{align*}
    A_j=\left(A_{j-1}\cup \{ q \in \mathcal B_{z,i}(w): \exists x \in \partial q, 0 \overset{\cup E_j}{\longleftrightarrow} x \}\right)\setminus E_j.
\end{align*}
See Figure \ref{fig:Exprolation} for an example. Eventually, there are no active boxes left, so no box $q_{i+1}$ can be chosen. We then say that the exploration process has finished, and denote the total number of exploration steps by $T$.

\begin{figure}
    \centering
    \includegraphics[width=0.45\textwidth]{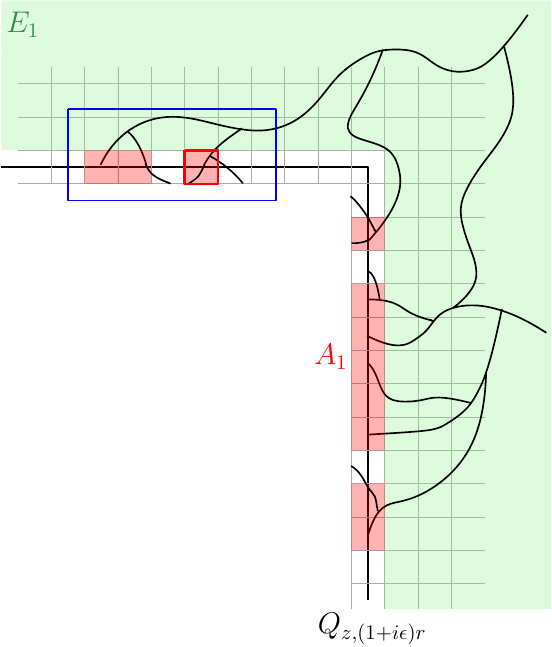}
    \hspace{40pt}
    \raisebox{50pt}{\includegraphics[width=0.3\textwidth]{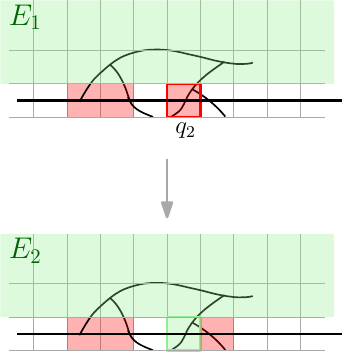}}
    \caption{Left: Schematic for the definitions of $E_1$ and $A_1$. The gray small boxes are those in $\mathcal B_{z,i}(w)$. Right: Enlarged view of the blue boxed region in the left panel. If the red boxed active box is chosen as $q_2$ at step $2$, it is added to $E_1$ to form $E_2=E_1\cup q_2$, and the box to its right becomes active, since it can be reached from $0$ by paths going only through $E_2$. }
    \label{fig:Exprolation}
\end{figure}

We say that a box $q \in B_{z,i}(w)$ has a crossing if it contains a path connecting the cluster of $0$ in the explored region to a point $x \in \partial \mathcal Q_{z,i}$ such that $B(x,s^d)\cap \mathcal Q_{z,i}^c$ is entirely contained in $q$. Furthermore, we say that a box $q \in \mathcal B_{z,i}(w)$ is s-locally bad if it contains an s-locally bad vertex $x \in \partial \mathcal Q_{z,i}$ such that $B(x, s^d)\cap \mathcal Q_{z,i}^c\subseteq q$.

Now, let $N(w)$ denote the total number of explored boxes of the partition $\mathcal B_{z,i}(w)$ that have a crossing, and let $N_s(w)$ denote the total number of explored boxes that have a crossing and are s-locally bad. Note that, since the boxes of $\mathcal B_{z,i}(w)$ are disjoint and each crossing box contains at least one pioneer point, $N(w)\le X_{r,i}$. Moreover, since for any $x \in \partial \mathcal Q_{z,i}$ there exists $w \in \mathcal U$ and a box $q \in B_{z,i}(w)$ such that $B(x,s^d) \cap \mathcal Q_{z,i}^c  \subseteq q$, we obtain
\begin{align*}
    \bP\left[X_{r,i}\ge \alpha r^2,X_{r,i}^{s-\mathrm{loc-bad}}\ge \frac{X_{r,i}}{s^2}\right]&\le  \bP\left[X_{r,i}\ge \alpha r^2,\:4^ds^{d^2}\sum_{w \in \mathcal U}N_{s}(w)\ge \frac{X_{r,i}}{s^2}\right]\\
    &\le \sum_{w \in \mathcal U}\bP\left[X_{r,i}\ge \alpha r^2,N_{s}(w)\ge \frac{X_{r,i}}{s^{2+d^2}2^{3d}}\right].
  \end{align*}
We will show shortly that conditioning on $N(w)$, the quantity $N_s(w)$ is stochastically dominated by a binomial random variable with parameters $N(w)$ and $e^{-c(\log s)^4}$. Assuming this, we obtain that for sufficiently large $K$ (and hence for all $s\ge K$), there exists a constant $\kappa$ such that
  \begin{align*}
&\bP\left[X_{r,i}\ge \alpha r^2,N_{s}(w)\ge \frac{X_{r,i}}{s^{2+d^2}2^{3d}}\right]\\
&=\E\left[\ind{X_{r,i}\ge \alpha r^2}\bP\left[N_{s}(w)\ge \frac{X_{r,i}}
{s^{2+d^2}2^{3d}}\mid N(w),X_{r,i}\right]\right]\\
&=\E\left[\ind{X_{r,i}\ge \alpha r^2}\bP\left[N_{s}(w)-N(w)e^{-c(\log s)^4}\ge \frac{X_{r,i}}
{s^{2+d^2}2^{3d}}-N(w)e^{-c(\log s)^4}\mid N(w),X_{r,i}\right]\right]\\
&\le \E\left[\ind{X_{r,i}\ge \alpha r^2}\bP\left[N_{s}(w)-N(w)e^{-c(\log s)^4}\ge \frac{X_{r,i}}
{2s^{2+d^2}2^{3d}}\mid N(w),X_{r,i}\right]\right]\\
&\le \E\left[\ind{X_{r,i}\ge \alpha r^2}\bP\left[N_{s}(w)-N(w)e^{-c(\log s)^4}\ge \frac{\sqrt{\alpha r^2 N(w)}}{2s^{2+d^2}2^{3d}}\mid N(w),X_{r,i}\right]\right]\\
&\leq \bP\left[N_{s}(w)-N(w)e^{-c(\log s)^4}\ge \frac{\sqrt{\alpha r^2 N(w)}}{2s^{2+d^2}2^{3d}}\right]\\
&\le \exp\big(-\kappa \alpha s^{-2d^2-4} r^2\big)
  \end{align*}
  where: in the first inequality, we take $s$ to be sufficiently large; in the second inequality, we use that $X_{r,i}\geq \max\{\alpha r^2,N(w)\}$ on the event in the indicator; remove the indicator to obtain the third inequality; and, for the fourth, apply Azuma's inequality; see, e.g., \cite[Theorem~7.2.1]{AS08}. Since we have $|\mathcal U|=2^d$, the desired estimate \eqref{smalls} follows.

  Finally, we show that we can upper bound $N_s(w)$ by a binomial random variable with parameters $N(w)$ and $e^{-c(\log s)^4}$. Let $q_1,q_2,\cdots$ denote the sequence of explored boxes, and let $\tilde{q}_1,\tilde{q}_2, \cdots$ denote the subsequence consisting of those boxes that contain a crossing. We define a sequence of stopping times $\{\sigma_j\}$ by $\sigma_0=0$ and
  \begin{align*}
      \sigma_{j+1}=\inf\{k\ge \sigma_j+1: q_k~\text{has a crossing to}~\partial\mathcal Q_{z,i}\}.
  \end{align*}
  Following the argument of \cite[Proposition~5.7]{ASS}, we upper bound $\bP[\tilde{q}_{\sigma_j}~\text{is $s$-locally bad}|\mathcal F_{\sigma_{j-1}}]$, where $\mathcal F_j$ denotes the percolation configuration of the explored region after $j$-th box has been added. For notational convenience, we write $\tilde{q}_j=\tilde{q}_{\sigma_{j}}$ and $\tilde{\mathcal F_j}=\mathcal F_{\sigma_{j-1}}$ for every $j\ge0$. We obtain
  \begin{align*}
     \bP[\tilde{q}_j~\text{is $s$-locally bad}|\tilde{\mathcal F_j}]&=\sum_{l,k}\bP[q_{l+k}~\text{is $s$-locally bad}, \sigma_{j-1}=l,\sigma_j-\sigma_{j-1}=k|\tilde{\mathcal F_j}]\\
     &\le \sum_{l,k}\bP[q_{l+k}~\text{is $s$-locally bad}, \sigma_{j-1}=l,\sigma_j-\sigma_{j-1}>k-1|\tilde{\mathcal F_j}]\\
     &\le \sum_{l,k}\bP[q_{l+k}~\text{is $s$-locally bad}] \bP[\sigma_{j-1}=l,\sigma_j-\sigma_{j-1}>k-1|\tilde{\mathcal F_j}].
  \end{align*}
  Here, the last inequality follows from the fact that the event that $q_{l+k}$ is $s$-locally bad depends only on the configuration inside of $q_{l+k}$. For the first probability, we use \eqref{slocgoodbound} to obtain
  \begin{align*}
     \bP[q_{l+k}~\text{is $s$-locally bad}] \lesssim s^{d^2}\max_{x \in \partial \mathcal Q_{z,i}}(1-\bP[\mathcal T_{s}^{\mathrm{loc}}(i;x)])\lesssim e^{-c(\log s)^4}.
  \end{align*}
Let $E_{l+t}^{\mathrm{cross}}$ denote the event that $q_{l+t}$ has no crossing. For the second probability, we obtain
   \begin{align*}
    \bP[\sigma_{j-1}=l,\sigma_j-\sigma_{j-1}>k-1|\tilde{\mathcal F_j}]&= \bP[\sigma_{j-1}=l,\cap_{t=1}^{k-1}\{q_{l+t}~\text{has no crossing}\}|\tilde{\mathcal F_j}] \\
    &=\bP[\sigma_{j-1}=l|\tilde{\mathcal F_j}]\bP[E_{l+1}^{\mathrm{cross}}|\tilde{\mathcal F_j}]\prod_{t=2}^{k-1}\bP[E_{l+t}^{\mathrm{cross}}|\cap^{t-1}_{m=1}E_{l+m}^{\mathrm{cross}},~\tilde{\mathcal F_j}]\\
    &\le \bP[\sigma_{j-1}=l|\tilde{\mathcal F_j}](1-ce^{-C'(\log s)^2})^{k-1},
  \end{align*}
  where the last inequality follows from \cite[Lemma~5.8]{ASS}, which implies that for any $x \in \mathcal \partial Q_{z,i}$ satisfying $B(x,s^d)\cap \mathcal Q_{z,i}^c\subseteq q$ and any $y \in \partial q \subseteq  \mathcal Q_{z,i}^c$ such that $y$ is connected to $0$ in the explored region, there exist constants $C$ and $c$ such that
  \begin{align*}
 \bP[(E_{l+t}^{\mathrm{cross}})^c|\cap^{t-1}_{m=1}E_{l+m}^{\mathrm{cross}},~\tilde{\mathcal F_j}]\ge \bP[x \overset{Q_{y,\lVert x-y\rVert }}{\longleftrightarrow} y] \ge ce^{-C(\log \lVert x-y\rVert)^2 }\ge ce^{-C'(\log s)^2}.
  \end{align*}
Therefore, combining the above estimates, we obtain
\begin{align}
 \bP[\tilde{q_j}~\text{is $s$-locally bad}|\tilde{\mathcal F_j}]&\lesssim e^{-c(\log s)^4}\sum_{l,k}\bP[\sigma_{j-1}=l|\tilde{\mathcal F_j}](1-ce^{-C'(\log s)^2})^{k-1}\nonumber\\
 &\le e^{-c(\log s)^4}\sum_{k}(1-ce^{-C'(\log s)^2})^{k-1}\nonumber\\
 &\lesssim e^{-c(\log s)^4}. \nonumber
\end{align}
This implies that, conditional on $N(w)$, the random variable $N_s(w)$ is stochastically dominated by a binomial random variable with parameter $N(w)$ and $Ce^{-c(\log s)^4}$, for some positive constant $C$ and $c$.
\end{proof}

We complete the section with our key estimate on the number of line good points.

\begin{lem}\label{KLgood}
Let $\Lambda>3$. There exist positive constants $c$ and $C_\Lambda$ such that, for all sufficiently large $K$, if $z \in Q_{0,\Lambda r}\setminus Q_{0,3r}$, $\varepsilon>0$, $i \in \{1, \cdots, \varepsilon^{-1}\}$, $\delta>0$ and $\alpha>0$ with $\alpha>2C(K)\delta$, where $C(K)=K^d/p_c^{2dK}$,
\begin{align*}
\mathbb {P}\left[X_{r,i}^{K-\mathrm{line-good}}\le \delta r^2| 0\leftrightarrow Q_{z,r}\right]\le C_{\Lambda}\left[r^{d+1}e^{-c(\log \alpha r^2)^4}+\frac{\alpha}{(i\varepsilon)^2}+ r^2\exp(-\delta r^2/2)\right].
\end{align*}
\end{lem}
\begin{proof}
We first decompose the probability as
\begin{align*}
&\bP[X_{r,i}^{K-\mathrm{line-good}}\le \delta r^2| 0 \leftrightarrow Q_{z,r}]\\
&\le \bP[X_{r,i}^{K-\mathrm{reg}}\le \delta' r^2| 0 \leftrightarrow Q_{z,r}]+\bP[X_{r,i}^{K-\mathrm{reg}}\ge \delta'r^2, X_{r,i}^{K-\mathrm{line-good}}\le \delta r^2| 0 \leftrightarrow Q_{z,r}]\\
&\lesssim_{\Lambda}\bP[X_{r,i}^{K-\mathrm{reg}}\le \delta' r^2| 0 \leftrightarrow Q_{z,r}]+r^2\bP[X_{r,i}^{K-\mathrm{reg}}\ge \delta'r^2, X_{r,i}^{K-\mathrm{line-good}}\le \frac{\delta}{\delta'} X_{r,i}^{K-\mathrm{reg}}]
\end{align*}
where $\delta'$ is chosen so that $C(K)\delta<\delta'<\alpha/2$ with $C(K)=K^d/p_c^{2dK}$. Here, the second inequality follows from Lemma~\ref{lambda}.

For the first term, we apply Lemmas~\ref{pioneerpoint} and \ref{Kregpoint} to obtain that
\begin{align*}
       &\bP[X_{r,i}^{K-\mathrm{reg}}\le \delta' r^2| 0 \leftrightarrow Q_{z,r}] \\
       &\le \bP[X_{r,i}\ge \alpha r^2,X_{r,i}^{K-\mathrm{reg}}\le \delta' r^2 | 0\leftrightarrow Q_{z,r}]+ \bP[X_{r,i}\le \alpha r^2| 0 \leftrightarrow Q_{z,r}]\\
       &\lesssim_{\Lambda} r^{d+1}e^{-c(\log \alpha r^2)^4}+\frac{\alpha}{(i\varepsilon)^2}.
    \end{align*}

For the second term, we recall the definition of $\mathrm{Reg}_i(B)$ from Definition~\ref{def:KLgood}. Conditioning on the number of the points in $\mathrm{Reg}_i(\mathcal C(0; \mathcal Q_{z,i}^c))$, the quantity $X_{r,i}^{K-\mathrm{line-good}}$ stochastically dominates a binomial distribution with parameters $|\mathrm{Reg}_i(\mathcal C(0; \mathcal Q_{z,i}^c))|$ and $p_c^{dK}$.  Therefore, it follows from the fact that $|\mathrm{Reg}_i(\mathcal C(0; \mathcal Q_{z,i}^c))|\gtrsim X_{r,i}^{K-\mathrm{reg}}/K^d$ and Azuma's inequality \cite[Theorem~7.2.1]{AS08} that, for suitably large $K$,
\begin{align*}
      & \bP[X_{r,i}^{K-\mathrm{reg}}\ge \delta'r^2, X_{r,i}^{K-\mathrm{line-good}}\le \frac{\delta}{\delta'} X_{r,i}^{K-\mathrm{reg}}]\\
       &\le \bP[|\mathrm{Reg}_i(\mathcal C(0; \mathcal Q_{z,i}^c))|\ge \delta'r^2/K^d, X_{r,i}^{K-\mathrm{line-good}}\le \frac{\delta}{\delta'}K^d|\mathrm{Reg}_i(\mathcal C(0; \mathcal Q_{z,i}^c))|]\\
       &\lesssim \exp\bigg(-\frac{p_c^{dK} \delta'r^2}{2K^d}\bigg)\bP[|\mathrm{Reg}_i(\mathcal C(0; \mathcal Q_{z,i}^c))|\ge \delta'r^2/K^d]\\
       &\le e^{-\delta r^2/2}.
    \end{align*}
    Here, we used the relation that $C(K)\delta<\delta'$. Therefore, by combining the above estimates, we obtain
    \begin{align*}
      \mathbb {P}[X_{r,i}^{K-\mathrm{line-good}}\le \delta r^2|0 \leftrightarrow Q_{z,r}]\lesssim_{\Lambda}  r^{d+1}e^{-c(\log \alpha r^2)^4}+\frac{\alpha}{(i\varepsilon)^2}+ r^2\exp(-\delta r^2/2).
    \end{align*}
\end{proof}

\subsection{Volume estimates via the second moment method}

In order to show that the number of points of $\mathcal{C}$ close to the $K$-line good points is suitably high, we will apply the second moment method. In particular, for a given $L\geq 1$, we consider the region
\begin{align*}
S_L^i=Q_{z, (i\varepsilon+1)r-L/2}\setminus Q_{z, (i\varepsilon+1)r-L}
\end{align*}
and define
\begin{align}\label{zli}
Z_{L}^i=\sum_{y \in S_L^i}\mathbbm{1}\left\{ y \xlongleftrightarrow[\mathrm{on}~Q_{y,2L}]{\mathrm{off}~C_e(0;\mathcal Q_{z,i}^c)}\mathrm{Reg}_i'(C_e(0;\mathcal Q_{z,i}^c))\cap Q_{y,L}\right\}.
\end{align}
(We include the `on' in the connection event to clarify the definition.) The key moment estimates are provided by Lemmas \ref{PZlowerbd} and \ref{PZupperbd}, and the main conclusion of the section is given in Lemma \ref{Iterating}. In order to establish these, we start by presenting two preparatory lemmas, the first of which concerns an estimate on the two-point function for a restricted percolation cluster.

\begin{lem}\label{restricted2pt}
    For each fixed $\Lambda>1$,  it holds that
\[\bP[x \xleftrightarrow{\mathcal{Q}_{0,\Lambda r}} y ] \asymp \tnorm{x-y}^{2-d},\qquad \forall x,y\in Q_{0,r},\:r\geq 1.
\]
\end{lem}
\begin{proof} This is \cite[Theorem~1.3]{CH}, whose hypotheses are implied by our standing assumptions $d>6$ and the two-point bound \eqref{eq:2pt}.
\end{proof}

To state the next result, which is an adaptation of \cite[Lemma 6.3]{ASS}, we introduce the $(d-4)$-capacity for any finite set $A$ by setting
\begin{align*}
    \mathrm{Cap}_{(d-4)}(A)=\left(\inf \left\{\sum_{x,y\in A} \tnorm{x-y}^{4-d}\mu(x)\mu(y): \:\mu~\text{ probability measure on}~A\right\}\right)^{-1}.
\end{align*}
We highlight that the following estimate is translation invariant; in the applications below we use it with the box centred at the relevant point rather than at the origin.

\begin{lem}\label{capacity}
For any positive constants $\Lambda>1$, $c\in (0,\Lambda-1)$, there exists $c' > 0$, such that the following holds. Let $A \subseteq \Z^d$ be a finite set containing the origin, and let $\mathrm{diam}(A):=\max_{x,y \in A} \|x-y\|$ be its diameter. Then
\[\bP(x \xleftrightarrow{ Q_{0,2\Lambda\text{diam}(A)}} A) \geq c' \cdot (d(x,A))^{2-d} \cdot \mathrm{Cap}_{d-4}(A)\]
for any $x\in Q_{0,\Lambda \mathrm{diam}(A)}$ such that $d(x, A) \geq c \cdot \mathrm{diam}(A)$. Here, for $x \in \Z^d$ and a subset $A \subseteq \Z^d$, we define $d(x,A)=\inf_{a \in A}\|x-a\|$.
\end{lem}

\begin{proof} Let $L=d(x,A)$, $B=Q_{0,2\Lambda\mathrm{diam}(A)}$ and let $v$ be a probability measure supported on $A$. We define
\begin{align*}
    Z=\sum_{a \in A}v(a)\frac{\ind{x\overset{B}{\longleftrightarrow }a}}{\bP[x \overset{B}{\longleftrightarrow } a]}.
\end{align*}
Note that the conditions of the lemma ensure that the probability in the denominator is strictly positive for all $x\in Q_{0,\Lambda \mathrm{diam}(A)}$ and $a\in A$. By applying the Paley-Zygmund inequality, we obtain
\begin{align*}
    \bP[x \overset{B}{\longleftrightarrow } A]\ge \bP[Z>0]\ge \frac{(\E Z)^2}{\E Z^2}.
\end{align*}
For the numerator, we have
\begin{align*}
    \E Z=\sum_{a \in A}v(a)\frac{\bP[x \overset{B}{\longleftrightarrow } a]}{\bP[x \overset{B}{\longleftrightarrow } a]}=1.
\end{align*}
The denominator is bounded above by
\begin{align*}
\E Z^2\le \sum_{a,b\in A}v(a)v(b)\frac{\bP[x  \leftrightarrow a,\: x \leftrightarrow  b]}{\bP[x \overset{B}{\longleftrightarrow } a]\bP[x \overset{B}{\longleftrightarrow } b]}\lesssim L^{2d-4}\sum_{a,b\in A}v(a)v(b)\bP[x  \leftrightarrow a, x \leftrightarrow  b].
\end{align*}
Here, in the first inequality we disregarded the restricted connections in the top line of the fraction, and in the second inequality, we applied Lemma~\ref{restricted2pt} for the probabilities in the bottom line,  together with the fact that $\lVert x-a \rVert\asymp \lVert x-b\rVert \asymp L$, which follows from $d(x,A)\ge  c\mathrm{diam}(A)$ (and the assumptions that $A$ contains the origin and $x\in Q_{0,\Lambda \mathrm{diam}(A)}$). Now, as in the proof of \cite[Lemma 6.3]{ASS}, we have that
\[\bP[x  \leftrightarrow a, x \leftrightarrow  b]\lesssim L^{6-2d}+L^{2-d}\tnorm{a-b}^{4-d}.\]
Substituting this above gives
\begin{align*}
    \E Z^2&\lesssim L^{2d-4}\sum_{a,b\in A}v(a)v(b)\left\{L^{6-2d}+L^{2-d}\tnorm{a-b}^{4-d}\right\}\\
    &\le L^2 + L^{d-2} \cdot \sum_{a,b \in A} \nu(a)\nu(b)\tnorm{a-b}^{4-d}.
\end{align*}
Hence, combining everything, this implies
\[\bP(x  \xleftrightarrow{B}A) \ge  \bP(Z > 0) \gtrsim
\frac{1}{L^2 + L^{d-2} \cdot \sum_{a,b \in A} \nu(a)\nu(b)\tnorm{a-b}^{4-d}}.
\]
Since
\[\sum_{a,b \in A} \nu(a)\nu(b)\tnorm{a-b}^{4-d} \gtrsim L^{4-d},\]
we obtain
\[
\bP(Z > 0) \gtrsim \frac{L^{2-d}}{\sum_{a,b \in A} \nu(a)\nu(b)\tnorm{a-b}^{4-d}}.
\]
Optimising over the choice of the probability measure $\nu$ proves the claim.
\end{proof}

We now give our lower bound on the first moment of $Z_L^i$, as introduced at \eqref{zli}. For the statement of this, we define $\mathcal S_i$ to be the set of possible extended configurations $\mathcal C_e(0;\mathcal Q^c_{z,i})$ (see \eqref{extclust}) for the level $i$ annulus, $C$ say, that are connected to the origin in $\mathcal Q_{z,i}^c$, and satisfy  $X_{r,i}^{K-\mathrm{line-good}}(C)> \delta r^2$, where $X_{r,i}^{K-\mathrm{line-good}}(C)$ is used to denote the number of $K$-line good points of $C$ on $\partial \mathcal Q_{z,i}$.

\begin{lem}\label{PZlowerbd} Let $\Lambda>3$. There exists a constant $c_1>0$ such that, for every $z \in Q_{0,\Lambda r}\setminus Q_{0,3r}$, $\varepsilon>0$, $i\in\{1,\ldots,\varepsilon^{-1}\}$, every $K$ satisfying $L\ge 10K$, where $L=\varepsilon r/4$, and every $C\in\mathcal S_i$,
\begin{align*}
\E\left[Z_{L}^i\ind{ \mathcal C_e(0; \mathcal Q_{z, i}^c)=C}\right]\ge c_1L^2X_{r,i}^{K-\mathrm{line-good}}(C)\bP\left[C_e(0; \mathcal Q_{z, i}^c)=C\right].
\end{align*}
\end{lem}
\begin{proof}
By the definition of $Z_{L}^i$, we have
    \begin{align*}
        \E\left[Z_{L}^i\ind{ \mathcal C_e(0; \mathcal Q_{z, i}^c)=C}\right]=\bP[\mathcal C_e(0; \mathcal Q_{z, i}^c)=C]\sum_{y \in S_L^i}\bP\left[y \xlongleftrightarrow[\mathrm{on}~Q_{y,2L}]{\mathrm{off}~C}\mathrm{Reg}_i'(C)\cap Q_{y,L}\right],
    \end{align*}
where we have used independence to separate the two events that depend on disjoint sets of edges. We will later show
    \begin{align}\label{localconnection}
        \bP\left[y \xlongleftrightarrow[\mathrm{on}~Q_{y,2L}]{\mathrm{off}~C}\mathrm{Reg}_i'(C)\cap Q_{y,L}\right]\gtrsim L^{2-d}|\mathrm{Reg}_i'(C)\cap Q_{y,L}|.
    \end{align}
Assuming \eqref{localconnection}, we obtain
\begin{align*}
   \E\left[Z_{L}^i\ind{ \mathcal C_e(0; \mathcal Q_{z, i}^c)=C}\right]&\gtrsim \bP[\mathcal C_e(0; \mathcal Q_{z, i}^c)=C]L^{2-d}\sum_{y \in S_L^i}|\mathrm{Reg}_i'(C)\cap Q_{y,L}|\\
   &= \bP[\mathcal C_e(0; \mathcal Q_{z, i}^c)=C]L^{2-d}\sum_{x \in \mathrm{Reg}_i'(C)}|Q_{x,L}\cap S_L^i| \\
   &\gtrsim L^2X_{r,i}^{K-\mathrm{line-good}}(C) \bP[\mathcal C_e(0; \mathcal Q_{z, i}^c)=C],
\end{align*}
which proves Lemma~\ref{PZlowerbd}.

In the remainder of the proof, we show \eqref{localconnection}. On the event that $y$ is connected to $\mathrm{Reg}_i'(C)\cap Q_{y,L}$ on $Q_{y,2L}$, but off $C$, there exists an $x'\in \mathrm{Reg}_i'(C)\cap Q_{y,L}$ connected to $y$ according to the relevant constraints. Let $Y$ be a point in $\mathrm{Reg}_i'(C)\cap Q_{y,L}$ such that $y$ is connected to $\partial Q_{Y,K}$ on $Q_{y,2L}$, but off $C\cup Q_{Y,K} $; if there is more than one such point, we choose the first according to an arbitrary ordering of points of $\mathrm{Reg}_i'(C)$. (See Figure \ref{fig:Ydef}.) Setting
\begin{align*}
    C^*=C\setminus \bigcup_{x\in \mathrm{Reg}_i(C):\, x~\text{is}~K\text{-line good}}L_x,
\end{align*}
we obtain
    \begin{align}
 \lefteqn{\bP\left[y \xlongleftrightarrow[\mathrm{on}~Q_{y,2L}]{\mathrm{off}~C}\mathrm{Reg}_i'(C)\cap Q_{y,L}\right]}&\nonumber\\
        &\ge \sum_{x' \in \mathrm{Reg}_i'(C)\cap Q_{y,L}}\bP\left[Y=x',\: y \xlongleftrightarrow[\mathrm{on}~Q_{y,2L}]{\mathrm{off}~C^*}Q_{x',K},\: \text{all edges in}~Q_{x',K}~\text{are open}\right]\nonumber\\
        &\ge C(K)\sum_{x' \in \mathrm{Reg}_i'(C)\cap Q_{y,L}}\bP\left[Y=x',\: y \xlongleftrightarrow[\mathrm{on}~Q_{y,2L}]{\mathrm{off}~C^*}Q_{x',K}\right]\nonumber\\
        &\ge C(K)\bP\left[ y \xlongleftrightarrow[\mathrm{on}~Q_{y,2L}]{\mathrm{off}~C^*} \bigcup_{x' \in \mathrm{Reg}_i'(C)\cap Q_{y,L}}Q_{x',K}\right]\nonumber\\
        &\ge C(K)\bP\left[y \xlongleftrightarrow[\mathrm{on}~Q_{y,2L}]{\mathrm{off}~C^*}\mathrm{Reg}_i'(C)\cap Q_{y,L}\right],\label{bound}
    \end{align}
    where the factor $C(K)=p_c^{d(2K+1)^d}$ arises from the independence of the two events in the probability. It remains to bound the last probability.

\begin{figure}
    \centering
    \includegraphics[width=0.6\textwidth]{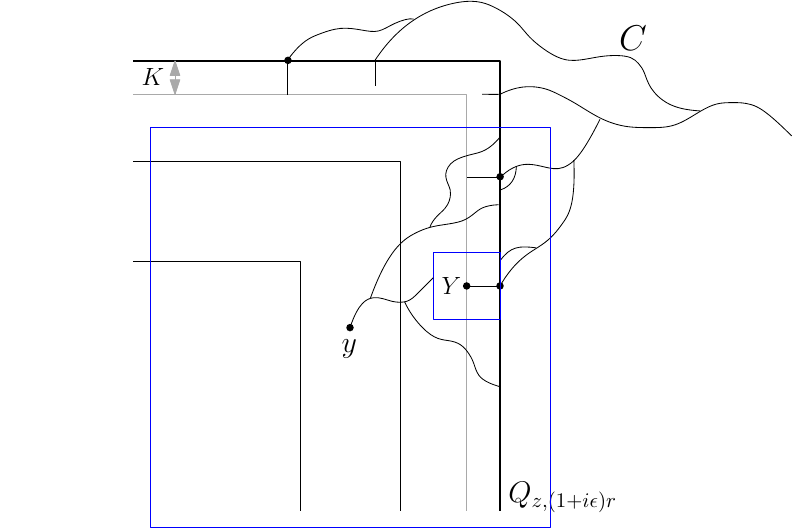}
    \caption{The blue box centered at $y$ and $Y$ are $Q_{y,L}$, and $Q_{Y,K}$, respectively. In the illustration, there is another point of $\mathrm{Reg}_i'(C)$ in $Q_{y,L}$ such that $y$ is connected, within  $Q_{y,2L}$ and off $C$, to the boundary of the box of side length $K$ centered at that point.  The point $Y$, however, is chosen as the first such point according to the ordering of $\mathrm{Reg}_i'(C)$. }
    \label{fig:Ydef}
\end{figure}

    Since the probability on the right-hand side of \eqref{bound} can be decomposed as
    \begin{align}
      &\bP\left[y \xlongleftrightarrow[\mathrm{on}~Q_{y,2L}]{\mathrm{off}~C^*}\mathrm{Reg}_i'(C)\cap Q_{y,L}\right]\nonumber\\
      &= \bP\left[y \xlongleftrightarrow[\mathrm{on}~Q_{y,2L}]{~} \mathrm{Reg}_i'(C)\cap Q_{y,L}\right]- \bP\left[y \xlongleftrightarrow[\mathrm{on}~Q_{y,2L}]{\mathrm{via}~C^*}\mathrm{Reg}_i'(C)\cap Q_{y,L}\right], \label{offrestconnection}
    \end{align}
    it suffices to derive a lower bound for the first term and an upper bound for the second term. Here, ``via $C^*$'' means that the connection uses at least one open edge of $C^*$.

    For the second term, by the BK inequality, we obtain
    \begin{align*}
        \lefteqn{\bP\left[y \xlongleftrightarrow[\mathrm{on}~Q_{y,2L}]{\mathrm{via}~C^*}\mathrm{Reg}_i'(C)\cap Q_{y,L}\right]}\\
        &\le \sum_{b \in C^*}\bP\left[y \xleftrightarrow[\text{on}~Q_{y,2L}]{~}b\right]\bP\left[b \xleftrightarrow[\text{on}~Q_{y,2L}]{~}\mathrm{Reg}_i'(C)\cap Q_{y,L}\right]\\
         &\lesssim L^{2-d}\sum_{b \in C^*}\bP\left[b \xleftrightarrow[\text{on}~Q_{y,2L}]{~}\mathrm{Reg}_i'(C)\cap Q_{y,L}\right]\\
         &\lesssim  L^{2-d}\sum_{b \in C^*}\sum_{x' \in \mathrm{Reg}_i'(C)\cap Q_{y,L}}\tau(x'-b),
    \end{align*}
where to deduce the second inequality, we apply the two-point estimate from \eqref{taubound}. (Note that the probability of $y$ being connected to $b$
for $b$ outside $Q_{y,2L}$ is zero, and otherwise $\|y-b\|\asymp L$.) By decomposing the sum according to the distance between $b$ and $x'$ (and applying \eqref{taubound}), we obtain
    \begin{align*}
        \sum_{b \in C^*}\sum_{x' \in \mathrm{Reg}_i'(C)\cap Q_{y,L}}\tau(x'-b)\asymp \sum_{x' \in \mathrm{Reg}_i'(C)\cap Q_{y,L}}\sum_{t=\log_2K}^{\infty}
        \frac{|B_t(x')|}{2^{t(d-2)}}
    \end{align*}
    where $B_t(x')=C^*\cap (Q_{x',2^{t+1}}\setminus Q_{x',2^{t}})$. (The restriction to $t\ge \log_2K$ is permitted since every point of $C^*$ is at distance at least $K$ from $x'$.) Moreover, from Definition \ref{def:Kreg}, we have that
    \begin{align*}
       |B_t(x')|&\le  (2^{t+1})^4(\log2^{t+1})^7+K(2^{t+1})^2(\log2^{t+1})^7
       \lesssim t^7(2^{t+1})^4.
    \end{align*}
(The second term is needed since $C^*$ removes only the full segments of the $K$-line good points, so the partial open segments of the other regular points remain in $C^*$; regularity bounds the number of their base points in $Q_{x',2^{t+2}}\cap\partial\mathcal Q_{z,i}$ by $(2^{t+2})^2(\log 2^{t+2})^7$, each contributing at most $K$ vertices. However, the term is absorbed since $K\le 2^{2t}$ for $t\ge\log_2K$.) Combining the above three inequalities yields
\begin{align}
    \lefteqn{\bP\left[y \xlongleftrightarrow[\mathrm{on}~Q_{y,2L}]{\mathrm{off}~C^*}\mathrm{Reg}_i'(C)\cap Q_{y,L}\right]}\nonumber\\
    &\lesssim L^{2-d}\sum_{x' \in \mathrm{Reg}_i'(C)\cap Q_{y,L}}\sum_{t=\log_2K}^{\infty}
        \frac{t^7(2^{t+1})^4}{2^{t(d-2)}}\nonumber\\
        &\lesssim \frac{L^{2-d}}{\sqrt{K}}|\mathrm{Reg}_i'(C)\cap Q_{y,L}|,\label{offconnection}
\end{align}
where the final inequality holds because $d>6$. On the other hand, by applying Lemma~\ref{capacity} and \cite[Claim~6.1]{ASS},
\begin{align}
    \bP\left[y \xlongleftrightarrow[\mathrm{on}~Q_{y,2L}]{~} \mathrm{Reg}_i'(C)\cap Q_{y,L}\right]&\gtrsim L^{2-d}\mathrm{Cap}_{d-4}(\mathrm{Reg}_i'(C)\cap Q_{y,L})\nonumber\\
    &\gtrsim L^{2-d}|\mathrm{Reg}_i'(C)\cap Q_{y,L}|. \label{restconnection}
\end{align}
Here, in the second inequality, we used the fact that $\mathrm{Reg}_i'(C)\cap Q_{y,L}$ satisfies the density condition of  \cite[Claim~6.1]{ASS}, which follows from the definition of regularity and the fact that $d>6$.
Indeed, let $x \in \mathrm{Reg}_i'(C)\cap Q_{y,L}$. The definition of  $\mathrm{Reg}_i'(C)$ implies that, whenever $s<3K$, $|Q_{x,s}\cap \mathrm{Reg}_i'(C)\cap Q_{y,L}|=1$. Next, consider $s\ge 3K$. The definition of regularity implies that, for any $x \in \mathrm{Reg}_i'(C)\cap Q_{y,L}$, any $\varepsilon>0$, and sufficiently large $K$
\begin{align*}
    |Q_{x,s} \cap \mathrm{Reg}_i'(C)\cap Q_{y,L}|\le \left| Q_{\tilde{x},2s}\cap C\cap \partial \mathcal Q_{z,i}\right|\lesssim s^2(\log s)^7 \lesssim s^{2+\varepsilon},
\end{align*}
where $\tilde{x}$ is the regular point  corresponding to $x$. Therefore, if $d>6$, there exist positive constants $C$ and $\alpha <d-4$ such that
\begin{align*}
 |Q_{x,s} \cap \mathrm{Reg}_i'(C)\cap Q_{y,L}|\le Cs^{\alpha},
\end{align*}
which is the density condition of \cite[Claim~6.1]{ASS}.

Inserting \eqref{offconnection} and \eqref{restconnection} into \eqref{offrestconnection}, we arrive at \eqref{localconnection}.
\end{proof}

The next result gives the corresponding second moment upper bound.

\begin{lem}\label{PZupperbd} Let $\Lambda >3$. There exists a constant $C_1>0$ such that, for every $z \in Q_{0,\Lambda r}\setminus Q_{0,3r}$, $\varepsilon>0$, $i\in\{1,\ldots,\varepsilon^{-1}\}$, every $K$ satisfying $L\ge 10K$, where $L=\varepsilon r/4$, every $\delta>0$ satisfying $\delta=\varepsilon^2$, and every $C\in\mathcal S_i$,
\begin{align*}
\E\left[(Z_{L}^i)^2\ind{ \mathcal C_e(0; \mathcal Q_{z, i}^c)=C}\right]\le C_1L^4X_{r,i}^{K-\mathrm{line-good}}(C)^2\bP\left[C_e(0; \mathcal Q_{z, i}^c)=C\right].
\end{align*}
\end{lem}
\begin{proof}
First, we have
\begin{align*}
    &\E\left[(Z_{L}^i)^2\ind{ \mathcal C_e(0; \mathcal Q_{z, i}^c)=C}\right]\\
    &\le \sum_{x_1,x_2 \in \mathrm{Reg}_i'(C)}\sum_{\substack{y_1\in S_L^i\cap Q_{x_1,L}\\ y_2\in S_L^i\cap Q_{x_2,L}}}\bP\left[y_1 \xlongleftrightarrow[~]{\mathrm{off}~C}x_1,y_2 \xlongleftrightarrow[~]{\mathrm{off}~C}x_2, C_e(0; \mathcal Q_{z, i}^c)=C \right]\\
    &\le \bP\left[C_e(0; \mathcal Q_{z, i}^c)=C \right]\sum_{x_1,x_2 \in \mathrm{Reg}_i'(C)}\sum_{\substack{y_1\in S_L^i\cap Q_{x_1,L}\\ y_2\in S_L^i\cap Q_{x_2,L}}}\bP\left[y_1 \longleftrightarrow x_1,y_2 \longleftrightarrow x_2\right],
\end{align*}
where the last inequality follows from the independence of the two events and by dropping the condition ``off $C$''. We bound the last sum, by considering four cases: $x_1=x_2$ and $y_1=y_2$; $x_1=x_2$ and $y_1\neq y_2$; $x_1\neq x_2$ and $y_1=y_2$; and $x_1\neq x_2$ and $y_1\neq y_2$.

In the case $x_1=x_2$ and $y_1=y_2$, the sum is easily estimated as follows:
\begin{align}
    \sum_{x_1\in \mathrm{Reg}_i'(C)}\sum_{y_1\in S_L^i\cap Q_{x_1,L}}\bP[y_1 \longleftrightarrow x_1]
    &\lesssim  \sum_{x_1\in \mathrm{Reg}_i'(C)}\sum_{y_1\in S_L^i\cap Q_{x_1,L}} \tnorm{x_1-y_1}^{2-d}\nonumber\\
    &\lesssim L^2X_{r,i}^{K-\mathrm{line-good}}(C) \label{easy}
\end{align}
where in the second inequality, we used \eqref{bxbd}.

For the case $x_1=x_2$ and $y_1\neq y_2$, there exists a vertex $w$ at which the paths from $x_1$ to $y_1$ and from $x_1$ to $y_2$ bifurcate. Then by using the BK inequality for the connection events to this point, we obtain
\begin{align}
   \lefteqn{\sum_{x_1\in \mathrm{Reg}_i'(C)}\sum_{y_1,y_2\in S_L^i\cap Q_{x_1,L}}\bP[y_1 \longleftrightarrow x_1,\:y_2 \longleftrightarrow x_1]}\nonumber\\
   &\le \sum_{x_1\in \mathrm{Reg}_i'(C)}\sum_{y_1,y_2\in S_L^i\cap Q_{x_1,L}}\sum_{w}\tau(y_1-w)\tau(y_2-w)\tau(w-x_1)\nonumber\\
   &=\sum_{x_1\in \mathrm{Reg}_i'(C)}\sum_{y_1\in S_L^i\cap Q_{x_1,L}}\sum_{w}\tau(y_1-w)\tau(w-x_1)\sum_{y_2\in S_L^i\cap Q_{x_1,L}}\tau(y_2-w)\nonumber\\
   &\lesssim L^2 \sum_{x_1\in \mathrm{Reg}_i'(C)}\sum_{y_1\in S_L^i\cap Q_{x_1,L}}G(y_1,x_1)\nonumber\\
   &\lesssim L^6X_{r,i}^{K-\mathrm{line-good}}(C)\label{easy2}
\end{align}
where we define
\begin{align*}
G(x,y):=\tnorm{x-y}^{4-d}.
\end{align*}
Here, the second and final inequalities follow from \eqref{bxbd}. Similarly, in the case $x_1\neq x_2$ and $y_1=y_2$, we obtain an upper bound of order $L^6X_{r,i}^{K-\mathrm{line-good}}(C)$ for the sum.

It remains to consider the case  $x_1\neq x_2$ and $y_1\neq y_2$. For the paths from $x_1$ to $y_1$, and from $x_2$ to $y_2$, there are two possibilities: either they are disjoint, or they intersect. If they intersect, let $u$ and $v$ be the first and last intersection points, respectively, along the path from $x_1$ to $y_1$. By applying the BK inequality, we obtain
\begin{align*}
 \lefteqn{\bP\left[y_1 \longleftrightarrow x_1,y_2 \longleftrightarrow x_2\right]}\\
 &\le \tau(x_1-y_1)\tau(x_2-y_2)+\sum_{u,v}\tau(x_1-u)\tau(u-v)\tau(v-y_1)\tau(x_2-v)\tau(u-y_2)\\
  &\hskip10mm+\sum_{u,v}\tau(x_1-u)\tau(u-v)\tau(v-y_1)\tau(x_2-u)\tau(v-y_2)
\end{align*}
For the first term, we obtain from \eqref{easy} that
\begin{align*}
  \sum_{x_1,x_2 \in \mathrm{Reg}_i'(C)}\sum_{\substack{y_1\in S_L^i\cap Q_{x_1,L}\\ y_2\in S_L^i\cap Q_{x_2,L}}}\tau(x_1-y_1)\tau(x_2-y_2)\lesssim L^4X_{r,i}^{K-\mathrm{line-good}}(C)^2.
\end{align*}
As for the second term, it is possible to follow an argument similar to \eqref{easy2} to give that
\begin{align*}
\lefteqn{\sum_{x_1,x_2 \in \mathrm{Reg}_i'(C)}\sum_{\substack{y_1\in S_L^i\cap Q_{x_1,L}\\ y_2\in S_L^i\cap Q_{x_2,L}}}\sum_{u,v}\tau(x_1-u)\tau(u-v)\tau(v-y_1)\tau(x_2-v)\tau(u-y_2)}\\
&=\lefteqn{\sum_{x_1,x_2 \in \mathrm{Reg}_i'(C)}\sum_{u,v}\tau(x_1-u)\tau(u-v)\tau(x_2-v)\sum_{\substack{y_1\in S_L^i\cap Q_{x_1,L}\\ y_2\in S_L^i\cap Q_{x_2,L}}}\tau(v-y_1)\tau(u-y_2)}\\
&\lesssim L^4\sum_{x_1,x_2 \in \mathrm{Reg}_i'(C)}\sum_{u}\tnorm{x_1-u}^{2-d}\sum_{v}\tnorm{u-v}^{2-d}\tnorm{x_2-v}^{2-d}\hspace{100pt}\\
&\lesssim L^4\sum_{x_1,x_2 \in \mathrm{Reg}_i'(C)}\tnorm{x_1-x_2}^{6-d}\\
&\lesssim L^4X_{r,i}^{K-\mathrm{line-good}}(C)^2.
\end{align*}
Here, the first inequality follows from \eqref{bxbd}, the second from \eqref{HHSbound}, and the final inequality simply follows from the fact that $\tnorm{x_1-x_2}^{6-d}\le 1$.

For the remaining term, by summing over $y_1$ and $v$, we obtain
\begin{align*}
 &\sum_{\substack{y_1\in S_L^i\cap Q_{x_1,L}\\ y_2\in S_L^i\cap Q_{x_2,L}}}\sum_{u,v}\tau(x_1-u)\tau(u-v)\tau(v-y_1)\tau(x_2-u)\tau(v-y_2)\\
 &\lesssim L^2 \sum_{y\in S_L^i\cap Q_{x_2,L}}\sum_{u}\tau(x_1-u)G(u,y)\tau(x_2-u)\\
 &\lesssim   L^2 \sum_{y\in S_L^i\cap Q_{x_2,L}}\sum_{u \in Q_{y,L/8}}\tau(x_1-u)G(u,y)\tau(x_2-u)\\
 &\qquad+ L^2 \sum_{y\in S_L^i\cap Q_{x_2,L}}\sum_{u \notin Q_{y,L/8}}\tau(x_1-u)G(u,y)\tau(x_2-u).
\end{align*}
For the case where $u \in Q_{y,L/8}$, using the bound $\tau(x_1-u)\tau(x_2-u)\lesssim L^{2(2-d)}$ and \eqref{bxbd}, we obtain
\begin{align*}
   L^2 \sum_{y\in S_L^i\cap Q_{x_2,L}}\sum_{u \in Q_{y,L/8}}\tau(x_1-u)G(u,y)\tau(x_2-u)
   &\lesssim L^2\cdot L^{2(2-d)}\sum_{y\in S_L^i\cap Q_{x_2,L}}\sum_{u \in Q_{y,L/8}}G(u,y)\\
   &\lesssim L^{10-d}.
\end{align*}
For the case where $u \notin Q_{y,L/8}$, applying the bound $G(u,y)\lesssim L^{4-d}$ together with \eqref{HHSbound} yields
\begin{align*}
   \lefteqn{L^2 \sum_{y\in S_L^i\cap Q_{x_2,L}}\sum_{u \notin Q_{y,L/8}}\tau(x_1-u)G(u,y)\tau(x_2-u)}\\
   &\lesssim L^2L^{4-d}\sum_{y\in S_L^i\cap Q_{x_2,L}}\sum_{u \notin Q_{y,L/8}}\tau(x_1-u)\tau(x_2-u)\\
   &\lesssim L^6G(x_1,x_2).
\end{align*}
Therefore, summing over $x_1,x_2 \in \mathrm{Reg}'_i(C)$, we obtain
\begin{align}\label{alldistinct}
  &\sum_{x_1,x_2 \in \mathrm{Reg}_i'(C)}\sum_{\substack{y_1\in S_L^i\cap Q_{x_1,L}\\ y_2\in S_L^i\cap Q_{x_2,L}}}\sum_{u,v}\tau(x_1-u)\tau(u-v)\tau(v-y_1)\tau(x_2-u)\tau(v-y_2)\nonumber\\
   &\lesssim \sum_{x_1,x_2 \in \mathrm{Reg}_i'(C)} (L^{10-d}+L^6G(x_1,x_2))\nonumber\\
    &\le L^{10-d}X_{r,i}^{K-\mathrm{line-good}}(C)^2+ L^6\sum_{x_1 \in \mathrm{Reg}_i'(C)}\sum_{t=\log_2K}^{\infty}2^{t(4-d)}|B_t(x_1)|
\end{align}
where $B_t(x_1)=\partial \mathcal Q_{z,i}\cap (Q_{x_1,2^{t+1}}\setminus Q_{x_1,2^{t}})$. It follows from Definition~\ref{def:Kreg} that
\begin{align*}
    |B_t(x_1)|\le (2^{t+1})^2(\log 2^{t+1})^7\lesssim t^7(2^{t+1})^2.
\end{align*}
Since $\sum_{t=\log_2K}^{\infty}2^{t(6-d)}t^7$ converges for $d>6$, the second term of \eqref{alldistinct} is bounded above by $L^6X_{r,i}^{K-\mathrm{line-good}}(C)$.

Combining the above estimates, we obtain
\begin{align*}
    \lefteqn{\sum_{x_1,x_2 \in \mathrm{Reg}_i'(C)}\sum_{\substack{y_1\in S_L^i\cap Q_{x_1,L}\\ y_2\in S_L^i\cap Q_{x_2,L}}}\bP[y_1 \longleftrightarrow x_1,y_2 \longleftrightarrow x_2]}\\
    &\lesssim L^6X_{r,i}^{K-\mathrm{line-good}}(C)+L^4X_{r,i}^{K-\mathrm{line-good}}(C)^2\\
    &\lesssim L^4X_{r,i}^{K-\mathrm{line-good}}(C)^2,
\end{align*}
where we used the fact that $X_{r,i}^{K-\mathrm{line-good}}(C)>\delta r^2$, $L=\varepsilon r/4$ and $\delta=\varepsilon^2$ to obtain the final equality. This completes the proof.
\end{proof}

We are now ready to put the various pieces together to conclude with the main estimate of this section. To state this, we define the volume of the cluster in $A_{z,i}$ by setting
\begin{align}
V(A_{z,i})=\left|\mathcal C\big(0; \mathcal Q_{z,{i-1}}^c\big)\cap A_{z,i} \right|.\label{voldef}
\end{align}

\begin{lem}\label{Iterating} Let $\Lambda>3$. There exist strictly positive constants $c$ and $c'$ such that for every $z \in Q_{0,\Lambda r}\setminus Q_{0,3r}$, $\eta>0$, $i\in\{1,\ldots,\varepsilon^{-1}\}$, where $\varepsilon=c'\eta^{1/4}$, and $C\in\mathcal S_i$, it holds that
\begin{align*}
\bP\left[V(A_{z,i})\le \eta r^4\mid \mathcal C_e(0; \mathcal Q_{z, i}^c)=C\right]\le 1-c.
\end{align*}
\end{lem}
\begin{proof}
Since $Z_{L}^i\le V(A_{z,i})$ for $L=\varepsilon r/4$, it suffices to show that there exists a constant $c$ such that
\begin{align*}
\bP\left[Z_{L}^i> \eta r^4\mid \mathcal C_e(0; \mathcal Q_{z, i}^c)=C\right]\ge c.
\end{align*}
By applying the Paley-Zygmund inequality together with Lemmas~\ref{PZlowerbd} and \ref{PZupperbd}, we obtain
\begin{align*}
\lefteqn{\bP\left[Z_{L}^i> \eta r^4\mid \mathcal C_e(0; \mathcal Q_{z, i}^c)=C\right]}\\
 &\ge  \bP\left[Z_{L}^i> \eta r^4\big(c_1L^2X_{r,i}^{K-\mathrm{line-good}}(C)\big)^{-1}\E[Z_{L}^i|\mathcal C_e(0; \mathcal Q_{z, i}^c)=C]\mid\mathcal C_e(0; \mathcal Q_{z, i}^c)=C\right]\\
 &\ge \left(1-\frac{\eta r^4}{c_1L^2X_{r,i}^{K-\mathrm{line-good}}(C)}\right)^2\frac{\left(\E[Z_{L}^i|\mathcal C_e(0; \mathcal Q_{z, i}^c)=C]\right)^2}{\E[(Z_{L}^i)^2|\mathcal C_e(0; \mathcal Q_{z, i}^c)=C]}\\
 &\ge \bigg(1-\frac{\eta r^2}{c_1\delta L^2 }\bigg)^2\frac{c_1^2}{C_1}\\
 &=c
\end{align*}
where $c_1$ and $C_1$ are the constants appearing in Lemma~\ref{PZlowerbd} and Lemma~\ref{PZupperbd}, respectively, and we have used that  $X_{r,i}^{K-\mathrm{line-good}}(C)>\delta r^2$ to deduce the final inequality. Moreover, the final equality is a consequence of the choices $L=\varepsilon r/4$, $\delta=\varepsilon^2$ and $\varepsilon=c'\eta^{1/4}$. This completes the proof.
\end{proof}

\subsection{Proof of Theorem \ref{thm:ulmb}}\label{proofsec}

To complete the proof of Theorem \ref{thm:ulmb}, we will apply Lemma \ref{Iterating} repeatedly for the different annuli in our decomposition. We will use the notation for volumes of annuli from \eqref{voldef}, and also define the event $E_i$ by setting
\begin{align*}
E_i=\left\{X_{r,i}^{K-\mathrm{line-good}}>\delta r^2, V(A_{z,i})\le \eta r^4\right\}.
\end{align*}

\begin{proof}[Proof of Theorem \ref{thm:ulmb}] For $\eta>0$, let $\varepsilon$, $\alpha$ and $\delta$ be given by
\begin{align}\label{conditions}
\varepsilon=c'\eta^{1/4},\qquad\alpha=4C(K)\varepsilon^2,\qquad\delta=\varepsilon^2,
\end{align}
where $c'$ is the constant appearing in Lemma~\ref{Iterating}, and $C(K)=\frac{K^d}{p_c^{2dK}}$ is the same constant as in the proof of Lemma~\ref{KLgood}. We then bound the probability of interest as follows:
\begin{align}
\lefteqn{\bP\left[|\mathcal C\cap Q_{z,2r}|\le \eta r^4\mid0 \leftrightarrow Q_{z,r}\right]}\nonumber\\
&\le \sum_{i=\varepsilon^{-1}/2}^{\varepsilon^{-1}}\bP\left[X_{r,i}^{K-\mathrm{line-good}}\le \delta r^2\mid0 \leftrightarrow Q_{z,r}\right]\nonumber\\
&\hspace{15pt}+\bP\left[|\mathcal C\cap Q_{z,2r}|\le \eta r^4\mbox{ and }X_{r,i}^{K-\mathrm{line-good}}> \delta r^2, \forall i\in \{\varepsilon^{-1}/2,\dots,\varepsilon^{-1}\}\mid0\leftrightarrow Q_{z,r}\right]\nonumber\\
&\le  \sum_{i=\varepsilon^{-1}/2}^{\varepsilon^{-1}}\bP\left[X_{r,i}^{K-\mathrm{line-good}}\le \delta r^2\mid0 \leftrightarrow Q_{z,r}\right]+\frac{1}{\bP[0 \leftrightarrow Q_{z,r}]}\bP\left[E_i,\forall i\in \{\varepsilon^{-1}/2,\dots,\varepsilon^{-1}\}\right]. \label{mass}
\end{align}
Applying Lemma~\ref{KLgood}, the first term can be bounded above as
\begin{align*}
\lefteqn{\sum_{i=\varepsilon^{-1}/2}^{\varepsilon^{-1}}\bP\left[X_{r,i}^{K-\mathrm{line-good}}\le \delta r^2\mid0 \leftrightarrow Q_{z,r}\right]}\\
  &\lesssim_{\Lambda} \sum_{i=\varepsilon^{-1}/2}^{\varepsilon^{-1}}\left[r^{d+1}e^{-c(\log \alpha r^2)^4}+\frac{\alpha}{(i\varepsilon)^2}+ r^2\exp(-\delta r^2/2)\right]\\
    &\lesssim_{\Lambda} \varepsilon^{-1}\left[r^{d+1}e^{-c(\log \alpha r^2)^4}+\alpha + r^2\exp(-\delta r^2/2)\right].
\end{align*}
For the second term on the right-hand side of \eqref{mass}, let $\tilde{\mathcal S}_i$ denote the set of possible extended configurations for the level $i$ annulus, $C$ say, that are connected to the origin in $\mathcal Q_{z,i}^c$, that satisfy  $X_{r,i}^{K-\mathrm{line-good}}(C)> \delta r^2$, and also for which the events $E_i$ occur for every $j=i+1,\cdots,\varepsilon^{-1}$. (Note this refines the definition of $\mathcal{S}_i$ above Lemma \ref{PZlowerbd}.) By applying Lemma~\ref{Iterating}, we obtain that for any $\varepsilon^{-1}/2\le j\le \varepsilon^{-1}-1$,
\begin{align*}
\bP\left[E_i,\forall i\in \{j,\dots,\varepsilon^{-1}\}\right]&=\sum_{C \in \tilde{\mathcal S}_j}\bP\left[V(A_{z,j})\le \eta r^4\mid C_e(0; \mathcal Q_{z, j}^c)=C\right]\bP\left[C_e(0; \mathcal Q_{z, j}^c)=C\right]\\
&\le (1-c)\bP\left[E_i,\forall i\in \{j+1,\dots,\varepsilon^{-1}\}\right].
\end{align*}
Iterating this inequality until reaching $\bP[E_{\varepsilon^{-1}}]$ yields
\begin{align*}
\bP\left[E_i,\forall i\in \{\varepsilon^{-1}/2,\dots,\varepsilon^{-1}\}\right]
&\le (1-c)\bP\left[E_i,\forall i\in \{\varepsilon^{-1}/2+1,\dots,\varepsilon^{-1}\}\right]\\
&\hspace{28pt}\vdots\\
& \le (1-c)^{\varepsilon^{-1}/2-1}\bP[E_{\varepsilon^{-1}}]\\
&\lesssim (1-c)^{\varepsilon^{-1}/2}\bP[0 \leftrightarrow Q_{z,2r}].
\end{align*}
Thus, by applying Lemma~\ref{lambda}, we obtain
\begin{align*}
    \frac{1}{\bP[0 \leftrightarrow Q_{z,r}]}\bP\left[E_i,\forall i\in \{\varepsilon^{-1}/2,\dots,\varepsilon^{-1}\}\right]\lesssim  (1-c)^{\varepsilon^{-1}/2}\frac{\bP[0 \leftrightarrow Q_{z,2r}]}{\bP[0 \leftrightarrow Q_{z,r}]}\lesssim_{\Lambda} (1-c)^{\varepsilon^{-1}/2}.
\end{align*}
Combining the above previous inequalities provides
\begin{align*}
\lefteqn{\bP\left[|\mathcal C\cap Q_{z,2r}|\le \eta r^4\mid0 \leftrightarrow Q_{z,r}\right]}\\&\lesssim_{\Lambda} \varepsilon^{-1}\left[r^{d+1}e^{-c(\log \alpha r^2)^4}+\alpha+ r^2\exp(-\delta r^2/2)\right]+(1-c)^{\varepsilon^{-1}/2}.
\end{align*}
Now, the upper bound here converges to $\varepsilon^{-1}\alpha+(1-c)^{\varepsilon^{-1}/2}$ as $r\rightarrow\infty$, and since, by the choices at \eqref{conditions}, $\alpha$ is of the order of $\varepsilon^2$, and $\varepsilon$ is of the order of $\eta^{1/4}$, this yields the desired bound.
\end{proof}

\subsection{Proof of Theorem \ref{thm:mainLMB}}\label{proofsec2}

Given Theorem \ref{thm:ulmb}, the proof of Theorem \ref{thm:mainLMB} is now straightforward. For the proof of the second claim, we use the following estimate on the tail of the distribution of the size of $\mathcal{C}$,
\begin{equation}\label{voltail}
\bP\left(|\mathcal{C}|\geq R\right)\asymp R^{-1/2},
\end{equation}
which, as recalled at the start of this section, holds under our hypotheses.

\begin{proof}[Proof of Theorem \ref{thm:mainLMB}]
Let $\Lambda\geq 3$, and define
\[\mathcal{A}_{\Lambda,R}:=\left\{\mathrm{diam}\mathcal{C}\in[R,\Lambda R]\right\}.\]
Furthermore, for $\delta\in (0,1)$, suppose $(z_i)_{i=1}^n$ are the points of
\[\left(Q_{0,\Lambda R}\backslash Q_{0,3\delta R/10}\right)\cap\left(\left(\frac{\delta}{10}R\right)\Z^d\right),\]
so that $n\asymp (\Lambda/\delta)^d$, and set, for $\eta\in (0,1)$,
\[\mathcal{B}_{\Lambda,\delta,\eta,R}:=\left\{\left|\mathcal{C}\cap Q_{z_i,\delta R/5}\right|> \eta R^4\mbox{ for all $z_i$ such that }
\left|\mathcal{C}\cap Q_{z_i,\delta R/10}\right|\neq 0\right\}.\]
(As we have commonly done, when $\delta R/10$ is not an integer, we consider $\lfloor \delta R/10\rfloor$.) Now, suppose $\mathcal{A}_{\Lambda,R}\cap\mathcal{B}_{\Lambda,\delta,\eta,R}$ holds. We claim that, for all $x\in \mathcal{C}$, there exists a $z_i$ such that $x\in Q_{z_i,7\delta R/10}$ and $|\mathcal{C}\cap Q_{z_i,\delta R/10}|\neq 0$. Indeed, if $x\in \mathcal{C}\backslash Q_{3\delta R/10}$, then there must be a $z_i$ such that $x\in Q_{z_i,\delta R/10}$, which automatically gives that $|\mathcal{C}\cap Q_{z_i,\delta R/10}|\neq 0$. On the other hand, since $\mathrm{diam}\mathcal{C}\geq R$, then there must exist a $z_i$ with $\|z_i\|\in[3\delta R/10,4\delta R/10]$ such that $|\mathcal{C}\cap Q_{z_i,\delta R/10}|\neq 0$. For $x\in \mathcal{C}\cap Q_{0,3\delta R/10}$, it holds for this choice of $z_i$ that $x\in Q_{z_i,7\delta R/10}$. Consequently, in either case,
\[\left|\mathcal{C}\cap Q_{x,\delta R}\right|\geq \left|\mathcal{C}\cap Q_{z_i,\delta R/5}\right|>\eta R^4.\]
This implies
\[\inf_{x\in \mathcal{C}}\left|\mathcal{C}\cap Q_{x,\delta R}\right|>\eta R^4.\]

From the above argument, we obtain that
\begin{align*}
    \bP\left(\inf_{x\in \mathcal{C}}\left|\mathcal{C}\cap Q_{x,\delta R}\right|\leq \eta R^4\mid0\leftrightarrow Q_{0,R}^c\right)&\leq \bP\left(\mathcal{A}_{\Lambda,R}^c\cup\mathcal{B}_{\Lambda,\delta,\eta,R}^c\mid0\leftrightarrow Q_{0,R}^c\right)\\
    &\leq \bP\left(\mathcal{A}_{\Lambda,R}^c\mid0\leftrightarrow Q_{0,R}^c\right)+\bP\left(\mathcal{B}_{\Lambda,\delta,\eta,R}^c\mid0\leftrightarrow Q_{0,R}^c\right).
\end{align*}
For the first of these terms, we simply note that
\begin{align}
\bP\left(\mathcal{A}_{\Lambda,R}^c\mid0\leftrightarrow Q_{0,R}^c\right)&\le\bP\left(0\leftrightarrow Q_{0,\Lambda R/2}^c\mid0\leftrightarrow Q_{0,R}^c\right)\nonumber\\
&=\frac{\bP\left(0\leftrightarrow Q_{0,\Lambda R/2}^c\right)}{\bP\left(0\leftrightarrow Q_{0,R}^c\right)}\nonumber\\
&\lesssim \frac{1}{\Lambda^2},\label{lambd}
\end{align}
where we apply \eqref{onearm} to deduce the final inequality; here we use that, since $0\in\mathcal C$, the event $\{\mathrm{diam}\,\mathcal{C}>\Lambda R\}$ implies $0\leftrightarrow Q^c_{0,\Lambda R/2}$. As for the second, we have that
\begin{align*}
\lefteqn{\bP\left(\mathcal{B}_{\Lambda,\delta,\eta,R}^c\mid0\leftrightarrow Q_{0,R}^c\right)}\\
&=\bP\left(\left|\mathcal{C}\cap Q_{z_i,\delta R/5}\right|\leq  \eta R^4\mbox{ for some $z_i$ such that }
\left|\mathcal{C}\cap Q_{z_i,\delta R/10}\right|\neq 0\mid0\leftrightarrow Q_{0,R}^c\right)\\
&\lesssim \left(\frac{\Lambda}{\delta}\right)^d\sup_{i=1,\dots,n}\bP\left(\left|\mathcal{C}\cap Q_{z_i,\delta R/5}\right|\leq  \eta R^4\mbox{ and }\left|\mathcal{C}\cap Q_{z_i,\delta R/10}\right|\neq0\mid0\leftrightarrow Q_{0,R}^c\right)\\
&\leq \left(\frac{\Lambda}{\delta}\right)^d\sup_{i=1,\dots,n}\frac{\bP\left(0\leftrightarrow Q_{z_i,\delta R/10}\right)}{\bP\left(0\leftrightarrow Q_{0,R}^c\right)}
\bP\left(\left|\mathcal{C}\cap Q_{z_i,\delta R/5}\right|\leq  \eta R^4\mid0\leftrightarrow Q_{z_i,\delta R/10}\right)\\
&\leq C_{\Lambda,\delta}\sup_{z\in Q_{0,\Lambda R}\backslash Q_{0,3\delta R/10}}\bP\left(\left|\mathcal{C}\cap Q_{z,\delta R/5}\right|\leq  \eta R^4\mid0\leftrightarrow Q_{z,\delta R/10}\right),
\end{align*}
where the final inequality is a consequence of \eqref{onearm} and Lemma \ref{lambda}, with the constant $C_{\Lambda,\delta}$ only depending on $\Lambda$ and $\delta$. By Theorem \ref{thm:ulmb}, the final expression converges to zero as $R\rightarrow\infty$ and then $\eta\rightarrow0$. Hence, recalling \eqref{lambd}, subsequently taking $\Lambda\rightarrow \infty$ yields the first part of the result.

For the second part, we start by observing that
\begin{align*}
\bP\left(\mathrm{diam}\mathcal{C}\leq \varepsilon R,\:|\mathcal{C}|\geq R^4\right)&\leq \bP\left(|\mathcal{C}\cap Q_{0,\varepsilon R}|\geq R^4\right)\\
&\leq R^{-4} \sum_{x\in Q_{0,\varepsilon R}}\tau(x)\\
&\lesssim  R^{-4}(\varepsilon R)^2\\
&=\varepsilon^2R^{-2},
\end{align*}
where we have applied the bound on $\tau(x)$ from \eqref{taubound} and \eqref{bxbd}. Hence,
\begin{align*}
    \lefteqn{\bP\left(\inf_{x\in \mathcal{C}}\left|\mathcal{C}\cap Q_{x,\delta R}\right|\leq \eta R^4\mid|\mathcal{C}|\geq R^4\right)}\\
    &\leq \varepsilon^2+\bP\left(\inf_{x\in \mathcal{C}}\left|\mathcal{C}\cap Q_{x,\delta R}\right|\leq \eta R^4,\:0\leftrightarrow Q_{0,{\varepsilon R/2}}^c\mid|\mathcal{C}|\geq R^4\right)\\
    &\leq \varepsilon^2+\frac{\bP\left(0\leftrightarrow Q_{0,{\varepsilon R/2}}^c\right)}{\bP\left(|\mathcal{C}|\geq R^4\right)}\bP\left(\inf_{x\in \mathcal{C}}\left|\mathcal{C}\cap Q_{x,\delta R}\right|\leq \eta R^4\mid0\leftrightarrow Q_{0,{\varepsilon R/2}}^c\right)\\
    &\lesssim \varepsilon^2+\varepsilon^{-2}\bP\left(\inf_{x\in \mathcal{C}}\left|\mathcal{C}\cap Q_{x,\delta R}\right|\leq \eta R^4\mid0\leftrightarrow Q_{0,{\varepsilon R/2}}^c\right),
\end{align*}
where the final inequality follows from \eqref{onearm} and \eqref{voltail}. Now, by the first part of the theorem, we have that the second term converges to zero as $R\rightarrow\infty$ and then $\eta\rightarrow0$. Thus, the proof is completed on letting $\varepsilon\rightarrow 0$.
\end{proof}

\section{Convergence of the rescaled cluster measure}
\label{sec:measureconv}

Let us sketch the argument for the convergence of the rescaled cluster
measure. The classical method of moments does not apply, for two reasons:
\begin{itemize}
\item the restriction events $\{\mu(\R^d)>\eps\}$ are invisible to moments
(their indicators are not expressible through $k$-point functions);
\item $\bN$ is not a finite measure: it is infinite, though
$\sigma$-finite (Proposition~\ref{prop:totalmass},
Section~\ref{sec:sbmstructure}).

\end{itemize}
We resolve
this by \emph{size-biasing}: for $\varphi_0\in\Cc$ with $\varphi_0\not\equiv0$ we set
\begin{equation}\label{eq:sizebias}
\nuh_n \;:=\; \mu(\varphi_0)\,\nu_n(d\mu),
\qquad
\Nh \;:=\; \mu(\varphi_0)\,\bN(d\mu).
\end{equation}
Unlike $\bN$, the size-biased measure $\Nh$ is now a
\emph{finite} measure: its total mass is
$\Nh(\cM)=\bN[\mu(\varphi_0)]=\int G\varphi_0\in(0,\infty)$, finite because
$\varphi_0$ has compact support and the Green function $G$ is locally integrable
(the $k=1$ case of the first-moment identity, Proposition~\ref{prop:occmoments};
see \eqref{eq:polar}).
Weighting by $\mu(\varphi_0)$ raises the order of every moment by one, so the
moments of $\nuh_n$ are again rescaled $k$-point functions supplied by
\cite{CCHS26} (this convergence of moments is established in
Sections~\ref{sec:trees} and~\ref{sec:percolation}); the limit $\Nh$ is determined by its moments
(Lemma~\ref{lem:determinacy}), and we prove that $\{\nuh_n\}$ is tight using the one-arm bound
\cite{KN11} (Lemma~\ref{lem:tightness}). Hence $\nuh_n\Rightarrow\Nh$
(Theorem~\ref{thm:biased}, Section~\ref{sec:biased}), and \emph{un-biasing} --- dividing by
$\mu(\varphi_0)$ and exhausting with $\varphi_0\uparrow\mathbf 1$ --- recovers
the $\sigma$-finite convergence of $\nu_n$ to $\bN$, Theorem~\ref{thm:mainA}
(Section~\ref{sec:sigmafinite}).

\subsection{Tree integrals: uniform bounds, moment measures, exponential-moment growth}
\label{sec:trees}

Everything in this section is valid for $d>4$: the only dimensional input is
the integrability at infinity of the squared Green kernel. The percolation hypotheses --- and with them the restriction $d>6$ ---
enter only from Section~\ref{sec:percolation} onward.

\subsubsection{Binary trees and tree integrals}\label{sec:treedef}

For $k\ge1$ let $\T_k$ denote the set of binary trees with leaf set
$\{0,1,\dots,k\}$: trees in which the $k+1$ labelled \emph{leaves} have
degree $1$ and every other (\emph{internal}) vertex has degree $3$. Such a
tree has $k-1$ internal vertices and $2k-1$ edges, and
\begin{equation}\label{eq:treecount}
|\T_k|=(2k-3)!!\;\le\;2^k\,k!\,,
\end{equation}
with the convention $(-1)!!=1$ (so $|\T_1|=1$, the single edge). Here $(2k-3)!!=(2k-3)(2k-5)\cdots3\cdot1$ denotes the double factorial, the product of the odd integers up to $2k-3$; the count $|\T_k|=(2k-3)!!$ follows by induction, since every tree in $\T_k$ arises from a unique tree in $\T_{k-1}$ by subdividing one of its $2k-3$ edges with a new internal vertex to which leaf~$k$ is attached, so that $|\T_k|=(2k-3)\,|\T_{k-1}|$. The leaf
labelled $0$ plays the role of the \emph{root}, which will be pinned at the
origin. For $T\in\T_k$, with internal vertex set $V_{\mathrm{int}}(T)$ and edge
set $E(T)$, and for pairwise distinct $z_1,\dots,z_k\in\R^d\setminus\{0\}$,
define
\begin{equation}\label{eq:IT}
I_T(z_1,\dots,z_k)
:=\int_{(\R^d)^{V_{\mathrm{int}}(T)}}
\prod_{\{a,b\}\in E(T)}G(u_a-u_b)
\prod_{v\in V_{\mathrm{int}}(T)} du_v,
\qquad u_0:=0,\ u_i:=z_i .
\end{equation}
Here $G$ is the Green function of Brownian motion of Section~\ref{sec:notation}, $G(x)=c_d|x|^{-(d-2)}$.
When $V_{\mathrm{int}}(T)=\emptyset$ --- equivalently $k=1$ --- the single tree $T\in\T_1$ has the one edge
$\{0,1\}$, so $I_T(z_1)=G(u_0-u_1)=G(z_1)$; this is the case used in the
two-point normalization below.
Convergence of these integrals (for $d>4$) is part of
Lemma~\ref{lem:envelope} below.

\begin{remark}\label{rem:treeconvention}
In \cite{CCHS26} the $k$-point limit is expressed through binary trees on
the $k$ labelled leaves $y_0,\dots,y_{k-1}$ (with $k-2$ internal vertices and
$2k-3$ edges), all leaves being free arguments. Specializing $y_0=0$
identifies their tree class for the $(m{+}1)$-point function with our
$\T_m$.
\end{remark}

\subsubsection{A uniform domination bound}

Set
\[
\bar G(x):=|x|^{-(d-2)},\qquad x\in\R^d\setminus\{0\},
\]
so that $G=c_d\,\bar G$ (recall $c_d=\Gamma(d/2-1)/(2\pi^{d/2})$ from Section~\ref{sec:notation}). The following domination bound serves two purposes: it
shows the tree integrals \eqref{eq:IT} converge for $d>4$, with a bound that grows exponentially with the size of the tree, and it provides the integrable dominating
function for the lattice moment sums of Section~\ref{sec:percolation}.

\begin{lemma}[Uniform domination bound]\label{lem:envelope}
Let $d>4$ and $R_0\ge1$. There is a constant $C_\star=C_\star(R_0)$$<\infty$
such that for
every $k\ge1$, every $T\in\T_k$, and all measurable
$\phi_1,\dots,\phi_k$ with $|\phi_i|\le\mathbf 1_{\overline B_{R_0}}$,
\begin{equation}\label{eq:envelope}
\int_{(\R^d)^{k}}\int_{(\R^d)^{V_{\mathrm{int}}(T)}}
\prod_{\{a,b\}\in E(T)}\bar G(u_a-u_b)
\prod_{v\in V_{\mathrm{int}}(T)}du_v\,
\prod_{i=1}^k|\phi_i(z_i)|\,dz_i
\;\le\;C_\star^{\,k},
\end{equation}
with $u_0:=0$ and $u_i:=z_i$ as in \eqref{eq:IT}. In particular
$I_T(z_1,\dots,z_k)<\infty$ for Lebesgue-a.e.\ $(z_1,\dots,z_k)$, and, after
enlarging $C_\star$ if necessary,
\[
\int_{(\R^d)^k}I_T(\vec z)\prod_i|\phi_i(z_i)|\,dz_i\le
C_\star^{\,k}.
\]

\end{lemma}

\begin{proof}
Since the integrand depends on the $\phi_i$ only through $|\phi_i|$, we may assume $0\le\phi_i\le\mathbf 1_{\overline B_{R_0}}$. By scaling we may take $R_0=1$ (rescaling space by
$R_0$ multiplies the left side of \eqref{eq:envelope} by $R_0^{2(2k-1)}$ ---
a factor $R_0^{d}$ per integration variable against $R_0^{-(d-2)}$ per edge
kernel --- which is again geometric in $k$). Set
\[
m(u):=(1+|u|)^{-(d-2)} .
\]
We will use repeatedly the elementary polar-coordinate identity
\begin{equation}\label{eq:polar}
\int_{|v|\le R}|v|^{-(d-2)}\,dv=\frac{\omega_{d-1}}{2}\,R^2
\qquad(R>0),
\end{equation}
where $\omega_{d-1}$ is the surface area of the unit sphere (the exponent
$d-1$ from the Jacobian minus $d-2$ from the kernel leaves $r^1$).
We make two claims.

\emph{Claim 1 (leaf reduction).} There is $c_1=c_1(d)$ such that for
$0\le\phi\le\mathbf 1_{\overline B_1}$,
\[
f_\phi(u):=\int\bar G(u-z)\phi(z)\,dz\;\le\;c_1\,m(u).
\]
Indeed, for $|u|\ge2$ and $z\in\overline B_1$ one has $|u-z|\ge|u|-1$, and
since $t\mapsto\frac{1+t}{t-1}$ is decreasing on $(1,\infty)$ with value $3$
at $t=2$, $(|u|-1)^{-(d-2)}\le 3^{d-2}(1+|u|)^{-(d-2)}$; thus
$\bar G(u-z)\le 3^{d-2}m(u)$, and integrating over $z\in\overline B_1$
gives $f_\phi(u)\le 3^{d-2}(\omega_{d-1}/d)\,m(u)$ (here $\omega_{d-1}/d=|B_1|$ is the volume of the unit ball). For $|u|\le2$,
$\int_{B_1}\bar G(u-z)\,dz\le\int_{B_3}|w|^{-(d-2)}dw<\infty$ (finite
by \eqref{eq:polar} with $R=3$), while $m(u)\ge(1+2)^{-(d-2)}=3^{-(d-2)}$
there, so $f_\phi(u)\le c\,m(u)$. Taking $c_1$ the larger of the two constants
proves the claim.

\emph{Claim 2 (merging).} There is $c_2=c_2(d)$ such that if
$0\le f\le\kappa\,m$ and $0\le g\le\kappa'\,m$, then
\[
h(u):=\int\bar G(u-w)\,f(w)g(w)\,dw\;\le\;c_2\,\kappa\kappa'\,m(u).
\]
To see this, note $fg\le\kappa\kappa' m^2$ and
$m^2(w)=(1+|w|)^{-2(d-2)}\in L^1(\R^d)$ because $2(d-2)>d$ for $d>4$; write
$c_m:=\|m^2\|_{L^1}$. For $|u|\le2$,
\[
h(u)\le\kappa\kappa'\Bigl[\int_{|u-w|\le1}|u-w|^{-(d-2)}\,dw
+c_m\Bigr]\le c\,\kappa\kappa'\;\le\;c'\,\kappa\kappa'\,m(u),
\]
the local integral being finite by \eqref{eq:polar} with $R=1$, and the last inequality using only that $|u|\le2$, so that $m(u)\ge3^{-(d-2)}$.
For $|u|\ge2$, split the integral into three regions:
(a) $|w|\le|u|/2$: then $\bar G(u-w)\le(|u|/2)^{-(d-2)}$, contributing
at most $2^{d-2}c_m\kappa\kappa'|u|^{-(d-2)}\le c\,\kappa\kappa'\,m(u)$ (using $|u|^{-(d-2)}\le2^{d-2}m(u)$, valid since $|u|\ge2$);
(b) $|w|\ge|u|/2$ and $|u-w|\le|u|/2$: then
$m^2(w)\le(1+|u|/2)^{-2(d-2)}$, while by the substitution $v=u-w$ and
\eqref{eq:polar},
\[
\int_{|u-w|\le|u|/2}|u-w|^{-(d-2)}\,dw
=\int_{|v|\le|u|/2}|v|^{-(d-2)}\,dv=\frac{\omega_{d-1}}{2}\Bigl(\frac{|u|}{2}\Bigr)^2
\le c|u|^2;
\]
hence this region contributes at most
$c\,\kappa\kappa'\,|u|^{2}(1+|u|/2)^{-2(d-2)}
\le c'\,\kappa\kappa'(1+|u|)^{-(d-2)}=c'\,\kappa\kappa'\,m(u)$, using $1+|u|/2\ge(1+|u|)/2$,
$|u|\le1+|u|$ and $2(d-2)-2\ge d-2$, i.e.\
$d\ge4$;
(c) $|w|\ge|u|/2$ and $|u-w|\ge|u|/2$: then
$\bar G(u-w)\le(|u|/2)^{-(d-2)}$, contributing at most
$2^{d-2}c_m\kappa\kappa'|u|^{-(d-2)}\le c\,\kappa\kappa'\,m(u)$ (again $|u|^{-(d-2)}\le2^{d-2}m(u)$ since $|u|\ge2$). As each of the three regions is bounded by a multiple of $\kappa\kappa'\,m(u)$, Claim 2 follows.

Now fix $T\in\T_k$ and orient it away from the root leaf $0$. Integrate out
the variables in order of decreasing graph distance from the root, as
follows: each non-root leaf $i$ is first replaced by the function
$f_{\phi_i}$ of its parent's position (Claim 1); then, iteratively, each
internal vertex $v$ whose two children have already been reduced to
functions $f,g$ of $u_v$ with $f\le\kappa m$, $g\le\kappa'm$ is integrated
out, producing the function $h\le c_2\kappa\kappa' m$ of its parent's
position (Claim 2). After all $k-1$ internal vertices are processed, there
remains the edge into the root. If $k=1$ the tree is the single edge
$\{0,1\}$ and there is nothing left to integrate: the left side of
\eqref{eq:envelope} equals $f_{\phi_1}(0)\le c_1\,m(0)=c_1$ by Claim~1
alone. If $k\ge2$ the root's unique neighbour is an internal vertex; once
every other internal vertex is processed, the two subtree weights
$f\le\kappa\,m$ and $g\le\kappa'\,m$ hang at this neighbour, and the final
integral
\[
\int \bar G(0-u)\,f(u)g(u)\,du
\]
is Claim 2's computation evaluated at the root, giving at most
$c_2\,\kappa\kappa'\,m(0)=c_2\,\kappa\kappa'$. Each application of
Claim 1 contributes a factor $c_1$ (there are $k$ of them) and each
application of Claim 2 a factor $c_2$ (at most $k-1$ of them), so the total
is at most $C_\star^k$ with $C_\star=C_\star(R_0)$. Tonelli's
theorem (all integrands nonnegative) justifies the iterated integration and
yields the a.e.\ finiteness of $I_T$ and \eqref{eq:envelope}. For the final
display, each of the $2k-1$ edges of $T$ carries $G=c_d\bar G$, so the left
side gains a factor $c_d^{\,2k-1}\le\bigl((1\vee c_d)^{2}\bigr)^{k}$, which is
absorbed by enlarging $C_\star$ to $(1\vee c_d)^{2}C_\star$.
\end{proof}

\subsubsection{Moment measures of the occupation measure under $\bN$}

We next express the moments of $\mu$ under $\bN$ as sums of tree integrals.
The starting point is the classical moment formula for superprocesses at
fixed times, which goes back to Dynkin \cite{Dyn88} and, in the graphical
form convenient here, to Adler \cite[Theorem~3.1]{Adl93}; see also
\cite[Theorem 2.3]{Eth00}. We state it under $\bN$, in the normalization fixed
by \eqref{eq:laplace}.

\begin{proposition}[Fixed-time moment formula]
For $k\ge1$, $t_1,\dots,t_k>0$ and $\phi_1,\dots,\phi_k\in\Cc$,
\begin{equation}\label{eq:timemoments}
\bN\Bigl[\prod_{i=1}^k X_{t_i}(\phi_i)\Bigr]
=\gamma^{\,k-1}\sum_{T\in\T_k}
\int
\prod_{\{a,b\}\in E(T)} p_{\,s_b-s_a}\bigl(y_b-y_a\bigr)
\prod_{v\in V_{\mathrm{int}}(T)} ds_v\,dy_v\,
\prod_{i=1}^k\phi_i(z_i)\,dz_i ,
\end{equation}
where each vertex $a$ of $T$ carries a space-time point $(s_a,y_a)$; the
root is pinned at $(0,0)$, the leaf $i$ at $(t_i,z_i)$; the edges
$\{a,b\}$ are written with $a$ the endpoint closer to the root; internal
vertices are integrated over $(0,\infty)\times\R^d$; and $p_s\equiv0$ for
$s\le0$ (which enforces the genealogical time-ordering). Here $\gamma$ is the branching parameter of \eqref{eq:laplace} (Section~\ref{sec:SBM}) and $p_s(x)=(2\pi s)^{-d/2}e^{-|x|^2/2s}$ is the heat kernel of Section~\ref{sec:notation}.
\end{proposition}

\begin{proposition}[Occupation moment measures]\label{prop:occmoments}
For every $k\ge1$ and $\phi_1,\dots,\phi_k\in\Ccg$,
\begin{equation}\label{eq:occmoments}
\bN\Bigl[\prod_{i=1}^k\mu(\phi_i)\Bigr]
=\gamma^{\,k-1}\sum_{T\in\T_k}\int_{(\R^d)^k}
I_T(z_1,\dots,z_k)\prod_{i=1}^k\phi_i(z_i)\,dz_i ,
\end{equation}
the integral converging absolutely. In particular the identity holds
for all $\phi_i\in\Ccg$, not only nonnegative ones.
Equivalently, as locally finite measures off the diagonals,
\[\bN\left[\mu(dz_1)\cdots\mu(dz_k)\right]
=\gamma^{k-1}\sum_{T\in\T_k}I_T(\vec z)\,d\vec z.\]
\end{proposition}

\begin{proof}
We first take $\phi_i\in\Cc$ nonnegative.
Since $\mu(\phi_i)=\int_0^\infty X_{t}(\phi_i)\,dt$ and everything is
nonnegative, Tonelli's theorem gives
\[
\bN\Bigl[\prod_{i=1}^k\mu(\phi_i)\Bigr]
=\int_{(0,\infty)^k}\bN\Bigl[\prod_{i=1}^kX_{t_i}(\phi_i)\Bigr]
\,dt_1\cdots dt_k ,
\]
and we may substitute \eqref{eq:timemoments} and integrate in any order
(by Lemma~\ref{lem:envelope}, the final expression is finite, which justifies
the use of Tonelli throughout). Perform the time integrations as follows.
Each leaf time $t_i$ appears in exactly one heat kernel, namely along the
edge into leaf $i$ from its parent $a$; integrating,
\[
\int_0^\infty p_{\,t_i-s_a}(z_i-y_a)\,dt_i
=\int_0^\infty p_u(z_i-y_a)\,du=G(z_i-y_a),
\]
using $p_u\equiv0$ for $u\le0$. Next, process the internal vertices in order
of decreasing distance from the root: when $v$ is processed, the kernels
along its two child edges have already been integrated, so $s_v$ appears
only in the kernel of the edge from its parent $a$, and
\[
\int_{0}^\infty p_{\,s_v-s_a}(y_v-y_a)\,ds_v=G(y_v-y_a).
\]
Finally the root edge contributes $\int_0^\infty
p_{s}(y)\,ds=G(y)$ with the root pinned at $(0,0)$. After all time
integrations, every edge of $T$ carries a factor $G$ of the spatial
displacement and the internal spatial variables remain, producing exactly
$I_T(z_1,\dots,z_k)$ upon integrating $y_v$, $v\in V_{\mathrm{int}}(T)$.
Summing over $T\in\T_k$ and integrating against $\prod\phi_i$ gives
\eqref{eq:occmoments}; finiteness is Lemma~\ref{lem:envelope}. The
measure-level statement follows since $\Cc$ is measure-determining (see, e.g., \cite{Kall}).
Finally, for general $\phi_i\in\Ccg$ write
$\phi_i=\phi_i^+-\phi_i^-$ with $\phi_i^\pm\in\Cc$; the domination bound
\eqref{eq:envelope} applied to $|\phi_i|$ gives absolute convergence of every
term, so Fubini and multilinearity extend the identity from $\Cc$ to
$\Ccg$.
\end{proof}

\subsubsection{Exponential-moment growth}

\begin{corollary}[Exponential-moment bound]\label{cor:carleman}
Let $k\ge1$ and $\phi_1,\dots,\phi_k\in\Ccg$ be supported in $\overline B_{R_0}$ with
$\|\phi_i\|_\infty\le M_0$. Then
\begin{equation}\label{eq:momentbound}
\bN\Bigl[\prod_{i=1}^k|\mu(\phi_i)|\Bigr]
\;\le\;\gamma^{k-1}\,(2k-3)!!\,\bigl(C_\star M_0\bigr)^{k}
\;\le\;A^{\,k}\,k!\,,
\end{equation}
for $A=A(R_0,M_0)$. Consequently, for every
$\varphi\in\Ccg$ supported in $\overline B_{R_0}$ with $\|\varphi\|_\infty\le M_0$, the
moment generating function of $\mu(\varphi)$ under the size-biased measure
$\Nh=\mu(\varphi_0)\,\bN$ is finite near the origin:
\begin{equation}\label{eq:mgf}
\Nh\bigl[e^{\theta\,\mu(\varphi)}\bigr]
=\sum_{k\ge0}\frac{\theta^k}{k!}\,\bN\bigl[\mu(\varphi_0)\,\mu(\varphi)^k\bigr]
<\infty\qquad\text{for } |\theta|<1/A',
\end{equation}
where $A'=A'(R_0,M_0,\varphi_0)$. The same holds jointly: for
$\varphi^1,\dots,\varphi^j\in\Ccg$ supported in $\overline B_{R_0}$ with
$\|\varphi^i\|_\infty\le M_0$, the joint MGF
$\Nh[\exp(\sum_i\theta_i\mu(\varphi^i))]$ is finite for $|\theta|$ small
enough, by H\"older's inequality.
\end{corollary}

\begin{proof}
Each summand on the right-hand side of \eqref{eq:occmoments} is finite and bounded uniformly in $T$ by the second display of Lemma~\ref{lem:envelope}: applied to $|\phi_i|/M_0\le\mathbf 1_{\overline B_{R_0}}$ it gives $\int_{(\R^d)^k}I_T(\vec z)\prod_i|\phi_i(z_i)|\,d\vec z\le(C_\star M_0)^k$. The first bound of \eqref{eq:momentbound} combines this with the tree count \eqref{eq:treecount} and $(2k-3)!!\le2^kk!$. For \eqref{eq:mgf}, set
$a_k:=\bN[\mu(\varphi_0)\,\mu(\varphi)^k]$; by \eqref{eq:momentbound} (applied
to the $k+1$ functions $\varphi_0,\varphi,\dots,\varphi$, all supported in a
common ball and bounded by $\max(\|\varphi_0\|_\infty,M_0)$) we have
$a_k\le (A')^{k+1}(k+1)!$ for a suitable $A'$, so
$\sum_k\theta^k a_k/k!\le A'\sum_k((k+1)(A'\theta)^k)<\infty$ for
$|\theta|<1/A'$. The joint statement follows from the one-dimensional one by
H\"older: $\Nh[\prod_i e^{\theta_i\mu(\varphi^i)}]
\le\prod_i\Nh[e^{j\theta_i\mu(\varphi^i)}]^{1/j}$, finite when each
$|j\theta_i|<1/A'$.
\end{proof}

\begin{remark}
The bound \eqref{eq:momentbound} says that $\mu(\phi)$ has, after
size-biasing, a finite exponential moment under $\bN$ --- in sharp contrast
with the total mass $\mu(\R^d)$ of Proposition~\ref{prop:totalmass} (in
Section~\ref{sec:sbmstructure} below), which has
none. It is the compact support of $\phi$ that produces the gain: large values
of $\mu(\phi)$ require the SBM to linger near $\supp\phi$, an entropic cost,
whereas large total mass is achieved by spreading out.
\end{remark}

\subsection{Percolation input and convergence of moments}
\label{sec:percolation}

\subsubsection{The hypotheses: history and known inputs}\label{sec:assumptions}

Let $d>6$ and consider critical bond percolation on $\Z^d$. We review the
history of the hypotheses \eqref{eq:2pt} and \eqref{eq:guiding-kp} and record
two further known theorems.

The uniform two-point function bound \eqref{eq:2pt} is known to hold in the mean-field regime:
for the nearest-neighbour model the $x$-space two-point asymptotics are due to
Hara \cite{Har08} in sufficiently high dimensions, with the threshold lowered
to $d\ge11$ by Fitzner and van der Hofstad \cite{FH17}; for sufficiently
spread-out models in all $d>6$ they are due to Hara, van der Hofstad and Slade \cite{HHS03}. Based on the two-point asymptotics, Kozma
and Nachmias \cite{KN11} established the one-arm bound of
Proposition~\ref{prop:KN}, originally for $d\ge19$, with the dimensional
requirement reduced to $d\ge11$ by \cite{FH17}. More recently, Liu and Slade
\cite{LS26,LS26spread} considerably simplified the proofs of the two-point
asymptotics for both the nearest-neighbour and spread-out models, while also
obtaining \eqref{taubound} with a refined error term; their approach is based
on weak derivatives and Fourier analysis on $L^p$-spaces. Only the upper bound is used quantitatively
below; the lower bound enters through Proposition~\ref{prop:KN}.

The explicit form of the constants $\mathfrak a$ and $\lambda$ of \eqref{eq:guiding-kp}, how they are pinned down by the low-order cases, and the dictionary to the normalization of \cite{CCHS26} are collected in Appendix~\ref{app:dictionary}.

\begin{proposition}[Tree-graph inequality; Aizenman and Newman
\cite{AN84}]\label{prop:AN}
For all $m\ge1$ and \emph{all} $x_1,\dots,x_m\in\Z^d$ (not necessarily
distinct, and not necessarily distinct from $0$),
\begin{equation}\label{eq:AN}
\tau_{m+1}(0,x_1,\dots,x_m)\;\le\;
\sum_{T\in\T_m}\ \sum_{(z_v)_{v\in V_{\mathrm{int}}(T)}\in(\Z^d)^{V_{\mathrm{int}}(T)}}
\ \prod_{\{a,b\}\in E(T)}\tau_2(w_a,w_b),
\end{equation}
where $w_0:=0$, $w_i:=x_i$ for leaves and $w_v:=z_v$ for internal vertices.
The bound holds verbatim when some of the points coincide: with the
convention that a repeated point is not repeated in the connection event,
$\tau_{m+1}(0,x_1,\dots,x_m)$ is then a connection function of the distinct
points among $0,x_1,\dots,x_m$, which is still dominated by the displayed
sum (the right-hand side only grows when leaves are placed at coinciding
positions).
\end{proposition}

\begin{proposition}[One-arm bound; Kozma and Nachmias \cite{KN11}]
\label{prop:KN}
Under \eqref{eq:2pt} there are $0<K_2'\le K_2<\infty$ such
that
\[
K_2'\,r^{-2}\;\le\;\bP\bigl(0\leftrightarrow\partial B_r\bigr)\;\le\;K_2\,r^{-2},
\qquad r\ge1 .
\]
A streamlined derivation of these one-arm estimates is given by Asselah, Schapira and Sousi \cite{ASS}.
\end{proposition}

From this point until the end of Section~\ref{sec:onearm} we assume $d>6$;
the hypotheses \eqref{eq:2pt} and/or \eqref{eq:guiding-kp} in force are
indicated in each statement.

\subsubsection{Exact lattice-to-continuum rewriting}

Recall the definitions \eqref{eq:empirical} of $\mu_n$ and $\nu_n$, and recall the branching parameter $\gamma$ of the super-Brownian motion (Section~\ref{sec:SBM}), identified in Section~\ref{sec:guiding} as
\begin{equation}\label{eq:gammadef}
\gamma=\lambda/\mathfrak a .
\end{equation}
For $m\ge1$ and $\phi_1,\dots,\phi_m\in\Cc$ we have, by definition,
\begin{equation}\label{eq:momid}
\int_{\cM}\prod_{i=1}^m\mu(\phi_i)\,d\nu_n
=n^2\,\E\Bigl[\prod_{i=1}^m\mu_n(\phi_i)\Bigr]
=\frac{n^{2}}{\mathfrak a^m n^{4m}}
\sum_{x_1,\dots,x_m\in\Z^d}\prod_{i=1}^m\phi_i(x_i/n)\;
\tau_{m+1}(0,x_1,\dots,x_m),
\end{equation}
where coinciding arguments among $0,x_1,\dots,x_m$ are allowed in the sum. Writing each lattice sum as an integral over unit cells,
$\sum_{x\in\Z^d}g(x)=n^{d}\int_{\R^d}g(\fl{ny})\,dy$, identity
\eqref{eq:momid} becomes \emph{exactly}
\begin{equation}\label{eq:Fn}
\int\prod_{i=1}^m\mu(\phi_i)\,d\nu_n=\int_{(\R^d)^m}F_n^{(m)}(\vec y)\,d\vec y,
\end{equation}
where
\begin{equation}\label{eq:Fndef}
F^{(m)}_n(\vec y):=\mathfrak a^{-m}\,n^{(d-4)m+2}\,
\tau_{m+1}\bigl(0,\fl{ny_1},\dots,\fl{ny_m}\bigr)
\prod_{i=1}^m\phi_i\Bigl(\tfrac{\fl{ny_i}}{n}\Bigr).
\end{equation}
The power $n^{(d-4)m+2}$ collects the three scalings in \eqref{eq:momid}: a factor $n^2$ from $\nu_n=n^2\bP(\mu_n\in\cdot)$; a factor $n^{-4m}$ from the $m$ occupation factors, each carrying $(\mathfrak a n^4)^{-1}$ through $\mu_n(\R^d)=|\mathcal C|/(\mathfrak a n^4)$; and a factor $n^{dm}$ from writing the $m$ lattice sums as cell integrals. Together, $n^{2-4m+dm}=n^{(d-4)m+2}$.
The rewriting is exact; the convergence of the moments $\int\prod_i\mu(\phi_i)\,d\nu_n$ to $\bN[\prod_i\mu(\phi_i)]$ (Proposition~\ref{prop:momconv}) will be proved by dominated
convergence applied directly to \eqref{eq:Fn}--\eqref{eq:Fndef}, with no
Riemann-sum error term.

\begin{lemma}[Pointwise domination]\label{lem:domination}
Assume \eqref{eq:2pt} and let
$\phi_1,\dots,\phi_m$ be supported in $\overline B_{R_0}$ with
$\|\phi_i\|_\infty\le M_0$. There is $K_3$ such that for all
$n\ge1$ and all $\vec y\in(\R^d)^m$,
\begin{equation}\label{eq:dombound}
F^{(m)}_n(\vec y)\;\le\;\bigl(K_3/\mathfrak a\bigr)^{m}K_3^{m-1}\,M_0^m
\sum_{T\in\T_m} D_T(\vec y)\prod_{i=1}^m
\mathbf 1_{\overline B_{R_0+\sqrt d}}(y_i),
\end{equation}
where
\[
D_T(\vec y):=\int_{(\R^d)^{V_{\mathrm{int}}(T)}}
\prod_{\{a,b\}\in E(T)}\bar G(u_a-u_b)
\prod_{v}du_v ,\qquad u_0=0,\ u_i=y_i .
\]
Moreover $\int_{(\R^d)^m}\sum_{T}D_T(\vec y)\prod_i
\mathbf 1_{\overline B_{R_0+\sqrt d}}(y_i)\,d\vec y<\infty$.
\end{lemma}

\begin{remark}[Discretization of the two-point function]
\label{rem:discretize}
The proof below rests on a single deterministic comparison between the
discrete kernel and its continuum envelope, which we isolate here. For all
$u,u'\in\R^d$ and $n\ge1$, with $\bar G(x)=|x|^{-(d-2)}$,
\begin{equation}\label{eq:discretize}
\bigl(1\vee|\fl{nu}-\fl{nu'}|\bigr)^{-(d-2)}
\;\le\; c(d)\,n^{-(d-2)}\,\bar G(u-u'),
\qquad c(d):=(1+\sqrt d)^{\,d-2}.
\end{equation}
There are two cases. If $\fl{nu}\neq\fl{nu'}$, then $|\fl{nu}-\fl{nu'}|\ge1$,
so the left-hand side equals $|\fl{nu}-\fl{nu'}|^{-(d-2)}$; since
$|u-u'|\le|\fl{nu}-\fl{nu'}|/n+\sqrt d/n\le(1+\sqrt d)|\fl{nu}-\fl{nu'}|/n$, we
get $n^{-(d-2)}\bar G(u-u')=\bigl(n|u-u'|\bigr)^{-(d-2)}\ge(1+\sqrt
d)^{-(d-2)}|\fl{nu}-\fl{nu'}|^{-(d-2)}$, which is \eqref{eq:discretize} with
constant $(1+\sqrt d)^{d-2}$. If $\fl{nu}=\fl{nu'}$, the left-hand side is $1$,
while $|u-u'|\le\sqrt d/n$ forces
$n^{-(d-2)}\bar G(u-u')=\bigl(n|u-u'|\bigr)^{-(d-2)}\ge d^{-(d-2)/2}$, so
\eqref{eq:discretize} holds with $c(d)=d^{(d-2)/2}$. As $(1+\sqrt d)^{d-2}\ge
d^{(d-2)/2}$, the constant $c(d)=(1+\sqrt d)^{d-2}$ covers both cases.
Consequently, under the two-point bound \eqref{eq:2pt}, the rounded
two-point function obeys
\begin{equation}\label{eq:tauenvelope}
\tau_2\bigl(\fl{nu},\fl{nu'}\bigr)
\le K_1\bigl(1\vee|\fl{nu}-\fl{nu'}|\bigr)^{-(d-2)}
\le K_1\,c(d)\,n^{-(d-2)}\,\bar G(u-u') .
\end{equation}
\end{remark}

\begin{proof}
Apply the tree-graph inequality \eqref{eq:AN} to
$\tau_{m+1}(0,\fl{ny_1},\dots,\fl{ny_m})$ (valid for every $\vec y$,
including configurations where some $\fl{ny_i}$ coincide or vanish, by
Proposition~\ref{prop:AN}) and rewrite each internal lattice
sum as an integral over unit cells, $\sum_{z\in\Z^d}=n^d\int_{\R^d}
(\cdot)(\fl{nw})\,dw$. The resulting expression is an integral over
$(w_v)_{v\in V_{\mathrm{int}}}$ of a product over the $2m-1$ edges of factors
$\tau_2(\fl{nu_a},\fl{nu_b})$ evaluated at the rounded points, times
$n^{d(m-1)}$. By Remark~\ref{rem:discretize}
(equation~\eqref{eq:tauenvelope}), each edge factor is at most
$K_1c(d)\,n^{-(d-2)}\bar G(u_a-u_b)$. Collecting the powers of $n$,
three sources contribute: the prefactor $n^{(d-4)m+2}$ from the definition of
$F^{(m)}_n$ in \eqref{eq:Fn}; a factor $n^{d(m-1)}$ from rewriting the
$|V_{\mathrm{int}}(T)|=m-1$ internal lattice sums as cell integrals
($\sum_{z\in\Z^d}=n^d\int(\cdot)(\fl{nw})\,dw$, one factor $n^d$ per internal
vertex); and a factor $n^{-(d-2)}$ from each of the $|E(T)|=2m-1$ edges, i.e.\
$n^{-(d-2)(2m-1)}$ in total. Their product is
\[
n^{(d-4)m+2}\cdot n^{d(m-1)}\cdot n^{-(d-2)(2m-1)}=n^{0}=1
\qquad\text{for every }m\ge1
\]
(the exponent vanishes identically in $m$; this is the algebraic
identity underlying the choice of normalization \eqref{eq:empirical}).
Finally, $\phi_i(\fl{ny_i}/n)\le M_0\mathbf 1\{|y_i|\le R_0+\sqrt d\}$
(if $|y_i|>R_0+\sqrt d$ then $|\fl{ny_i}/n|\ge|y_i|-\sqrt d/n>R_0$ for all $n\ge1$, off the support of $\phi_i$). This proves
\eqref{eq:dombound}. The integrability of the dominating function is
Lemma~\ref{lem:envelope} with support radius $R_0+\sqrt d$, together with the
tree count \eqref{eq:treecount}.

\end{proof}

\subsubsection{Convergence of moments}

\begin{proposition}[Moment convergence]\label{prop:momconv}
 Assume \eqref{eq:2pt} and \eqref{eq:guiding-kp}. Then for every
$m\ge1$ and $\phi_1,\dots,\phi_m\in\Cc$,
\begin{equation}\label{eq:momconv}
\lim_{n\to\infty}\int_\cM\prod_{i=1}^m\mu(\phi_i)\,d\nu_n
=\gamma^{\,m-1}\sum_{T\in\T_m}\int_{(\R^d)^m}I_T(\vec y)
\prod_{i=1}^m\phi_i(y_i)\,dy_i
=\bN\Bigl[\prod_{i=1}^m\mu(\phi_i)\Bigr],
\end{equation}
with $\gamma$ as in \eqref{eq:gammadef}. Moreover, if $R_0\ge1$ and
$M_0<\infty$ are such that each $\phi_i$ is supported in $\overline B_{R_0}$ with
$\|\phi_i\|_\infty\le M_0$, then
\begin{equation}\label{eq:uniformmoments}
\sup_{n\ge1}\;\int\prod_{i=1}^m\mu(\phi_i)\,d\nu_n
\;\le\;A_1^{\,m}\,m!
\end{equation}
for a constant $A_1=A_1(R_0,M_0)$.
\end{proposition}

\begin{proof}
By \eqref{eq:Fn} the left side of \eqref{eq:momconv} equals
$\int F^{(m)}_n$, and by Lemma~\ref{lem:domination} the integrand
$F^{(m)}_n$ is dominated, \emph{for every} $\vec y$, by the fixed integrable
function on the right of \eqref{eq:dombound}. It therefore suffices to
identify the a.e.\ pointwise limit. For Lebesgue-a.e.\ $\vec y$ the
coordinates $y_1,\dots,y_m$ are pairwise distinct and nonzero --- the diagonal
$\{y_i=y_j \text{ for some } i\neq j\}\cup\{y_i=0\}$ being a Lebesgue-null
set, which dominated convergence ignores. At such $\vec y$ the rounded points
$\fl{ny_1},\dots,\fl{ny_m}$ are eventually distinct and nonzero, so
\eqref{eq:guiding-kp} applies; since each $\phi_i$ is continuous,
$\phi_i(\fl{ny_i}/n)\to\phi_i(y_i)$, and
$\mathfrak a^{-m}n^{(d-4)m+2}\tau_{m+1}(0,\fl{n\vec y})\to
\mathfrak a^{1-m}\lambda^{m-1}\sum_TI_T(\vec y)
=\gamma^{m-1}\sum_TI_T(\vec y)$. Hence, a.e., 
\[F^{(m)}_n(\vec y)\to\gamma^{m-1}\sum_TI_T(\vec y)\prod_i\phi_i(y_i).\]
Dominated convergence yields the first
equality in \eqref{eq:momconv}; the second is
Proposition~\ref{prop:occmoments}. The uniform bound
\eqref{eq:uniformmoments} follows from \eqref{eq:Fn}, \eqref{eq:dombound},
Lemma~\ref{lem:envelope} and \eqref{eq:treecount}.
\end{proof}

We will also need one moment estimate with a test function of growing
support.

\begin{lemma}\label{lem:growingsupport}
Assume \eqref{eq:2pt} and fix $\varphi_0\in\Cc$ supported in
$\overline B_{R_0}$, $0\le\varphi_0\le1$. There is a constant $C=C(R_0)$ such that for all $L\ge R_0\vee 1$, all
$0\le\psi\le\mathbf 1_{\overline B_{2L}}$ and all $n\ge1$,
\[
\int_\cM \mu(\varphi_0)\,\mu(\psi)\,d\nu_n\;\le\;C\,L^2 .
\]
\end{lemma}

\begin{proof}
By \eqref{eq:Fn} and Lemma~\ref{lem:domination} with $m=2$ (using
$\psi\le\mathbf 1_{\overline B_{2L}}$ in place of a compactly supported continuous
function, which is all the proof of Lemma~\ref{lem:domination} uses),
\[
\int\mu(\varphi_0)\mu(\psi)\,d\nu_n
\le C\int_{(\R^d)^2}\mathbf 1_{\overline B_{R_0+\sqrt d}}(y_1)\,
\mathbf 1_{\overline B_{2L+\sqrt d}}(y_2)
\int_{\R^d}\bar G(u)\bar G(y_1-u)\bar G(y_2-u)\,du\;dy_1\,dy_2 .
\]
Indeed, the sum over $\T_2$ consists of the single three-star: for $m=2$
the leaf set is $\{0,1,2\}$ and there is exactly one binary tree, namely the
one internal vertex $u$ joined to the three leaves $0,y_1,y_2$, so
$|\T_2|=(2\cdot2-3)!!=1$ and $D_T(\vec y)=\int_{\R^d}\bar G(u)\bar
G(y_1-u)\bar G(y_2-u)\,du$. Integrate $y_2$
first: $\sup_{u}\int_{B_{2L+\sqrt d}}\bar G(y_2-u)\,dy_2
\le\int_{B_{4L+2\sqrt d}}|v|^{-(d-2)}dv\le cL^2$
(the integral is largest when the singularity $u$ lies in the ball of integration, in which case $y_2-u$ ranges within $B_{4L+2\sqrt d}$). Then integrate $y_1$:
$\int_{B_{R_0+\sqrt d}}\bar G(y_1-u)\,dy_1\le c_1m(u)$ by Claim~1 of
Lemma~\ref{lem:envelope}. Finally $\int\bar G(u)\,m(u)\,du<\infty$ by
splitting into two regions: on $|u|\le1$, $\bar G(u)\,m(u)\le|u|^{-(d-2)}$
which is integrable since $d>2$ (by \eqref{eq:polar}); on $|u|\ge1$,
$\bar G(u)\,m(u)\le|u|^{-(d-2)}(1+|u|)^{-(d-2)}\le |u|^{-2(d-2)}$ which
is integrable since $2(d-2)>d$ for $d>4$.
\end{proof}

\subsection{Moment determinacy and convergence of the size-biased measures}
\label{sec:biased}

Throughout this section fix $\varphi_0\in\Cc$, $\varphi_0\not\equiv0$, and recall from \eqref{eq:sizebias} the finite measures $\nuh_n=\mu(\varphi_0)\,\nu_n$ and $\Nh=\mu(\varphi_0)\,\bN$ on $\cM$.
By Proposition~\ref{prop:momconv} (with one test function equal to
$\varphi_0$), all moments of $\nuh_n$ converge to the corresponding moments
of $\Nh$, and the total masses converge:
\begin{equation}\label{eq:hatmoments}
\int_\cM\prod_{i=1}^m\mu(\phi_i)\,d\nuh_n
\;\xrightarrow[n\to\infty]{}\;
\Nh\Bigl[\prod_{i=1}^m\mu(\phi_i)\Bigr]<\infty
\qquad(m\ge0,\ \phi_i\in\Cc),
\end{equation}
the case $m=0$ being
$\nuh_n(\cM)=\int\mu(\varphi_0)\,d\nu_n\to\bN[\mu(\varphi_0)]
=\int G\varphi_0\in(0,\infty)$, the last equality being the first-moment
identity, i.e.\ the case $k=1$ of Proposition~\ref{prop:occmoments}.
(The integral $\int G\varphi_0$ is finite because $\varphi_0$ has compact
support and $G$ is locally integrable by \eqref{eq:polar}, and positive
because $G>0$ and $\varphi_0\not\equiv0$.)

\subsubsection{Moment determinacy on $\cM$}

The determinacy input we use is the classical analytic criterion: a
finite measure on $\R^j$ whose moment generating function is finite in a
neighbourhood of the origin is determined by its moments. We record it in the
form we need.

\begin{lemma}[Multivariate determinacy via the MGF]
\label{lem:petersen}
Let $\Lambda,\Lambda'$ be finite Borel measures on $\R^j$ with equal,
finite mixed moments of all orders. Suppose there is $\delta>0$ such that
$\int e^{\delta|x|}\,d\Lambda<\infty$. Then $\Lambda=\Lambda'$.
\end{lemma}

\begin{proof}
By the Cram\'er--Wold device it is enough to match all one--dimensional
projections, after which only the classical \emph{single}--variable analytic
criterion is used; the exponential moment enters solely to make the relevant
characteristic function holomorphic on a strip around the real axis.

\emph{Step 1: reduction to one dimension.} For $\theta\in\R^j$ let
$p_\theta$ and $p_\theta'$ denote the laws of the linear functional
$x\mapsto\langle\theta,x\rangle$ under $\Lambda$ and $\Lambda'$; these are
finite measures on $\R$ of equal total mass $m_0:=\Lambda(\R^j)=\Lambda'(\R^j)$.
A finite Borel measure on $\R^j$ is determined, among finite measures, by the
family of all such one--dimensional projections (Cram\'er--Wold,
\cite[Thm.~29.4]{Bil95}), so it suffices to show $p_\theta=p_\theta'$ for each
fixed $\theta$. Since $\int_\R t^k\,dp_\theta=\int_{\R^j}\langle\theta,x\rangle^k
\,d\Lambda$ expands by the multinomial theorem into the mixed moments of
$\Lambda$ (and likewise for $p_\theta'$), the equality of all mixed moments
gives $p_\theta$ and $p_\theta'$ a common, finite moment sequence
$m_k:=\int_\R t^k\,dp_\theta=\int_\R t^k\,dp_\theta'$, $k\ge0$.

\emph{Step 2: a finite exponential moment for each projection.} Put
$\rho:=\delta/(1+|\theta|)$. Since $|\langle\theta,x\rangle|\le|\theta|\,|x|$ and
$\rho|\theta|<\delta$,
\[
C:=\int_\R e^{\rho|t|}\,dp_\theta=\int_{\R^j}e^{\rho|\langle\theta,x\rangle|}
\,d\Lambda\le\int_{\R^j}e^{\delta|x|}\,d\Lambda<\infty .
\]
For $p_\theta'$ the exponential moment is \emph{not} assumed; it is forced
by the equality of moments, which is the point worth detailing. By
Cauchy--Schwarz and equality of the even moments,
$\int_\R|t|^k\,dp_\theta'\le(\int t^{2k}\,dp_\theta')^{1/2}m_0^{1/2}
=m_{2k}^{1/2}m_0^{1/2}$; and from $t^{2k}\le(2k)!\,\rho^{-2k}e^{\rho|t|}$ we get
$m_{2k}\le(2k)!\,\rho^{-2k}C$. Hence, with $\delta':=\rho/4$ and using
$\sqrt{(2k)!}\le 2^k k!$,
\[
\int_\R e^{\delta'|t|}\,dp_\theta'
=\sum_{k\ge0}\frac{\delta'^{\,k}}{k!}\int_\R|t|^k\,dp_\theta'
\le m_0^{1/2}C^{1/2}\sum_{k\ge0}\frac{\sqrt{(2k)!}}{k!}\Bigl(\frac{\delta'}{\rho}
\Bigr)^{k}
\le m_0^{1/2}C^{1/2}\sum_{k\ge0}2^{-k}<\infty ,
\]
and \emph{a fortiori} $\int_\R e^{\delta'|t|}\,dp_\theta\le C<\infty$ (as
$\delta'<\rho$). Thus both projections have a finite exponential moment at rate
$\delta'$.

\emph{Step 3: the single--variable analytic criterion, and the role of the
strip.} Let $p$ be either $p_\theta$ or $p_\theta'$, so $\int_\R
e^{\delta'|t|}\,dp<\infty$ by Step 2. For a complex number $z=x+iy$ with
$|y|<\delta'$ one has $|e^{izt}|=e^{-yt}\le e^{\delta'|t|}$, so the
characteristic function
\[
\widehat p(z):=\int_\R e^{izt}\,dp(t)\qquad\text{converges absolutely on}\qquad
\{\,z=x+iy:\ |y|<\delta'\,\}
\]
and is holomorphic there (differentiation under the integral is justified
by the same domination), with $\widehat p^{(k)}(0)=i^k m_k$. This is exactly
what the exponential moment buys: it upgrades $\widehat p$ from a merely
continuous function on the line $\R$ to a holomorphic function on a genuine
two--dimensional neighbourhood of $\R$, where its Taylor coefficients at $0$ ---
the moments --- determine it everywhere. Consequently $\widehat{p_\theta}$ and
$\widehat{p_\theta'}$ are holomorphic on the same strip with identical
derivatives of all orders at $0$; by the identity theorem they agree on the
strip, hence on $\R$, and a finite measure on $\R$ is determined by its
characteristic function. Therefore $p_\theta=p_\theta'$.

As $\theta\in\R^j$ was arbitrary, the Cram\'er--Wold device gives
$\Lambda=\Lambda'$. The single--variable analytic characteristic--function
criterion is classical (Billingsley \cite[Chapter~30]{Bil95}); cf.\ Petersen \cite{Pet82} for a related reduction of the multidimensional moment problem to its coordinate marginals.
\end{proof}

\begin{lemma}[Determinacy on $\cM$]\label{lem:determinacy}
Let $\Lambda,\Lambda'$ be finite Borel measures on $\cM$ such that for all
$m\ge0$ and $\phi_1,\dots,\phi_m\in\Cc$,
\[
\int\prod_{i=1}^m\mu(\phi_i)\,d\Lambda
=\int\prod_{i=1}^m\mu(\phi_i)\,d\Lambda'<\infty,
\]
and such that for every $\phi\in\Cc$ supported in a ball $B_{R}$ the
function $\mu\mapsto\mu(\phi)$ has a finite exponential moment under
$\Lambda$, i.e.\ $\int e^{\delta\mu(\phi)}\,d\Lambda<\infty$ for some
$\delta=\delta(\phi)>0$. Then $\Lambda=\Lambda'$.
\end{lemma}

\begin{proof}
\emph{Step 1: finite-dimensional pushforwards agree.}
Fix $\phi^1,\dots,\phi^j\in\Cc$ (nonnegative) and let
$\Pi:\cM\to\R_+^j$, $\Pi(\mu):=(\mu(\phi^1),\dots,\mu(\phi^j))$. The
pushforwards $\Pi_*\Lambda$ and $\Pi_*\Lambda'$ are finite measures on $\R^j$
with equal mixed moments of all orders (these are exactly the mixed moments
$\int\prod_i\mu(\phi^i)^{a_i}d\Lambda$ in the hypothesis). Moreover, set
$\phi:=\phi^1+\dots+\phi^j\in\Cc$. For $\mu\in\cM$ the point
$x:=\Pi(\mu)=(\mu(\phi^1),\dots,\mu(\phi^j))$ lies in the nonnegative orthant
$\R_+^j$, so its Euclidean norm is dominated by its $\ell^1$ norm:
$|x|\le x_1+\dots+x_j=\mu(\phi^1)+\dots+\mu(\phi^j)=\mu(\phi)$. Consequently
$\int_{\R^j}e^{\delta|x|}\,d(\Pi_*\Lambda)(x)\le\int_\cM e^{\delta\mu(\phi)}
\,d\Lambda<\infty$ for some $\delta>0$, by the exponential--moment hypothesis
applied to $\phi$. By Lemma~\ref{lem:petersen}, $\Pi_*\Lambda=\Pi_*\Lambda'$.

\emph{Step 2: from cylinders to the Borel $\sigma$-algebra.}
Let $\mathcal C$ be the collection of cylinder sets
$\{\mu:\Pi(\mu)\in B\}$ over finite subfamilies $\phi^1,\dots,\phi^j\in\Cc$
and Borel $B\subset\R^j$. By Step 1, $\Lambda$ and $\Lambda'$ agree on
$\mathcal C$. The class $\mathcal C$ is a $\pi$-system (an intersection of two
cylinders is again a cylinder, over the union of the two index sets) and
generates the $\sigma$-algebra
$\sigma(\mu\mapsto\mu(\phi):\phi\in\Cc)$. By a standard fact for spaces of
measures, this evaluation $\sigma$-algebra is exactly the Borel
$\sigma$-algebra of the weak topology on $\cM$
(\cite[Lemma~4.7]{Kal17}): the Borel
$\sigma$-algebra is generated by $\{\mu\mapsto\mu(f):f\in C_b\}$, and the
maps $\mu\mapsto\mu(\phi)$, $\phi\in\Cc$, generate the same $\sigma$-algebra:
for $f\in C_b(\R^d)$ one has $\mu(f)=\lim_K\mu(f\cdot\chi_K)$ pointwise,
where $f\cdot\chi_K$ is the product of $f$ with a cutoff
$\chi_K\in C_c(\R^d)$ satisfying $\mathbf 1_{\overline B_K}\le\chi_K\le\mathbf 1_{\overline B_{K+1}}$
(so $f\cdot\chi_K\in C_c(\R^d)$). Indeed $f\cdot\chi_K\to f$ pointwise and
$|f\cdot\chi_K|\le\|f\|_\infty$, so $\mu(f\cdot\chi_K)\to\mu(f)$ by dominated
convergence, the constant $\|f\|_\infty$ being $\mu$--integrable precisely
because $\mu$ is a \emph{finite} measure ($\mu(\R^d)<\infty$); this finiteness
is all that is used. Hence $\mu\mapsto\mu(f)$ is a countable pointwise limit of
$\sigma(\Cc)$-measurable maps, and the two evaluation $\sigma$-algebras
coincide.

We conclude by Dynkin's $\pi$--$\lambda$ theorem. The $\pi$-system is
$\mathcal C$, and $\mathcal L:=\{A\in\mathcal B(\cM):\Lambda(A)=\Lambda'(A)\}$
is a $\lambda$-system: it contains the whole space $\cM$ (since
$\Lambda(\cM)=\Lambda'(\cM)<\infty$, the $m=0$ case of the hypothesis) and is
closed under proper differences and increasing unions, both because $\Lambda$
and $\Lambda'$ are finite. By Step~1, $\mathcal C\subseteq\mathcal L$; since
$\mathcal C$ generates $\mathcal B(\cM)$, the theorem gives
$\mathcal L=\mathcal B(\cM)$, i.e.\ $\Lambda=\Lambda'$.
\end{proof}

\subsubsection{Tightness}

\begin{lemma}[Tightness]\label{lem:tightness}
 Assume \eqref{eq:2pt}. Then the family
$\{\nuh_n\}_{n\ge1}$ is tight in $\cMM$: the total masses are bounded, and
for every $\eta>0$ there is a compact $\mathcal K\subset\cM$ with
$\sup_n\nuh_n(\cM\setminus\mathcal K)\le\eta$.
\end{lemma}

\begin{proof}
Boundedness of the masses is contained in \eqref{eq:hatmoments}
($m=0$).

\emph{Step 1 (far-field control).} Fix $K\ge1$ and recall $\supp\mu_n=n^{-1}\mathcal C$, so that
$\{\supp\mu\not\subset\overline B_K\}$ is the event $\{\mathcal C\not\subset\overline B_{Kn}\}$. By
Cauchy--Schwarz on the (positive, non-normalized) measure $\nu_n$,
\begin{equation}\label{eq:CSarm}
\begin{aligned}
\nuh_n\bigl(\{\mu:\ \supp\mu\not\subset\overline B_K\}\bigr)
&=\int\mu(\varphi_0)\,\mathbf 1\{\mathcal C\not\subset\overline B_{Kn}\}\,d\nu_n\\
&\le\Bigl(\int\mu(\varphi_0)^2\,d\nu_n\Bigr)^{\!1/2}
\Bigl(n^2\,\bP\bigl(0\leftrightarrow\partial B_{Kn}\bigr)\Bigr)^{\!1/2}
\;\le\;\frac{C_0}{K},
\end{aligned}
\end{equation}
uniformly in $n$, with $C_0:=A_1\sqrt{2K_2}$. The first factor is bounded by
the case $m=2$ of \eqref{eq:uniformmoments} with both test functions equal to
$\varphi_0$, namely $\int\mu(\varphi_0)^2\,d\nu_n\le A_1^2\,2!$; this is the role
of the second moment. The second factor uses the one-arm bound
(Proposition~\ref{prop:KN}), applicable since $Kn\ge1$: it is at most
$n^2K_2(Kn)^{-2}=K_2K^{-2}$, which yields exactly the bound $C_0/K$ in
\eqref{eq:CSarm} because both factors enter through their square roots.

\emph{Step 2 (total-mass control).} Next, for $L\ge R_0\vee1$ take $\chi_L\in C_c(\R^d)$ with
$\mathbf 1_{\overline B_L}\le\chi_L\le\mathbf 1_{\overline B_{2L}}$. On
$\{\supp\mu\subset\overline B_L\}$ one has $\mu(\R^d)=\mu(\chi_L)$, so for every
$M\ge1$, by Chebyshev and Lemma~\ref{lem:growingsupport} (applicable
because $L\ge R_0\vee1$, with $\psi=\chi_L\le\mathbf 1_{\overline B_{2L}}$),
\begin{align*}
\nuh_n\bigl(\mu(\R^d)>M\bigr)&
\le\nuh_n\bigl(\supp\mu\not\subset\overline B_L\bigr)
+\nuh_n\bigl(\mu(\chi_L)>M\bigr)\\
&\le\frac{C_0}{L}+\frac{1}{M}\int\mu(\varphi_0)\mu(\chi_L)\,d\nu_n\\
&\le\frac{C_0}{L}+\frac{C L^2}{M}.
\end{align*}
Choosing $L=M^{1/3}$ --- legitimate once $M\ge R_0^3$, so that $L\ge
R_0\vee1$ --- gives $\sup_n\nuh_n(\mu(\R^d)>M)\le C_1M^{-1/3}$ for all such
$M$.

\emph{Step 3 (assembling a compact set).} Given $\eta>0$, choose $M$ with $C_1M^{-1/3}\le\eta/2$ (and $M\ge R_0^3$, as
permitted above) and $K_j\uparrow$
with $\sum_jC_0/K_j\le\eta/2$, and set
\[
\mathcal K:=\bigl\{\mu\in\cM:\ \mu(\R^d)\le M,\
\mu\bigl(\R^d\setminus\overline B_{K_j}\bigr)\le 2^{-j}\ \forall j\bigr\}.
\]
Then $\mathcal K$ is relatively compact in the weak topology by Prohorov's
theorem: every $\mu\in\mathcal K$ has total mass $\mu(\R^d)\le M$, and for any
$\eps>0$, choosing $j$ with $2^{-j}<\eps$, every $\mu\in\mathcal K$ satisfies
$\mu(\R^d\setminus\overline B_{K_j})\le2^{-j}<\eps$; thus the compact ball $\overline B_{K_j}$
carries all but $\eps$ of the mass uniformly over $\mathcal K$. Bounded total
mass together with this uniform tightness gives, by Prohorov's theorem,
that $\overline{\mathcal K}$ is weakly compact. Moreover,
\[
\nuh_n(\cM\setminus\overline{\mathcal K})\le
\nuh_n(\mu(\R^d)>M)+\sum_j\nuh_n\bigl(\mu(\R^d\setminus\overline B_{K_j})>2^{-j}\bigr)
\le\frac{\eta}{2}+\sum_j\frac{C_0}{K_j}\le\eta,
\]
where the bound on the $j$-th term uses the inclusion
$\{\mu(\R^d\setminus\overline B_{K_j})>2^{-j}\}\subset\{\supp\mu\not\subset\overline B_{K_j}\}$
(if $\mu$ puts positive mass outside $\overline B_{K_j}$ then its support is not
contained in $\overline B_{K_j}$) together with the far-field bound \eqref{eq:CSarm} at
$K=K_j$, namely $\nuh_n(\supp\mu\not\subset\overline B_{K_j})\le C_0/K_j$.
\end{proof}

\begin{remark}
The far-field control \eqref{eq:CSarm} is the one place where polynomial
moments are powerless: $\int\mu(\varphi_0)\,\mu(\R^d\setminus B_K)\,d\nu_n$
diverges, both for $\nu_n$ and in the limit, because
$\int_{|y|>K}G(y)\,dy=\infty$. The one-arm probability supplies exactly the
missing decay, with the exponent $2$ of \cite{KN11} matching the
$K^{-2}$ exit-probability scaling of the SBM range
(Proposition~\ref{prop:range}, Section~\ref{sec:sbmstructure} below); this
is no accident, as Section~\ref{sec:onearm} makes precise.
\end{remark}

\subsubsection{Convergence of the size-biased measures}

\begin{theorem}\label{thm:biased}
 Assume \eqref{eq:2pt} and \eqref{eq:guiding-kp}. Then
for every $\varphi_0\in\Cc$ with $\varphi_0\not\equiv0$,
\[
\nuh_n\;\xRightarrow[n\to\infty]{}\;\Nh
\qquad\text{weakly in }\cMM .
\]
\end{theorem}

\begin{proof}
By Lemma~\ref{lem:tightness} and Prohorov's theorem (in its abstract form
for finite measures on the Polish space $\cM$), every subsequence of
$(\nuh_n)$ has a further subsequence $(\nuh_{n_k})_{k\ge1}$ converging weakly
to some finite measure $\Lambda$ on $\cM$; it suffices to show $\Lambda=\Nh$
for every such limit point, as then the full sequence converges to $\Nh$.
Fix one such convergent subsequence $(\nuh_{n_k})$, $\nuh_{n_k}
\Rightarrow\Lambda$.

\emph{Step 1: $\Lambda$ has the moments of $\Nh$.}
Fix $m\ge1$ and $\phi_1,\dots,\phi_m\in\Cc$, and write
$P(\mu):=\prod_{i=1}^m\mu(\phi_i)$. This $P$ is unbounded, hence \emph{not}
a legitimate test function for the weak convergence
$\nuh_{n_k}\Rightarrow\Lambda$; we therefore truncate. For $R>0$ let
$P_R(\mu):=\prod_{i=1}^m(\mu(\phi_i)\wedge R)$, which is bounded and
continuous on $\cM$ in its weak topology (each $\mu\mapsto\mu(\phi_i)$ is weakly
continuous since $\phi_i\in C_b(\R^d)$, and $t\mapsto t\wedge R$ is continuous).
Since $\nuh_{n_k}\Rightarrow\Lambda$,
\begin{equation}\label{eq:PRconv}
\int P_R\,d\nuh_{n_k}\xrightarrow[k\to\infty]{}\int P_R\,d\Lambda .
\end{equation}
Moreover, uniformly in $k$ (indeed in $n$),
\[
0\le\int(P-P_R)\,d\nuh_{n_k}
\le\sum_{i=1}^m\int P\,\mathbf 1\{\mu(\phi_i)>R\}\,d\nuh_{n_k}
\le\frac1R\sum_{i=1}^m\int \mu(\phi_i)\,P\,d\nuh_{n_k}
\le\frac{C_m}{R},
\]
by the uniform moment bounds \eqref{eq:uniformmoments} (applied with $m+2$
functions); here $C_m$ depends on $m$ and on the fixed functions
$\varphi_0,\phi_1,\dots,\phi_m$ (through their common support radius and
sup-norms), but not on $k$ or $R$. We combine these as follows. The two inequalities
$P_R\le P$ and $\int(P-P_R)\,d\nuh_{n_k}\le C_m/R$ (the latter from the display
above) sandwich the truncated integral:
\[
\int P\,d\nuh_{n_k}-\frac{C_m}{R}\ \le\ \int P_R\,d\nuh_{n_k}\ \le\ \int
P\,d\nuh_{n_k}\qquad(\text{all }k,\ R>0).
\]
Now fix $R$ and let $k\to\infty$. The middle term converges to $\int
P_R\,d\Lambda$ by \eqref{eq:PRconv}; the outer terms converge because
$\int P\,d\nuh_{n_k}\to\Nh[P]$ --- this is the moment convergence
\eqref{eq:hatmoments}, \emph{not} a consequence of the weak convergence
$\nuh_{n_k}\Rightarrow\Lambda$, which does not apply to the unbounded $P$ ---,
holding for the full sequence $(\nuh_n)$ and hence along the subsequence
$(\nuh_{n_k})$.
Taking $k\to\infty$ in the sandwich therefore gives, for each fixed $R$,
\begin{equation}\label{eq:PRsandwich}
\Nh[P]-\frac{C_m}{R}\ \le\ \int P_R\,d\Lambda\ \le\ \Nh[P].
\end{equation}
Finally let $R\to\infty$. The left-hand bound tends to $\Nh[P]$; on the
middle, $P_R\uparrow P$ pointwise on $\cM$ as $R\to\infty$, so by monotone
convergence $\int P_R\,d\Lambda\uparrow\int P\,d\Lambda$. Hence
\eqref{eq:PRsandwich} forces
\[
\int P\,d\Lambda=\lim_{R\to\infty}\int P_R\,d\Lambda=\Nh[P],
\]
the finiteness $\Nh[P]<\infty$ following from Corollary~\ref{cor:carleman},
whence $P\in L^1(\Lambda)$.
The case $m=0$ (total masses) holds since $\nuh_{n_k}\Rightarrow\Lambda$ gives
$\nuh_{n_k}(\cM)\to\Lambda(\cM)$ (the constant $1$ being bounded and
continuous), together with the $m=0$ case of \eqref{eq:hatmoments},
displayed below it.

\emph{Step 2: identification.}
By Step 1, $\Lambda$ and $\Nh$ have identical (finite) moments. By
Corollary~\ref{cor:carleman} (equation~\eqref{eq:mgf}), $\Nh$ has a finite
exponential moment of $\mu(\phi)$ for every $\phi\in\Cc$ supported in a ball.
By Lemma~\ref{lem:determinacy} we have $\Lambda=\Nh$. Since every
subsequential limit equals $\Nh$, the full sequence converges:
$\nuh_n\Rightarrow\Nh$.
\end{proof}

\subsection{Structural properties of the canonical measure}
\label{sec:sbmstructure}

We now collect the structural facts about the canonical measure $\bN$ used
in the convergence statements that follow: the scaling relation, the law of the
total mass, and a Paley--Zygmund lower bound; the geometry of the range is
treated at the start of Section~\ref{sec:onearm}, where it is used. Only the Paley--Zygmund bound draws on the moment machinery of
Sections~\ref{sec:trees} and~\ref{sec:percolation} (through
Corollary~\ref{cor:carleman}); the other results follow directly from the
Laplace characterization \eqref{eq:laplace}.

\begin{proposition}[Scaling]\label{prop:scaling}
For $\lambda>0$ define $\Phi_\lambda:\cM\to\cM$ by
$(\Phi_\lambda\mu)(A):=\lambda^{-4}\mu(\lambda A)$. Then
\[
\bN\circ\bigl(\mu\mapsto\Phi_\lambda\mu\bigr)^{-1}
\;=\;\lambda^{-2}\,\bN
\qquad\text{as measures on }\cM .
\]
In particular $\supp(\Phi_\lambda\mu)=\lambda^{-1}\Ran$, so  the family
$K\mapsto\bN(\Ran\not\subset\overline B_K)$ scales as $K^{-2}$; see
Proposition~\ref{prop:range}.
\end{proposition}

\begin{proof}
Under $\bP_{\eps\delta_0}$, Brownian scaling together with the form of the
branching mechanism in \eqref{eq:laplace} shows that
$X^{(\lambda)}_t(A):=\lambda^{-2}X_{\lambda^2t}(\lambda A)$ is a
$(\frac12\Delta,\gamma)$-SBM started from $\lambda^{-2}\eps\,\delta_0$:
indeed if $u$ solves \eqref{eq:laplace} then
$u^{(\lambda)}(t,x):=\lambda^{2}u(\lambda^2t,\lambda x)$ solves the same
equation, and matching Laplace functionals gives the claim. Moreover
\[
\int_0^\infty X^{(\lambda)}_t\,dt
=\lambda^{-2}\int_0^\infty X_{\lambda^2 t}(\lambda\,\cdot)\,dt
=\lambda^{-4}\int_0^\infty X_{s}(\lambda\,\cdot)\,ds=\Phi_\lambda\mu .
\]
Hence, using the definition of $\bN$,
\[
\bN\bigl(\Phi_\lambda\mu\in\cdot\bigr)
=\lim_{\eps\downarrow0}\tfrac1\eps
\bP_{\lambda^{-2}\eps\,\delta_0}\bigl(\mu\in\cdot\bigr)
=\lambda^{-2}\lim_{\eps'\downarrow0}\tfrac1{\eps'}
\bP_{\eps'\delta_0}\bigl(\mu\in\cdot\bigr)
=\lambda^{-2}\,\bN\bigl(\mu\in\cdot\bigr). \qedhere
\]
\end{proof}

\begin{proposition}[Total mass]\label{prop:totalmass}
The law of $\mu(\R^d)$ under $\bN$ is atomless, has infinite mean on
$\{\mu(\R^d)>\eps\}$ for every $\eps>0$, and $\bN$ is an infinite but
$\sigma$-finite measure. The explicit Laplace transform
$\bN\bigl[1-e^{-\theta\mu(\R^d)}\bigr]=\sqrt{2\theta/\gamma}$
($\theta>0$) holds, giving the tail
$\bN\bigl(\mu(\R^d)>\eps\bigr)=\sqrt{2/(\pi\gamma)}\,\eps^{-1/2}$; this
explicit form is used for the cluster-tail consistency check
(Remark~\ref{rem:tailknown}) and for the finiteness of the restricted masses
$\bN(\mu(\R^d)>\eps)$; the rest of the paper requires only atomlessness.
\end{proposition}

\begin{proof}
Atomlessness follows from the scaling property
(Proposition~\ref{prop:scaling}), granted the finiteness $\bN(\mu(\R^d)>\eps)<\infty$ established independently below: $\Phi_\lambda$ multiplies
$\mu(\R^d)$ by $\lambda^{-4}$ and $\bN$ by $\lambda^{-2}$. Hence if the law
of $\mu(\R^d)$ under $\bN$ had an atom of mass $a>0$ at some value $w_0>0$,
then applying Proposition~\ref{prop:scaling} to the event
$\{\mu(\R^d)=\lambda^{-4}w_0\}$ --- whose $\Phi_\lambda$-preimage is
$\{\mu(\R^d)=w_0\}$ --- would give
$\bN\bigl(\mu(\R^d)=\lambda^{-4}w_0\bigr)
=\lambda^{2}\,\bN\bigl(\mu(\R^d)=w_0\bigr)=\lambda^{2}a$ for every
$\lambda>0$. For $\eps<w_0$ and $\lambda$ ranging over
$\bigl(1,(w_0/\eps)^{1/4}\bigr)$ this produces uncountably many distinct
atoms in $(\eps,w_0)$, each of $\bN$-mass $\ge a$, contradicting finiteness
of $\bN$ on $\{\mu(\R^d)>\eps\}$. For the explicit transform:
the total-mass process $Z_t:=X_t(\R^d)$ is, under $\bP_{\eps\delta_0}$, the
Feller diffusion $dZ_t=\sqrt{\gamma Z_t}\,dB_t$, $Z_0=\eps$, and
$W:=\mu(\R^d)=\int_0^\infty Z_t\,dt$. Set $v:=\sqrt{2\theta/\gamma}$ and
$M_t:=\exp\bigl(-\theta\int_0^t Z_s\,ds- v\,Z_t\bigr)$.
By It\^o's formula $M$ is a bounded martingale (the drift cancels by the
choice of $v$); optional stopping gives
$\E_{\eps\delta_0}[e^{-\theta W}]=e^{-\eps v}$, while the Poissonian
decomposition of $\bP_{\eps\delta_0}$ with intensity $\eps\bN$ gives
$\E_{\eps\delta_0}[e^{-\theta W}]=\exp(-\eps\,\bN[1-e^{-\theta W}])$, whence
$\bN[1-e^{-\theta W}]=\sqrt{2\theta/\gamma}$. Since
$\int_0^\infty(1-e^{-\theta x})\,x^{-3/2}dx=2\sqrt{\pi\theta}$, inverting the
transform gives the density $(2\pi\gamma)^{-1/2}x^{-3/2}$ on $(0,\infty)$ and
the stated tail.
\end{proof}

\begin{lemma}\label{lem:PZ}
Let $\phi\in\Cc$, $\phi\not\equiv0$. Then $\bN[\mu(\phi)]=\int G\phi\,
\in(0,\infty)$ and $\bN[\mu(\phi)^2]<\infty$
(Corollary~\ref{cor:carleman}); consequently
$\bN\bigl(\mu(\phi)>\eps\bigr)>0$ for all sufficiently small $\eps>0$, and
$\bN\bigl(\mu(\phi)>\eps\bigr)<\infty$ for all $\eps>0$.
\end{lemma}

\begin{proof}
The first-moment identity $\bN[\mu(\phi)]=\int G\phi$ is the case
$k=1$ of Proposition~\ref{prop:occmoments}, and lies in $(0,\infty)$ since
$\phi\ge0$, $\phi\not\equiv0$ and $G>0$; the second-moment finiteness
$\bN[\mu(\phi)^2]<\infty$ is the case $k=2$ of
Corollary~\ref{cor:carleman}. Since $\bN[\mu(\phi)]<\infty$, dominated
convergence gives $\bN[\mu(\phi)\mathbf 1\{\mu(\phi)\le\eps\}]\to0$ as
$\eps\downarrow0$, so for $\eps$ small enough
$\bN[\mu(\phi)\mathbf 1\{\mu(\phi)>\eps\}]\ge\tfrac12\int G\phi>0$; the
Cauchy--Schwarz inequality
$\bN[\mu(\phi);\mu(\phi)>\eps]^2\le\bN[\mu(\phi)^2]\,\bN(\mu(\phi)>\eps)$ then
forces $\bN(\mu(\phi)>\eps)>0$. Finiteness of $\bN(\mu(\phi)>\eps)$ for every
$\eps>0$ is Markov's inequality together with $\bN[\mu(\phi)]<\infty$.
\end{proof}

\subsection{$\sigma$-finite convergence of the cluster law}
\label{sec:sigmafinite}

\subsubsection{Un-biasing: restricted convergence}

\begin{theorem}[Restricted convergence]\label{thm:restricted}
 Assume \eqref{eq:2pt} and \eqref{eq:guiding-kp}, and fix
$\varphi_0\in\Cc$ with $\varphi_0\not\equiv0$. Then for every $\eps>0$ outside the (at
most countable) set $\{\eps:\bN(\mu(\varphi_0)=\eps)>0\}$
(we call such $\eps$ \emph{admissible}),
\[
\nu_n\bigl(\,\cdot\,;\ \mu(\varphi_0)>\eps\bigr)
\;\xRightarrow[n\to\infty]{}\;
\bN\bigl(\,\cdot\,;\ \mu(\varphi_0)>\eps\bigr)
\qquad\text{weakly in }\cMM,
\]
and both sides are finite measures.
\end{theorem}

\begin{proof}
Finiteness of $\bN(\mu(\varphi_0)>\eps)$ is Lemma~\ref{lem:PZ}. Define
\[
h(\mu):=\frac{\mathbf 1\{\mu(\varphi_0)>\eps\}}{\mu(\varphi_0)},
\qquad \mu\in\cM,
\]
with $h(\mu):=0$ when $\mu(\varphi_0)=0$. Then $0\le h\le1/\eps$, $h$
is continuous at every $\mu$ with $\mu(\varphi_0)\neq\eps$ (the map
$\mu\mapsto\mu(\varphi_0)$ is weakly continuous since
$\varphi_0\in C_b(\R^d)$), and the discontinuity set
$\{\mu(\varphi_0)=\eps\}$ is $\Nh$-null by the choice of $\eps$. Indeed
$\Nh(\mu(\varphi_0)=\eps)=\eps\,\bN(\mu(\varphi_0)=\eps)$ for $\eps>0$, so
$\eps$ is an atom of the law of $\mu(\varphi_0)$ under $\Nh$ if and only if it is
one under $\bN$; thus $\Nh$ and $\bN$ share the \emph{same} exceptional set,
which is countable (the law of $\mu(\varphi_0)$ under $\bN$ restricted to
$\{\mu(\varphi_0)>\delta\}$, $\delta>0$, is a finite measure and so has at most
countably many atoms), and the chosen $\eps$ is admissible for both.
Since
\[
\nu_n\bigl(\,\cdot\,;\mu(\varphi_0)>\eps\bigr)=h\,\nuh_n,
\qquad
\bN\bigl(\,\cdot\,;\mu(\varphi_0)>\eps\bigr)=h\,\Nh,
\]
the claim follows from Theorem~\ref{thm:biased}: if $\nuh_n\Rightarrow\Nh$
(finite measures) and $h$ is bounded and
$\Nh$-a.e.\ continuous, then $h\,\nuh_n\Rightarrow h\,\Nh$ (test against
$F\in C_b(\cM)$; $Fh$ is bounded and $\Nh$-a.e.\ continuous, and the
portmanteau theorem in its a.e.-continuity form applies).
\end{proof}

\subsubsection{The total-mass formulation}

We now prove Theorem~\ref{thm:mainA}, under the hypotheses $d>6$,
\eqref{eq:2pt} and \eqref{eq:guiding-kp}, which are all valid for the
nearest-neighbour model with $d\ge 11$.

\begin{proof}[Proof of Theorem~\ref{thm:mainA}]
Fix $F\in C_b(\cM)$, $\|F\|_\infty\le1$, and $\eps>0$. For $K\ge1$ let
$\chi_K\in\Cc$ with $\mathbf 1_{\overline B_K}\le\chi_K\le\mathbf 1_{\overline B_{K+1}}$. We
compare the events $\{\mu(\R^d)>\eps\}$ and $\{\mu(\chi_K)>\eps\}$ under
$\nu_n$: since $\mu(\chi_K)\le\mu(\R^d)$ always, and
$\mu(\chi_K)=\mu(\R^d)$ whenever $\supp\mu\subset\overline B_K$, the symmetric
difference is contained in $\{\supp\mu\not\subset\overline B_K\}$, whose
$\nu_n$-measure is
$n^2\bP(0\leftrightarrow\partial B_{Kn})\le K_2K^{-2}$ by
Proposition~\ref{prop:KN}. Hence, uniformly in $n$,
\begin{equation}\label{eq:sandwich1}
\Bigl|\int F\,\mathbf 1\{\mu(\R^d)>\eps\}\,d\nu_n
-\int F\,\mathbf 1\{\mu(\chi_K)>\eps\}\,d\nu_n\Bigr|\le K_2K^{-2}.
\end{equation}
On the $\bN$ side, $\mu(\chi_K)\uparrow\mu(\R^d)$ as $K\to\infty$, so the events
$\{\mu(\chi_K)>\eps\}$ increase to $\{\mu(\R^d)>\eps\}$; hence the restricted
measures $\bN(\,\cdot\,;\mu(\chi_K)>\eps)$ --- each a \emph{finite} measure on
$\cM$ by Lemma~\ref{lem:PZ} --- increase setwise to
$\bN(\,\cdot\,;\mu(\R^d)>\eps)$ by monotone convergence, whence
\begin{equation}\label{eq:sandwich2}
\Bigl|\bN\bigl[F;\mu(\chi_K)>\eps\bigr]
-\bN\bigl[F;\mu(\R^d)>\eps\bigr]\Bigr|
\le\bN\bigl(\mu(\R^d)>\eps\bigr)-\bN\bigl(\mu(\chi_K)>\eps\bigr)
\xrightarrow[K\to\infty]{}0 .
\end{equation}
Now fix $\delta>0$. Choose $K$ with $K_2K^{-2}\le\delta$ and with the
right side of \eqref{eq:sandwich2} at most $\delta$. We next choose a
level $\eps'\in(\eps/2,\eps)$ at which the restricted convergence
(Theorem~\ref{thm:restricted}) is available for the test function $\chi_K$;
this requires $\eps'$ to avoid the at-most-countable set
$\{s:\bN(\mu(\chi_K)=s)>0\}$ of atoms of the law of $\mu(\chi_K)$ under $\bN$
(this is the only role of ``the exceptional $\eps$-set for $\chi_K$ is
countable''). We additionally require $\bN(\eps'<\mu(\R^d)\le\eps)\le\delta$,
which is possible because the law of $\mu(\R^d)$ under $\bN$ has the explicit
atomless density of Proposition~\ref{prop:totalmass}, so this quantity tends
to $0$ as $\eps'\uparrow\eps$. Two more comparisons:
first, by Theorem~\ref{thm:restricted} applied with $(\chi_K,\eps')$,
\begin{equation}\label{eq:restrconv}
\int F\,\mathbf 1\{\mu(\chi_K)>\eps'\}\,d\nu_n
\longrightarrow \bN\bigl[F;\mu(\chi_K)>\eps'\bigr];
\end{equation}
second, the discrepancy between the levels $\eps$ and $\eps'$ is
controlled, uniformly in large $n$, by
\begin{equation}\label{eq:levelgap}
\begin{aligned}
\limsup_n\,\nu_n\bigl(\eps'<\mu(\chi_K)\le\eps\bigr)
&\le\bN\bigl(\eps'\le\mu(\chi_K)\le\eps\bigr)\\
&\le\bN\bigl(\eps'\le\mu(\R^d)\le\eps\bigr)
+\Bigl(\bN\bigl(\mu(\R^d)>\eps\bigr)-\bN\bigl(\mu(\chi_K)>\eps\bigr)\Bigr)\\
&\le \delta+\delta=2\delta .
\end{aligned}
\end{equation}
Here the first inequality is the portmanteau theorem (closed-set form) for the
weakly convergent finite measures of Theorem~\ref{thm:restricted}, applied with the weight $\chi_K$ at the level $\eps'$ (cf.\ \eqref{eq:restrconv}): writing
$\widetilde\nu_n:=\nu_n(\,\cdot\,;\mu(\chi_K)>\eps')$ and
$\widetilde\bN:=\bN(\,\cdot\,;\mu(\chi_K)>\eps')$, the set
$C:=\{\eps'\le\mu(\chi_K)\le\eps\}$ is closed in $\cM$ (as $\mu\mapsto\mu(\chi_K)$
is weakly continuous) and $\{\eps'<\mu(\chi_K)\le\eps\}=C\cap\{\mu(\chi_K)>\eps'\}$,
so $\nu_n(\eps'<\mu(\chi_K)\le\eps)=\widetilde\nu_n(C)$ and, by
\eqref{eq:restrconv} and portmanteau,
$\limsup_n\widetilde\nu_n(C)\le\widetilde\bN(C)=\bN(\eps'<\mu(\chi_K)\le\eps)
\le\bN(\eps'\le\mu(\chi_K)\le\eps)$. For the second inequality we
split the event $\{\eps'\le\mu(\chi_K)\le\eps\}$ according to whether
$\mu(\R^d)\le\eps$ or not. Since $\mu(\chi_K)\le\mu(\R^d)$, the part with
$\mu(\R^d)\le\eps$ satisfies $\eps'\le\mu(\chi_K)\le\mu(\R^d)\le\eps$, hence is
contained in $\{\eps'\le\mu(\R^d)\le\eps\}$. The part with $\mu(\R^d)>\eps$
satisfies $\mu(\chi_K)\le\eps<\mu(\R^d)$, hence is contained in
$\{\mu(\R^d)>\eps\}\setminus\{\mu(\chi_K)>\eps\}$, whose $\bN$-mass is the
displayed difference $\bN(\mu(\R^d)>\eps)-\bN(\mu(\chi_K)>\eps)$. The third
inequality uses the two controls fixed above: the first term satisfies
$\bN(\eps'\le\mu(\R^d)\le\eps)=\bN(\eps'<\mu(\R^d)\le\eps)\le\delta$, the equality
holding because $\mu(\R^d)$ is atomless under $\bN$ (no atom at $\eps'$) and the
bound by the choice of $\eps'$ above, and the second is the right side of
\eqref{eq:sandwich2} at level $\eps$, which was arranged to be $\le\delta$ in
the choice of $K$.

Assembling, write
$\nu_n[F;\mu(\R^d)>\eps]-\bN[F;\mu(\R^d)>\eps]$ as the telescoping sum of:
(a) $\nu_n[F;\mu(\R^d)>\eps]-\nu_n[F;\mu(\chi_K)>\eps]$, bounded by
$K_2K^{-2}\le\delta$ via \eqref{eq:sandwich1};
(b) $\nu_n[F;\mu(\chi_K)>\eps]-\nu_n[F;\mu(\chi_K)>\eps']$, bounded in
$\limsup_n$ by $2\delta$ (recall $\|F\|_\infty\le1$) via \eqref{eq:levelgap};
(c) $\nu_n[F;\mu(\chi_K)>\eps']-\bN[F;\mu(\chi_K)>\eps']\to0$ by
\eqref{eq:restrconv};
(d) $\bN[F;\mu(\chi_K)>\eps']-\bN[F;\mu(\R^d)>\eps]$, whose modulus is at most
(as $\|F\|_\infty\le1$) the $\bN$-mass of the symmetric difference of
$\{\mu(\chi_K)>\eps'\}$ and $\{\mu(\R^d)>\eps\}$; that symmetric difference is
contained in
$\{\eps'<\mu(\R^d)\le\eps\}\cup\bigl(\{\mu(\R^d)>\eps\}\setminus\{\mu(\chi_K)>\eps\}\bigr)$,
of $\bN$-mass at most $\delta+\delta=2\delta$ by the choice of $\eps'$ and
of $K$, via \eqref{eq:sandwich2}.
Hence $\limsup_n|\nu_n[F;\mu(\R^d)>\eps]-\bN[F;\mu(\R^d)>\eps]|\le C\delta$,
and $\delta>0$ was arbitrary.
\end{proof}

\begin{remark}[Recovery of the known cluster-tail asymptotics]
\label{rem:tailknown}
Evaluating the total-mass restriction (Theorem~\ref{thm:mainA})
on the constant test function $F\equiv1$ and using
$\mu_n(\R^d)=|\mathcal C|/(\mathfrak a n^4)$ together with the explicit tail
$\bN(\mu(\R^d)>\eps)=\sqrt{2/(\pi\gamma)}\,\eps^{-1/2}$ of
Proposition~\ref{prop:totalmass} gives, by the same monotone-interpolation in
$N$ as for the one-arm radius,
\[
\lim_{N\to\infty}\sqrt N\,\bP\bigl(|\mathcal C|\ge N\bigr)
=\sqrt{\frac{2\mathfrak a}{\pi\gamma}} .
\]
This recovers a known result: the order $\bP(|\mathcal C|\ge N)\asymp N^{-1/2}$ is classical in the high-dimensional regime (see the discussion around \eqref{voltail}), while the sharp local limit law $\bP(|\mathcal C|=N)\sim A\,N^{-3/2}$ was obtained by Hara and Slade \cite{HS00b}. We record it only as a consistency
check on the normalization \eqref{eq:empirical}, the identification of the
constant being an automatic output of the $k$-point asymptotics rather than a
new input.
\end{remark}

%
\begin{corollary}[Conditional convergence]\label{cor:conditional}\label{cor:conditionalmass}
Let either $W(\mu):=\mu(\varphi_0)$, with $\varphi_0\in\Cc$, $\varphi_0\not\equiv0$, fixed
and $\eps$ admissible in the sense of Theorem~\ref{thm:restricted} and small
enough that $\bN(\mu(\varphi_0)>\eps)>0$ (Lemma~\ref{lem:PZ}); or
$W(\mu):=\mu(\R^d)$, with $\eps>0$ arbitrary --- every $\eps$ being admissible
here, since $\mu(\R^d)$ has no atoms under $\bN$ and
$\bN(\mu(\R^d)>\eps)\in(0,\infty)$ (Proposition~\ref{prop:totalmass}). Then
\[
\bP\bigl(\mu_n\in\cdot\;\big|\;W(\mu_n)>\eps\bigr)
\;\xRightarrow[n\to\infty]{}\;
\bN\bigl(\,\cdot\,\big|\;W(\mu)>\eps\bigr)
:=\frac{\bN(\,\cdot\,;\,W(\mu)>\eps)}{\bN(W(\mu)>\eps)}
\]
weakly as probability measures on $\cM$, and
$n^{2}\,\bP\bigl(W(\mu_n)>\eps\bigr)\to\bN\bigl(W(\mu)>\eps\bigr)\in(0,\infty)$.
\end{corollary}

\begin{proof}
Divide the convergence of Theorem~\ref{thm:restricted} (for
$W(\mu)=\mu(\varphi_0)$) resp.\ Theorem~\ref{thm:mainA} (for
$W(\mu)=\mu(\R^d)$), tested against a bounded continuous $F$, by the same
convergence tested against the constant function $1$, i.e.\ by
$\nu_n(W(\mu)>\eps)\to\bN(W(\mu)>\eps)$; positivity and finiteness of the
normalizer are Lemma~\ref{lem:PZ} resp.\ Proposition~\ref{prop:totalmass}.
\end{proof}

\section{Consequences of the measure convergence: range convergence and sharp one-arm asymptotics}
\label{sec:onearm}

In this section we combine the measure convergence (Theorem~\ref{thm:mainA})
with the lower mass bound of Section~\ref{sec:lmb} and deduce the two
geometric results: the rescaled cluster converges as a compact set to the
range of super-Brownian motion (Theorem~\ref{thm:mainD}), and the one-arm
probability has an exact $r^{-2}$ asymptotic with identified constant
(Theorem~\ref{thm:mainC}). Recall
$\theta_K=\bN(\Ran\not\subset\overline B_K)=\theta_1K^{-2}$ from
Proposition~\ref{prop:range} below.

Let us sketch the argument. Weak convergence of measures cannot see whether
points of the support carry mass; the bridge is the deterministic principle of
Athreya--L\"ohr--Winter (Lemma~\ref{lem:suppconv}): weak convergence upgrades
to Hausdorff convergence of supports once a lower mass bound holds uniformly
away from the origin. That bound enters through the quantitative form,
Theorem~\ref{thm:ulmb}, via Lemma~\ref{lem:nothin}, which upgrades the
qualitative statement that the cluster \emph{reaches} a region --- a property
of the support, invisible to weak convergence --- into the statement that the
empirical measure gives that region mass of order $r^4$, to which the
portmanteau theorem applies. For the one-arm asymptotics the exit event must first be
localized by a weight $\varphi_0$, so it is not Theorem~\ref{thm:mainD}
itself that is evaluated but its $\varphi_0$-conditioned variant
(Remark~\ref{rem:rangephi}); evaluating that variant on the exit event and
removing the localization gives
$n^2\,\bP(0\leftrightarrow\partial B_{Kn})\to\theta_1K^{-2}$, and scaling in
$K$ produces the one-arm constant.

\subsection{The range of super-Brownian motion}

We first record the structural facts about the range $\Ran=\supp\mu$ under
$\bN$; the proof uses the measure convergence, Theorem~\ref{thm:mainA}.

\begin{proposition}[Range]\label{prop:range}
$\bN$-a.e.:
\begin{enumerate}
\item[\rm(i)] $\Ran=\supp\mu$ coincides with the closure of
$\{0\}\cup\bigcup_{t>0}\supp X_t$, and $0\in\Ran$;
\item[\rm(ii)] $\Ran$ is compact and connected.
\end{enumerate}
Moreover, for $K>0$,
\[
\theta_K:=\bN\bigl(\Ran\not\subset\overline B_K\bigr)\in(0,\infty),
\qquad \theta_K=\theta_1\,K^{-2}.
\]
\end{proposition}

\begin{proof}[Proof sketch and references]
Parts (i)--(ii) are classical properties of the
range of super-Brownian motion under the canonical measure; we indicate
references rather than reproving them.
(i) The inclusion $\supp\mu\subset\overline{\bigcup_t\supp X_t}$ is
immediate from $\mu=\int X_t\,dt$. Conversely, if $x\in\supp X_t$ for some
$t>0$, the weak continuity of $s\mapsto X_s$ gives $X_s(B(x,r))>0$ for $s$
in a neighbourhood of $t$, hence $\mu(B(x,r))>0$, for every $r>0$. That $0\in\Ran$ will follow from the snake representation used in~(ii).

(ii) Under $\bN$, the range admits the Brownian snake representation
\[
\Ran=\{\hat W_s:\,0\le s\le\sigma\},
\]
where $s\mapsto\hat W_s$ is the (continuous) snake head, $\sigma$ is the lifetime of the snake---which coincides with the extinction time of \eqref{eq:survival}---and $\hat W_0=0$ \cite[Ch.~IV.5, display~(a)]{LG99}. Being the continuous image of the compact interval $[0,\sigma]$, the range $\Ran$ is compact and connected; and since it contains $\hat W_0=0$, this also yields the claim $0\in\Ran$ of~(i).
For the exit quantity, $\theta_K\in(0,\infty)$. We first prove finiteness. Fix
$\rho\in(0,K)$ and let $g\in C_b(\R^d)$ with $0\le g\le1$, $g\equiv0$ on
$B_{K-\rho}$ and $g>0$ on $\{|x|>K-\rho\}$. If $\Ran\not\subset\overline B_K$ then
$\Ran=\supp\mu$ contains a point $y$ with $|y|\ge K>K-\rho$, and since $\mu$
charges every neighbourhood of $y$ we have $\mu(g)>0$; thus
$\{\Ran\not\subset\overline B_K\}\subset\{\mu(g)>0\}$, which is an open subset of $\cM$
because $\mu\mapsto\mu(g)$ is weakly continuous. By the total-mass restricted
convergence (Theorem~\ref{thm:mainA}) and the portmanteau theorem for open
sets, for every $\eps>0$,
\[
  \bN\bigl(\mu(g)>0;\,\mu(\R^d)>\eps\bigr)
  \;\le\;\liminf_{n\to\infty}\nu_n\bigl(\mu(g)>0\bigr).
\]
Moreover $\mu(g)>0$ forces $\mu_n$ to charge $\{|x|>K-\rho\}$, i.e.\
$\mathcal C\not\subset\overline B_{(K-\rho)n}$, so by $\nu_n=n^2\bP(\mu_n\in\cdot)$ and the
Kozma--Nachmias one-arm upper bound (Proposition~\ref{prop:KN}), for every $n$
with $(K-\rho)n\ge1$,
\[
  \nu_n\bigl(\mu(g)>0\bigr)
  \;\le\;n^2\,\bP\bigl(0\leftrightarrow\partial B_{(K-\rho)n}\bigr)
  \;\le\;n^2\,K_2\,\bigl((K-\rho)n\bigr)^{-2}
  \;=\;K_2\,(K-\rho)^{-2}.
\]
Hence $\bN(\mu(g)>0;\,\mu(\R^d)>\eps)\le K_2(K-\rho)^{-2}$; letting
$\eps\downarrow0$ (monotone convergence, using
$\{\Ran\not\subset\overline B_K\}\subset\{\mu\ne0\}$) and then $\rho\downarrow0$ gives
$\theta_K=\bN(\Ran\not\subset\overline B_K)\le K_2K^{-2}<\infty$. Positivity follows from
Lemma~\ref{lem:PZ} applied to a test function supported in
$B_{2K}\setminus\overline B_K$ together with (i).
 The scaling
$\theta_K=\theta_1K^{-2}$ is a direct consequence of
Proposition~\ref{prop:scaling}: since
$\{\Ran\not\subset\overline B_K\}=\{\supp(\Phi_K\mu)\not\subset\overline B_1\}$ and
$\bN\circ\Phi_K^{-1}=K^{-2}\bN$, we get $\theta_K=K^{-2}\theta_1$.
\end{proof}

\subsection{No thin visited regions}

We use the lower mass bound in the quantitative form of
Theorem~\ref{thm:ulmb}, through the following consequence: at macroscopic
distances the cluster has no asymptotically thin visited regions. Recall that
$Q_{z,s}$ denotes the lattice $\ell_\infty$ box (Section~\ref{sec:notation}).

\begin{lemma}[No thin visited regions]\label{lem:nothin}
We invoke Theorem~\ref{thm:ulmb}; we also use the one-arm bound
(Proposition~\ref{prop:KN}).
Fix $K<\infty$, $\kappa\in(0,K)$ and
$\delta\in(0,\kappa/(5\sqrt d)]$. Then for every
$\eta\in(0,1]$,
\begin{multline*}
\limsup_{n\to\infty}\;
n^2\,\bP\Bigl(\exists\,x\in \mathcal C:\ \kappa n\le|x|\le Kn,\
\bigl|\mathcal C\cap \overline B(x,3\sqrt d\,\delta n)\bigr|\le\eta(\delta n)^4\Bigr)
\\
\le\;C(d)\,C_3\,K_2\Bigl(\frac{K}{\delta}\Bigr)^{d}\kappa^{-2}\,\eta^{1/4},
\end{multline*}
where $C_3=C_3(\Lambda)$ with $\Lambda:=(K+\delta)/\delta$ is the
constant of Theorem~\ref{thm:ulmb} and $K_2$ is the constant of the one-arm
bound (Proposition~\ref{prop:KN}).
\end{lemma}

\begin{figure}[htb]
\centering
\begin{tikzpicture}[scale=0.52]
\draw[thick] (0,0) circle (3);
\draw[thick] (0,0) circle (6.5);
\fill (0,0) circle (2.5pt); \node[below left] at (0,0) {$0$};
\node at (2.55,-2.15) {\small$|y|=\kappa n$};
\node at (5.35,-4.35) {\small$|y|=Kn$};
\foreach \a in {10,32,...,350} \fill[gray] ({4.1*cos(\a)},{4.1*sin(\a)}) circle (1.4pt);
\foreach \a in {22,44,...,350} \fill[gray] ({5.3*cos(\a)},{5.3*sin(\a)}) circle (1.4pt);
\draw[very thick, decorate, decoration={random steps,segment length=6pt,amplitude=2.5pt}] (0,0) -- (3.95,2.65);
\fill (3.95,2.65) circle (2.5pt); \node[below right] at (3.95,2.55) {$x$};
\fill (4.35,3.05) circle (2pt); \node[right] at (4.45,3.05) {$\hat x$};
\draw (3.65,2.35) rectangle (5.05,3.75);
\draw[dashed] (2.95,1.65) rectangle (5.75,4.45);
\draw[dotted,thick] (3.95,2.65) circle (2.45);
\node[align=left] at (8.6,4.3) {\small$Q_{\hat x,2r_n}$\\[-2pt]\small$(r_n=\delta n)$};
\node at (1.25,5.2) {\small$\overline B(x,3\sqrt d\,\delta n)$};
\end{tikzpicture}
\caption{The net construction in Lemma~\ref{lem:nothin}. A point
$x\in \mathcal C$ realizing the thin-region event lies within $\delta n/2$ of a net
point $\hat x$ (grey dots); the cluster then reaches the box
$Q_{\hat x,r_n}$ (solid), while the enlarged box
$Q_{\hat x,2r_n}$ (dashed) is contained in the ball
$\overline B(x,3\sqrt d\,\delta n)$ (dotted), so a thin ball forces a light box and
Theorem~\ref{thm:ulmb} applies.}
\label{fig:netconstruction}
\end{figure}
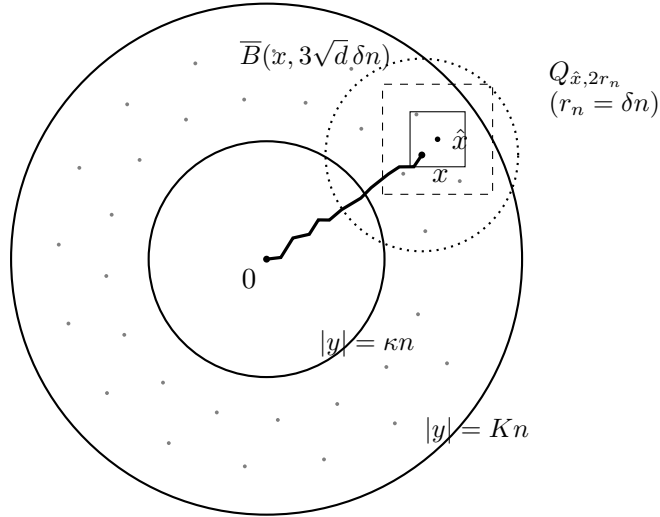

\begin{proof}
Denote by $\mathcal E_n$ the event in the statement, i.e.\
$\mathcal E_n=\{\exists\,x\in \mathcal C:\ \kappa n\le|x|\le Kn,\ |\mathcal C\cap
\overline B(x,3\sqrt d\,\delta n)|\le\eta(\delta n)^4\}$.
The net construction underlying the proof is illustrated in
Figure~\ref{fig:netconstruction}.
Let $N$ be a maximal subset of $\{y\in\Z^d:\kappa n\le|y|\le Kn\}$ with
the property that $|x-y|>\delta n/2$ for all distinct $x,y\in N$. By
maximality, every lattice point $y$ with $\kappa n\le|y|\le Kn$ satisfies
$|y-\hat x|\le\delta n/2$ for some $\hat x\in N$, so $N$ is a
$(\delta n/2)$-net of the lattice points of the annulus (which suffices below,
since the covered point lies in $\mathcal C\subset\Z^d$). Moreover, the Euclidean balls
$\{B(\hat x,\delta n/4):\hat x\in N\}$ are pairwise disjoint and contained in
$B_{(K+\delta)n}$, so a standard volume comparison yields
$|N|\le C(d)(K/\delta)^d$. Write
$r_n:=\delta n\to\infty$. On $\mathcal E_n$, pick a point $x\in \mathcal C$
realizing the event and $\hat x\in N$ with
$|x-\hat x|\le\delta n/2$. Then $\lVert x-\hat x\rVert_\infty\le|x-\hat x|\le\delta n/2\le r_n$,
so $x\in Q_{\hat x,r_n}$ and $0\leftrightarrow Q_{\hat x,r_n}$;
and $Q_{\hat x,2r_n}\subset \overline B(\hat x,2\sqrt d\,r_n)\subset
\overline B(x,3\sqrt d\,\delta n)$, so $|\mathcal C\cap Q_{\hat
x,2r_n}|\le\eta(\delta n)^4=\eta\,r_n^4$. Hence by a union
bound,
\begin{equation}\label{eq:nothinunion}
\bP(\mathcal E_n)\le\sum_{\hat x\in N}
\bP\bigl(0\leftrightarrow Q_{\hat x,r_n}\bigr)\cdot
\bP\Bigl(|\mathcal C\cap Q_{\hat x,2r_n}|\le\eta r_n^4
\,\Big|\,0\leftrightarrow Q_{\hat x,r_n}\Bigr).
\end{equation}
For the conditional factor, each $\hat x\in N$ satisfies
$\lVert\hat x\rVert_\infty\ge|\hat x|/\sqrt d\ge\kappa n/\sqrt d\ge
4\delta n>3r_n$ (using $\delta\le\kappa/(5\sqrt d)$, which gives
$\kappa\ge4\sqrt d\,\delta$), and
$\lVert\hat x\rVert_\infty\le|\hat x|\le(K+\delta)n=\Lambda r_n$ with
$\Lambda:=(K+\delta)/\delta>3$.
Theorem~\ref{thm:ulmb} therefore applies with this $\Lambda$ and bounds
this conditional factor --- the conditional probability, given $0\leftrightarrow Q_{\hat x,r_n}$,
that the cluster leaves at most $\eta r_n^4$ of its sites in $Q_{\hat x,2r_n}$
--- uniformly over $\hat x\in N$ by $C_3\,\eta^{1/4}+o_n(1)$, the $o_n(1)$ coming
from the $\limsup_{r\to\infty}$ in the theorem. For the reach factor, $0\leftrightarrow Q_{\hat x,r_n}$ forces $0\leftrightarrow\partial B_{(3\kappa/5)n}$
(every site of $Q_{\hat x,r_n}$ is at Euclidean distance
$\ge|\hat x|-\sqrt d\,r_n\ge(\kappa-\sqrt d\,\delta)n\ge(3\kappa/5)n$ from
the origin, using $\delta\le\kappa/(5\sqrt d)$)%
, so
by the one-arm bound (Proposition~\ref{prop:KN}),
\[
\bP\bigl(0\leftrightarrow Q_{\hat x,r_n}\bigr)
\le\bP\bigl(0\leftrightarrow\partial B_{(3\kappa/5)n}\bigr)
\le K_2\bigl((3\kappa/5)n\bigr)^{-2}=\tfrac{25}{9}K_2(\kappa n)^{-2}.
\]
Summing \eqref{eq:nothinunion} over the $|N|\le C(d)(K/\delta)^d$ net points
and multiplying by $n^2$,
\[
\limsup_{n\to\infty} n^2\,\bP(\mathcal E_n)
\le C(d)\Bigl(\frac K\delta\Bigr)^d\cdot \tfrac{25}{9}K_2\kappa^{-2}\cdot C_3\,\eta^{1/4}
= C(d)\,C_3\,K_2\Bigl(\frac K\delta\Bigr)^d\kappa^{-2}\,\eta^{1/4},
\]
which is the claim.
\end{proof}

\subsection{Diameter tightness}

Beyond the lower mass bound of Lemma~\ref{lem:nothin}, Hausdorff
convergence of the supports needs one further, standard ingredient: the cluster
does not send a thin tendril out to infinity, i.e.\ its rescaled diameter is
tight. This is the only escape route invisible both to weak convergence of
measures and to the property of Lemma~\ref{lem:nothin}, and it is
precluded by a one-line consequence of the one-arm bound. The two ingredients play orthogonal
roles in the range proof: Lemma~\ref{lem:nothin} supplies the lower mass bound
(no occupied region is asymptotically thin), forcing every limit point of the
support to carry mass, while the diameter tightness below caps the support in a
fixed ball, which is what restores compactness in the non-compact hyperspace
$\mathcal K(\R^d)$.

\begin{lemma}[Diameter tightness]\label{lem:diamtight}
Let $\eps>0$, so that
$b_\eps:=\bN(\mu(\R^d)>\eps)\in(0,\infty)$ (Proposition~\ref{prop:totalmass},
Corollary~\ref{cor:conditionalmass}). Then the rescaled radius
$\sup\{|y|:y\in n^{-1}\mathcal C\}$ is tight under
$\bP(\,\cdot\mid\mu_n(\R^d)>\eps)$: for every $\theta>0$ there is
$K=K(\theta,\eps)<\infty$ with
\[
\limsup_{n\to\infty}\
\bP\bigl(n^{-1}\mathcal C\not\subset\overline B_K\,\big|\,\mu_n(\R^d)>\eps\bigr)\ \le\ \theta .
\]
The same holds with $\mu_n(\R^d)>\eps$ replaced by $\mu_n(\varphi_0)>\eps$ for
any $\varphi_0\in\Cc$ with $\varphi_0\not\equiv0$, but only for admissible $\eps$ in the
sense of Corollary~\ref{cor:conditional}.
\end{lemma}

\begin{proof}
Since $\{\sup\{|y|:y\in n^{-1}\mathcal C\}>K\}=\{n^{-1}\mathcal C\not\subset\overline B_K\}=\{\mathcal C\not\subset
\overline B_{Kn}\}=\{0\leftrightarrow\partial B_{Kn}\}$, the one-arm bound
Proposition~\ref{prop:KN} gives
$\bP(n^{-1}\mathcal C\not\subset\overline B_K)\le K_2(Kn)^{-2}$. By
Corollary~\ref{cor:conditionalmass}, $n^2\bP(\mu_n(\R^d)>\eps)\to b_\eps\in
(0,\infty)$, so $\bP(\mu_n(\R^d)>\eps)\ge\tfrac12 b_\eps n^{-2}$ for all large
$n$. Dividing,
\[
\bP\bigl(n^{-1}\mathcal C\not\subset\overline B_K\,\big|\,\mu_n(\R^d)>\eps\bigr)
\ \le\ \frac{K_2(Kn)^{-2}}{\tfrac12 b_\eps n^{-2}}
\ =\ \frac{2K_2}{b_\eps}\,K^{-2},
\]
for all large $n$, the two powers of $n^{-2}$ --- one-arm and survival
--- cancelling exactly. Choosing $K$ with $2K_2 b_\eps^{-1}K^{-2}\le
\theta$ gives the claim. The $\varphi_0$ version is identical, using
Corollary~\ref{cor:conditional} (which gives
$n^2\bP(\mu_n(\varphi_0)>\eps)\to\bN(\mu(\varphi_0)>\eps)\in(0,\infty)$) in place
of Corollary~\ref{cor:conditionalmass}.
\end{proof}

Section~\ref{sec:lmb} (Theorem~\ref{thm:mainLMB}) in fact establishes a stronger, \emph{global} uniform
lower mass bound --- no occupied region is thin at \emph{any} scale, not merely
in a fixed annulus --- from which diameter tightness also follows; we do not
need that strength here, and the range proof below uses only
Lemma~\ref{lem:nothin} and Lemma~\ref{lem:diamtight}, both established within
this part.

\subsection{A deterministic support-convergence lemma}

The passage from measure convergence to range convergence rests on a
purely deterministic principle, isolated by Athreya--L\"ohr--Winter
\cite{ALW16} in the Gromov setting: on a fixed locally compact space, weak
convergence of measures upgrades to Hausdorff convergence of their
supports as compact sets precisely when a uniform lower mass bound holds
along the sequence and the supports remain uniformly bounded.
We record the version we need, on $\R^d$ with the Euclidean metric, where the
ambient space is fixed and no metric embedding is required. Our proof
follows the corresponding argument of \cite{ALW16}.

\begin{lemma}[Support convergence from a uniform lower mass bound]
\label{lem:suppconv}
Let $\mu_n,\mu$ be finite Borel measures on $\R^d$ with
$\mu_n\Rightarrow\mu$ weakly and $\mu\neq0$. Suppose:
\begin{enumerate}[label=\textup{(\roman*)},nosep,leftmargin=2.2em]
\item \emph{(Inner-uniform lower mass bound.)} For every $\rho_0>0$ and
every $\delta>0$ there exist $c=c(\rho_0,\delta)>0$ and $n_0$ such that for
all $n\ge n_0$,
\begin{equation}\label{eq:LMBdet}
\mu_n\bigl(\overline B(x,\delta)\bigr)\ \ge\ c(\rho_0,\delta)
\qquad\text{for every }x\in\supp\mu_n\text{ with }|x|\ge\rho_0\,;
\end{equation}
\item \emph{(Uniform boundedness.)} $R_\ast:=\sup_n\sup\{|x|:x\in\supp\mu_n\}<\infty$;
\item \emph{(Anchored core.)} $0\in\supp\mu$.
\end{enumerate}
Then $\supp\mu_n\to\supp\mu$ in the Hausdorff metric on the space
$\mathcal K(\R^d)$ of nonempty compact subsets of $\R^d$.
\end{lemma}

The quantifier in~(i) ranges over the \emph{inner} radius $\rho_0$, not
over a large outer cutoff: escape of support mass to infinity --- the sole
obstruction to Hausdorff convergence once the supports are known to converge
pointwise --- is excluded by
the uniform boundedness~(ii), while the single degenerate small-scale limit
point, the origin, is supplied directly by~(iii). The mass near $0$ itself is
\emph{not} assumed to be bounded below, which is why~(i) is imposed
only away from it.

\begin{proof}
All supports lie in the fixed ball $\overline B_{R_\ast}$ by~(ii); in
particular $\supp\mu\subseteq\overline B_{R_\ast}$ (any $z$ with $|z|>R_\ast$ has
a neighbourhood $U$ disjoint from every $\supp\mu_n$, so $\mu(U)\le\liminf_n
\mu_n(U)=0$ and $z\notin\supp\mu$), so $\supp\mu$ is a nonempty compact set, and
all the sets in play belong to $\mathcal K(\R^d)$. Hausdorff convergence
$\supp\mu_n\to\supp\mu$ is the conjunction of two one-sided statements
(\cite[\S 12.3]{SW08}):
\[
\textup{(U)}\quad\sup_{x\in\supp\mu_n}\dist(x,\supp\mu)\to0,
\qquad
\textup{(L)}\quad\sup_{y\in\supp\mu}\dist(y,\supp\mu_n)\to0 .
\]

\emph{Pointwise lower inclusion} (does not use \eqref{eq:LMBdet}). Let
$z\in\supp\mu$ and $r>0$. The open ball $B(z,r)$ has $\mu(B(z,r))>0$ (as
$z\in\supp\mu$), so by the portmanteau theorem
$\liminf_n\mu_n(B(z,r))\ge\mu(B(z,r))>0$; hence
$\supp\mu_n\cap B(z,r)\neq\emptyset$ for all large $n$. Thus every
$z\in\supp\mu$ is a limit of points $x_n\in\supp\mu_n$.

\emph{Pointwise upper inclusion} (uses \eqref{eq:LMBdet} and~(iii)).
Suppose $x_n\in\supp\mu_{n}$ along a subsequence with $x_n\to x$. If $x=0$ then
$x\in\supp\mu$ by~(iii). If $x\neq0$, set $\rho_0:=|x|/2>0$, so $|x_n|\ge\rho_0$
for large $n$. For any $0<\delta\le\rho_0$ and large $n$,
$\overline B(x_n,\delta)\subseteq\overline B(x,2\delta)$, hence by
\eqref{eq:LMBdet} (applicable since $|x_n|\ge\rho_0$),
\[
\mu_n\bigl(\overline B(x,2\delta)\bigr)\ \ge\
\mu_n\bigl(\overline B(x_n,\delta)\bigr)\ \ge\ c(\rho_0,\delta)\ >\ 0 .
\]
By the portmanteau theorem for the closed ball $\overline B(x,2\delta)$
--- in the closed-set form $\mu(F)\ge\limsup_n\mu_n(F)$, which is where the
hypothesis that $\mu_n,\mu$ are \emph{finite} measures enters: it forces
$\mu_n(\R^d)\to\mu(\R^d)$, so no mass escapes and the closed-set inequality
holds ---,
$\mu(\overline B(x,2\delta))\ge\limsup_n\mu_n(\overline B(x,2\delta))\ge
c(\rho_0,\delta)>0$. As this holds for every small $\delta>0$,
$x\in\supp\mu$. In all cases the limit of a convergent sequence
$x_n\in\supp\mu_n$ lies in $\supp\mu$.

\emph{Conclusion} \textup{(U)}. If \textup{(U)} failed there would be
$\eps_0>0$ and, along a subsequence, points $x_n\in\supp\mu_n$ with
$\dist(x_n,\supp\mu)\ge\eps_0$. By~(ii) the $x_n$ are bounded, so a
sub-subsequence converges, $x_n\to x_\ast$; the upper inclusion gives
$x_\ast\in\supp\mu$, contradicting $\dist(x_n,\supp\mu)\ge\eps_0$. Hence
\textup{(U)} holds.

\emph{Conclusion} \textup{(L)}. Fix $\eps>0$. By compactness cover
$\supp\mu$ by finitely many balls $B(z_i,\eps)$, $z_i\in\supp\mu$,
$i=1,\dots,m$. By the lower inclusion each $z_i$ is a limit of points of
$\supp\mu_n$, so for all large $n$, $\dist(z_i,\supp\mu_n)<\eps$ for every $i$.
Any $y\in\supp\mu$ lies within $\eps$ of some $z_i$, whence
$\dist(y,\supp\mu_n)<2\eps$ for all large $n$; this gives \textup{(L)}. Together
\textup{(U)} and \textup{(L)} are Hausdorff convergence.
\end{proof}

\subsection{Range convergence}

We now prove Theorem~\ref{thm:mainD}, under the hypotheses $d>6$,
\eqref{eq:2pt} and \eqref{eq:guiding-kp}, which are all valid for the
nearest-neighbour model with $d\ge 11$; the lower mass bound proved in
Section~\ref{sec:lmb} enters through Lemma~\ref{lem:nothin} and requires no
additional hypothesis.

\begin{proof}[Proof of Theorem~\ref{thm:mainD}]
\emph{Tightness and Skorokhod representation.}
Write $\bP_n:=\bP(\,\cdot\mid\mu_n(\R^d)>\eps)$ and
$\bN^{(\eps)}:=\bN(\,\cdot\mid\mu(\R^d)>\eps)$, and set
$F_n:=n^{-1}\mathcal C=\supp\mu_n$. The first marginal converges,
$\mu_n\Rightarrow\mu$ under $\bP_n\to\bN^{(\eps)}$, by
Corollary~\ref{cor:conditionalmass}. The space $\mathcal
K(\R^d)$ is not compact, so tightness of the set coordinate is not automatic;
it is supplied instead by the diameter tightness of
Lemma~\ref{lem:diamtight}. Indeed, for $\theta>0$ pick $K=K(\theta)$ with
$\limsup_n\bP_n(F_n\not\subset\overline B_K)\le\theta$; the family
$\{A\in\mathcal K(\R^d):A\subseteq\overline B_K\}$ is Hausdorff-compact ---
the nonempty compact subsets of a fixed compact set form a compact space in
the Hausdorff metric (see \cite[Theorem~3.85(3)]{AB06}) ---, so
$F_n$ is tight in $\mathcal K(\R^d)$, and hence so is the pair $(\mu_n,F_n)$ in
$\cM\times\mathcal K(\R^d)$. Let $(\mu,F)$ be a pair distributed according to any joint subsequential
limit law. It suffices to show that $F=\supp\mu$ almost surely under every
such limit law: the limit law is then that of $(\mu,\supp\mu)$ under
$\bN^{(\eps)}$, the same for every subsequence, and tightness upgrades this to
convergence of the full sequence. The almost-sure arguments below are carried
out on a Skorokhod representation of one fixed convergent (further)
subsequence; since their conclusion identifies the limit \emph{law}, nothing
is lost in passing to it.

\emph{Auxiliary thin-region coordinates.} To transport the thin-region
information of Lemma~\ref{lem:nothin} to the almost-sure limit, we track it
through a countable family of $\{0,1\}$-valued coordinates. Index by the countable parameter set
$\mathcal Q:=\{(\rho,R,\delta)\in\mathbb Q^3:0<\rho<R,\ 0<\delta\le\rho/(5\sqrt d)\}$
and $\eta\in\mathbb Q\cap(0,1]$ the thin-region events
\[
\Omega_n(\rho,R,\delta,\eta)=\Bigl\{\exists\,x\in \mathcal C:\
\rho n\le|x|\le Rn,\ \bigl|\mathcal C\cap \overline B(x,3\sqrt d\,\delta n)\bigr|\le\eta(\delta n)^4\Bigr\},
\qquad Y_n^{(\rho,R,\delta,\eta)}:=\mathbf 1_{\Omega_n^c}.
\]
Augment the joint convergence with this countable family of $\{0,1\}$-valued
coordinates: along the chosen subsequence, by passing to a further
subsequence and applying the Skorokhod representation theorem on the Polish
space $\cM\times\mathcal K(\R^d)\times\{0,1\}^{\mathcal Q\times(\mathbb Q\cap(0,1])}$, we may assume that almost surely $\mu_n\to\mu$ weakly, $F_n\to F$
in the Hausdorff metric on $\mathcal K(\R^d)$ with $F$ compact ---
whence, a Hausdorff-convergent sequence of compact sets with compact limit
being uniformly bounded,
\[
R_\ast:=\sup_n\sup\{|y|:y\in F_n\}<\infty
\]
on this almost-sure event ---, and each
$Y_n^{(\rho,R,\delta,\eta)}\to Y^{(\rho,R,\delta,\eta)}\in\{0,1\}$.
(Tightness of the laws of the full vector follows from that of
$(\mu_n,F_n)$, the remaining coordinates living in a compact space; Prokhorov's
theorem and the Skorokhod representation theorem then apply.)
Writing
$b_\eps:=\bN(\mu(\R^d)>\eps)\in(0,\infty)$, we have
$\bP(\mu_n(\R^d)>\eps)\ge\tfrac12 b_\eps n^{-2}$ for large $n$ by
Corollary~\ref{cor:conditionalmass}; so by Lemma~\ref{lem:nothin},
\begin{equation}\label{eq:Omegaprob}
\bP\bigl(Y^{(\rho,R,\delta,\eta)}=0\bigr)\le\limsup_n\bP_n(\Omega_n)
\le 2b_\eps^{-1}\,C(R,\rho,\delta)\,\eta^{1/4}.
\end{equation}
Indeed, the second inequality is Lemma~\ref{lem:nothin} (applied with
$\kappa=\rho$, $K=R$) divided by the survival lower bound displayed above; the
first holds because $Y_n^{(\rho,R,\delta,\eta)}\to Y^{(\rho,R,\delta,\eta)}$ in
the discrete set $\{0,1\}$ almost surely, so Fatou's lemma applies to
$\mathbf 1_{\Omega_n}$.)

\emph{The limit satisfies the lower mass bound a.s.} Fix
$(\rho,R,\delta)\in\mathcal Q$. By \eqref{eq:Omegaprob}, along a sequence
$\eta_j\downarrow0$ we have $\sum_j\bP(Y^{(\rho,R,\delta,\eta_j)}=0)<\infty$ if
we thin $\eta_j$ so that $\sum_j\eta_j^{1/4}<\infty$; by Borel--Cantelli,
almost surely $Y^{(\rho,R,\delta,\eta_j)}=1$ for all large $j$, i.e.\ there is
a (random) $\eta_0=\eta_0(\rho,R,\delta)>0$ with
$Y^{(\rho,R,\delta,\eta_0)}=1$. Since $Y_n^{(\rho,R,\delta,\eta_0)}\to1$ and
these are $\{0,1\}$-valued, $Y_n^{(\rho,R,\delta,\eta_0)}=1$ for all large $n$,
meaning: no $x\in \mathcal C$ with $\rho n\le|x|\le Rn$ has
$|\mathcal C\cap \overline B(x,3\sqrt d\,\delta n)|\le\eta_0(\delta n)^4$. Intersecting over the
countable set $\mathcal Q$, almost surely this holds simultaneously for all
$(\rho,R,\delta)\in\mathcal Q$. We refer to the thin-region event $\Omega_n(\rho,R,\delta,\eta)$ as the
\emph{exclusion event at} $(\rho,R,\delta,\eta)$; the Borel--Cantelli argument
above thus shows that, almost surely, for every $(\rho,R,\delta)\in\mathcal Q$
there is $\eta_0>0$ such that for all large $n$ the exclusion event at
$(\rho,R,\delta,\eta_0)$ does not occur.

\emph{Verification of the hypotheses of Lemma~\ref{lem:suppconv}.} Work
on the almost-sure event above, on which $\mu_n\to\mu$ weakly,
$R_\ast<\infty$, and the thin-region exclusions hold for every
$(\rho,R,\delta)\in\mathcal Q$. Since $\mu(\R^d)>\eps$ under
$\bN^{(\eps)}$, $\mu\neq0$ a.s. We check the three hypotheses.

\emph{(i) Inner-uniform lower mass bound \eqref{eq:LMBdet}.} Fix
$\rho_0>0$ and $\delta'>0$. Choose rationals $\rho,R,\delta$ with
$0<\rho\le\rho_0$, $R>R_\ast$, $0<\delta\le\rho/(5\sqrt d)$ and
$3\sqrt d\,\delta\le\delta'$
(possible since $R_\ast<\infty$). Any $y\in\supp\mu_n=n^{-1}\mathcal C$ with
$|y|\ge\rho_0$ is $y=x/n$ for some $x\in \mathcal C$ with $\rho n\le\rho_0 n\le|x|$
and $|x|=n|y|\le R_\ast n<Rn$; off the exclusion event at
$(\rho,R,\delta,\eta_0)$ this $x$ satisfies $|\mathcal C\cap \overline B(x,3\sqrt d\,\delta n)|>\eta_0(\delta n)^4$, where $\eta_0=\eta_0(\rho,R,\delta)>0$ is the threshold fixed above, so
\[
\mu_n\bigl(\overline B(y,\delta')\bigr)\ \ge\
\mu_n\bigl(\overline B(y,3\sqrt d\,\delta)\bigr)
=\frac{|\mathcal C\cap \overline B(x,3\sqrt d\,\delta n)|}{\mathfrak a n^4}
>\frac{\eta_0\delta^4}{\mathfrak a}{}=:c(\rho_0,\delta')>0,
\]
for all large $n$, uniformly over all such $y$. This is the inner-uniform
bound~(i).

\emph{(ii) Uniform boundedness} is $R_\ast=\sup_n\sup\{|y|:y\in
\supp\mu_n\}<\infty$, established above from the Hausdorff convergence
$F_n\to F$ to a compact limit.

\emph{(iii) Anchored core} $0\in\supp\mu$ holds almost surely by
Proposition~\ref{prop:range} (the range of $\bN$ contains the origin).

Lemma~\ref{lem:suppconv} now gives $\supp\mu_n\to\supp\mu$ in the
Hausdorff metric, almost surely. Since also $F_n=\supp\mu_n\to F$ in
$\mathcal K(\R^d)$ and Hausdorff limits are unique, $F=\supp\mu$ a.s., as
required.
\end{proof}

\begin{remark}[Test-function conditioning]\label{rem:rangephi}
The proof above uses only three features of the conditioning event: that its probability is of exact
order $n^{-2}$, i.e.\
$n^2\bP(\mu_n(\R^d)>\eps)\to b_\eps\in
(0,\infty)$, that the rescaled diameter is tight under the conditional
law (Lemma~\ref{lem:diamtight}), and that $\mu_n\Rightarrow\mu$ under the
conditional law. All three hold verbatim with $\mu(\R^d)$ replaced by
$\mu(\varphi_0)$ for any $\varphi_0\in\Cc$ with $\varphi_0\not\equiv0$ and admissible $\eps$,
by Corollary~\ref{cor:conditional} (and Lemma~\ref{lem:diamtight} in its
$\varphi_0$ form). Hence the joint convergence
$(\mu_n,n^{-1}\mathcal C)\Rightarrow(\mu,\supp\mu)$ in the Hausdorff metric on
$\cM\times\mathcal K(\R^d)$ also holds under
$\bP(\,\cdot\mid\mu_n(\varphi_0)>\eps)$ toward
$\bN(\,\cdot\mid\mu(\varphi_0)>\eps)$; this is the form used in the proof of
Theorem~\ref{thm:mainC} below, where $\varphi_0$ equals $1$ on a ball and is supported in $\overline B_{K/2}$
to localize the exit event.
\end{remark}

\begin{remark}[On the exponent and the constant $C_1$]
The exponent $\tfrac14$ in Theorem~\ref{thm:ulmb} is used only through
``$\to0$ as $\eta\downarrow0$''; any positive power would do. What is essential
is the uniformity over the annulus and that the constant $C_3$ depends only on
the aspect ratio $\Lambda$. The
constraint $\lVert z\rVert_\infty>3r$ in Theorem~\ref{thm:ulmb} licenses the
inner edge of the tiling annulus (whence $\delta\le\kappa/(5\sqrt d)$ in
Lemma~\ref{lem:nothin}); the
residual core $\{|x|\lesssim\delta\}$ near the origin is not covered by
Theorem~\ref{thm:ulmb} and is instead absorbed by the
fact that $0\in\supp\mu$ a.s.\ (Proposition~\ref{prop:range}), as in the
verification of \eqref{eq:LMBdet} above.
\end{remark}

The one-arm asymptotics below is then proved in full in this section,
taking only Theorem~\ref{thm:mainD} as input.

\subsection{Sharp one-arm asymptotics as a corollary}

The exit radius of the range is an a.e.-continuous functional on
$\mathcal K(\R^d)$ (with the Hausdorff metric), so the
one-arm asymptotics follows by evaluating the range convergence on an exit
event and un-conditioning.

\begin{lemma}[Atomless exit radius]\label{lem:exitcont}
Let $\Theta:=\sup_{x\in\Ran}|x|$ be the exit radius of the range under
$\bN$. For every $K>0$, $\bN(\Theta=K)=0$; in particular the exit functional
$\Psi:\mathcal K(\R^d)\to\{0,1\}$, $\Psi(F):=\mathbf 1\{F\not\subset\overline B_K\}$
--- where $F$ denotes a generic compact set and $\mathcal K(\R^d)$ carries the
Hausdorff metric --- is continuous at $\bN$-almost every realization of the range
$F=\Ran$, and likewise under $\bN(\,\cdot\mid\mu(\varphi_0)>\eps)$.
\end{lemma}

\begin{proof}
By Propositions~\ref{prop:scaling} and~\ref{prop:range}, the survival
function $K\mapsto\bN(\Theta>K)=\theta_1K^{-2}$ is finite and continuous on
$(0,\infty)$; hence
$\bN(\Theta=K)=\lim_{h\downarrow0}\bigl(\theta_{K-h}-\theta_K\bigr)=0$ for
every $K>0$. For the Hausdorff metric, $\Theta(F)=\max_{x\in F}|x|$ is itself
$1$-Lipschitz on $\mathcal K(\R^d)$ (if $d_H(F,F')\le t$ then every point of one
set is within $t$ of the other, so $|\Theta(F)-\Theta(F')|\le t$); hence
$F\mapsto\mathbf 1\{F\not\subset\overline B_K\}=\mathbf 1\{\Theta(F)>K\}$ is continuous at
every $F$ with $\Theta(F)\neq K$. The discontinuity set is thus contained in
$\{\Theta=K\}$, which is $\bN$-null; the same holds under the absolutely
continuous conditioned measure.
\end{proof}

We now prove Theorem~\ref{thm:mainC}, under the same hypotheses as
Theorem~\ref{thm:mainD}.

\begin{proof}[Proof of Theorem~\ref{thm:mainC}]
Fix $K\ge1$; we first show $\lim_n n^2\bP(0\leftrightarrow\partial
B_{Kn})=\theta_K=\theta_1K^{-2}$, and then pass to general radii by monotone
interpolation.

Fix once and for all a scale $\delta=\delta(K)\in(0,1]$ small enough
that $\delta\le K/(40\sqrt d)$ --- so that $3\sqrt d\,\delta<K/8$, and
$\delta\le\kappa/(5\sqrt d)$ for the application of Lemma~\ref{lem:nothin}
with $\kappa=K/8$ below ---, and fix
$\varphi_0\in\Cc$ with
\[
\mathbf 1_{\overline B_{(K/4)+3\sqrt d\,\delta}}\le\varphi_0\le\mathbf 1_{\overline B_{K/2}}
\]
(such a $\varphi_0$ exists because $(K/4)+3\sqrt d\,\delta<K/2$; the
former annular set $A$ is replaced by the full ball on which
$\varphi_0\equiv1$). (See Figure~\ref{fig:onearmloc}.)
The nesting of the radii is
\[
\frac K8\;<\;\frac K4\;<\;\frac K4+3\sqrt d\,\delta\;<\;\frac K2\;<\;K:
\]
Lemma~\ref{lem:nothin} is applied with inner radius $\kappa=K/8$; the deposit
balls of radius $3\sqrt d\,\delta$ around points with $|x|/n\le K/4$ stay
inside the ball $\overline B_{(K/4)+3\sqrt d\,\delta}$ on which $\varphi_0\equiv1$;
$\varphi_0$ is supported in $\overline B_{K/2}$; and $B_K$ is the exit radius.

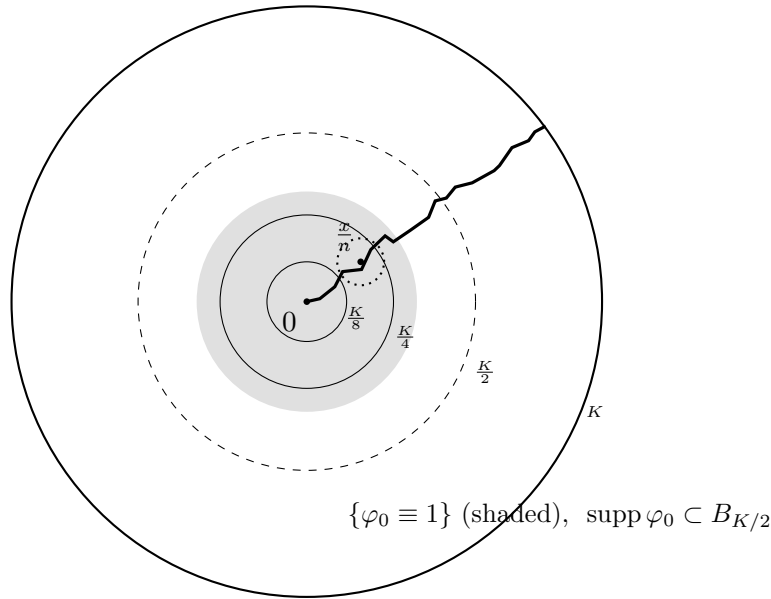
\begin{figure}[htb]
\centering
\begin{tikzpicture}[scale=0.62]
\fill[black!12] (0,0) circle (2.35); 
\draw (0,0) circle (0.85);
\draw (0,0) circle (1.85);
\draw[dashed] (0,0) circle (3.6);
\draw[thick] (0,0) circle (6.3);
\fill (0,0) circle (2pt); \node[below left] at (0,0) {$0$};
\node at (1.05,-0.35) {\tiny$\tfrac K8$};
\node at (2.1,-0.75) {\tiny$\tfrac K4$};
\node at (3.8,-1.5) {\tiny$\tfrac K2$};
\node at (6.15,-2.35) {\tiny$K$};
\draw[very thick, decorate, decoration={random steps,segment length=6pt,amplitude=2.5pt}] (0,0) -- (5.1,3.75);
\fill (1.15,0.85) circle (2pt); \node[above left=-2pt] at (1.15,0.95) {\small$\tfrac xn$};
\draw[dotted,thick] (1.15,0.85) circle (0.5);
\node[align=left] at (5.4,-4.6) {\small$\{\varphi_0\equiv1\}$ (shaded),\ \ $\supp\varphi_0\subset B_{K/2}$};
\end{tikzpicture}
\caption{Localization in the proof of Theorem~\ref{thm:mainC}. On
$\{0\leftrightarrow\partial B_{Kn}\}$ the cluster crosses the annulus
$\{K/8\le|y|\le K/4\}$; off the thin-region event, the crossing point $x$
deposits mass $\ge\eta\delta^4/\mathfrak a$ in the ball
$\overline B(x/n,3\sqrt d\,\delta)$ (dotted), which is contained in the shaded disc
$\overline B_{(K/4)+3\sqrt d\,\delta}$, on which $\varphi_0\equiv1$; so the exit event
forces $\mu_n(\varphi_0)>\eps_\eta$.}
\label{fig:onearmloc}
\end{figure} For an \emph{admissible} $\eps>0$ --- one
outside the at-most-countable set of atoms of the law of $\mu(\varphi_0)$ under
$\bN$, so that Corollary~\ref{cor:conditional} (via
Theorem~\ref{thm:restricted}) applies --- with $c_\eps:=\bN(\mu(\varphi_0)>\eps)>0$ (finite by Markov's
inequality and $\bN[\mu(\varphi_0)]=\int G\varphi_0<\infty$,
Lemma~\ref{lem:PZ})
we have $n^2\bP(\mu_n(\varphi_0)>\eps)\to c_\eps$ by the second assertion
of Corollary~\ref{cor:conditional}. The exit functional $\Psi(F)=\mathbf 1\{F\not\subset\overline B_K\}$ is bounded
and, by Lemma~\ref{lem:exitcont}, a.e.\ continuous for the conditioned limit
$\bN(\,\cdot\mid\mu(\varphi_0)>\eps)$. Testing it against the conditional
convergence of Theorem~\ref{thm:mainD} in its test-function form
(Remark~\ref{rem:rangephi}), and multiplying by
$n^2\,\bP\bigl(\mu_n(\varphi_0)>\eps\bigr)\to c_\eps$,
\[
n^2\,\bP\bigl(0\leftrightarrow\partial B_{Kn},\ \mu_n(\varphi_0)>\eps\bigr)
\;\xrightarrow[n\to\infty]{}\;\bN\bigl(\Ran\not\subset\overline B_K,\ \mu(\varphi_0)>\eps\bigr),
\]
using $\{0\leftrightarrow\partial B_{Kn}\}=\{n^{-1}\mathcal C\not\subset
\overline B_K\}=\{\supp\mu_n\not\subset\overline B_K\}$ on the support coordinate.

\emph{The right-hand side, as $\eps\downarrow0$ at this fixed
$\varphi_0$.} By monotone convergence,
\[\bN(\Ran\not\subset\overline B_K,\ \mu(\varphi_0)>\eps)\uparrow
\bN(\Ran\not\subset\overline B_K,\ \mu(\varphi_0)>0).\]
In fact,
$\{\Ran\not\subset\overline B_K\}\subset\{\mu(\varphi_0)>0\}$: $\bN$-a.e.\ we
have $0\in\Ran=\supp\mu$ (Proposition~\ref{prop:range}) and
$\varphi_0\equiv1$ on a neighbourhood of $0$, so $\mu(\varphi_0)>0$ whenever
$\mu\neq0$. Hence the right-hand side
increases to $\bN(\Ran\not\subset\overline B_K)=\theta_K$.

\emph{The left-hand side, as $\eta\downarrow0$ at this fixed
$\delta,\varphi_0$.} For $\eta\in(0,1]$ let $\mathcal E_n$ be the
thin-region event of Lemma~\ref{lem:nothin}, taken with inner radius
$\kappa=K/8$ and outer radius $K/4$ (the latter playing the role of the
parameter ``$K$'' of that lemma --- which we do not rename in the lemma, but
which is set to $K/4$ here, not to the exit radius $K$ of the present proof) and
the present $\delta,\eta$:
\[
\mathcal E_n
=\Bigl\{\exists\,x\in \mathcal C:\ \tfrac K8 n\le|x|\le\tfrac K4 n,\
|\mathcal C\cap \overline B(x,3\sqrt d\,\delta n)|\le\eta(\delta n)^4\Bigr\}
\]
(admissible since $\delta\le K/(40\sqrt d)=(K/8)/(5\sqrt d)$).
On the event $\{0\leftrightarrow\partial B_{Kn}\}$ the cluster $\mathcal C$ is
connected and contains $0$, hence meets every intermediate sphere; in
particular it has a point $x\in \mathcal C$ with $(K/8)n\le|x|\le(K/4)n$. Off
$\mathcal E_n$, that crossing point satisfies
$|\mathcal C\cap \overline B(x,3\sqrt d\,\delta n)|>\eta(\delta n)^4$, i.e.\
$\mu_n\bigl(\overline B(x/n,3\sqrt d\,\delta)\bigr)>\eta\delta^4/\mathfrak a$ ($\mathfrak a$ the
two-point constant of \eqref{eq:empirical}); since
$|x|/n\in[K/8,K/4]$, the ball $\overline B(x/n,3\sqrt d\,\delta)$ is contained in
$\overline B_{(K/4)+3\sqrt d\,\delta}$, where
$\varphi_0\equiv1$, so $\mu_n(\varphi_0)\ge\mu_n(\overline B(x/n,3\sqrt d\,\delta))
>\eta\delta^4/\mathfrak a=:\eps_\eta$. Thus
$\{0\leftrightarrow\partial B_{Kn},\ \mu_n(\varphi_0)\le\eps_\eta\}
\subset\mathcal E_n$, and by Lemma~\ref{lem:nothin} (with $\kappa=K/8$,
outer radius $K/4$ and $\delta$ fixed, so its prefactor is a constant
$C_\sharp=C_\sharp(d,K,\delta)$ --- written $C_\sharp$ to avoid the
constant $C_\star$ of Corollary~\ref{cor:carleman} ),
\[
\limsup_n n^2\,\bP\bigl(0\leftrightarrow\partial B_{Kn},\
\mu_n(\varphi_0)\le\eps_\eta\bigr)
\le\limsup_n n^2\,\bP(\mathcal E_n)
\le C_\sharp\,\eta^{1/4}\xrightarrow[\eta\downarrow0]{}0 .
\]

\emph{Conclusion.} Restrict to the countable-complement set of
$\eta$ for which $\eps_\eta=\eta\delta^4/\mathfrak a$ is an admissible level. For
such $\eta$, the displayed joint convergence at level $\eps_\eta$ gives
\[
n^2\bP\bigl(0\leftrightarrow\partial B_{Kn},\ \mu_n(\varphi_0)>\eps_\eta\bigr)
\xrightarrow[n\to\infty]{}\bN\bigl(\Ran\not\subset\overline B_K,\ \mu(\varphi_0)>\eps_\eta\bigr),
\]
while the omitted part is controlled uniformly by the thin-region estimate
just displayed:
\[
\limsup_n n^2\,\bP\bigl(0\leftrightarrow\partial B_{Kn},\ \mu_n(\varphi_0)\le\eps_\eta\bigr)
\le C_\sharp\,\eta^{1/4},\qquad C_\sharp=C_\sharp(d,K,\delta),
\]
with $C_\sharp$ from Lemma~\ref{lem:nothin}.
Adding the two and taking $\limsup_n$ and $\liminf_n$ separately,
\begin{multline*}
\bN\bigl(\Ran\not\subset\overline B_K,\ \mu(\varphi_0)>\eps_\eta\bigr)
\;\le\;\liminf_n n^2\bP(0\leftrightarrow\partial B_{Kn})
\;\le\;\limsup_n n^2\bP(0\leftrightarrow\partial B_{Kn})
\\
\;\le\;\bN\bigl(\Ran\not\subset\overline B_K,\ \mu(\varphi_0)>\eps_\eta\bigr)
+ C_\sharp\,\eta^{1/4}.
\end{multline*}
Letting $\eta\downarrow0$ along admissible levels, both ends converge to
$\bN(\Ran\not\subset\overline B_K)=\theta_K$ (the right-hand side by the monotone
convergence established above, since $\eps_\eta\downarrow0$, and
$C_\sharp\,\eta^{1/4}\to0$). Hence
\[
\lim_n n^2\,\bP\bigl(0\leftrightarrow\partial B_{Kn}\bigr)=\theta_K .
\]

\emph{Interpolation to general radii.} The previous step gives, for
each fixed $K\ge1$,
\[
(Kn)^2\,\bP\bigl(0\leftrightarrow\partial B_{Kn}\bigr)\to K^2\theta_K=\theta_1 ;
\]
take $K=1$. For $r\in[n,n+1]$, monotonicity of
$r\mapsto\bP(0\leftrightarrow\partial B_r)$ gives
\[
\Bigl(\tfrac{n}{n+1}\Bigr)^2 (n+1)^2\,
\bP\bigl(0\leftrightarrow\partial B_{n+1}\bigr)
\le r^2\,\bP\bigl(0\leftrightarrow\partial B_r\bigr)
\le\Bigl(\tfrac{n+1}{n}\Bigr)^2 n^2\,
\bP\bigl(0\leftrightarrow\partial B_{n}\bigr),
\]
and both bounds converge to $\theta_1$. Positivity and finiteness of
$\theta_1$ are part of Proposition~\ref{prop:range}.
\end{proof}

\begin{remark}[On the constant $\theta_1$]
Theorem~\ref{thm:mainC} identifies the one-arm constant as the
intrinsic SBM quantity $\theta_1=\bN(\Ran\not\subset\overline B_1)$; no closed form
for it is known. By the connections between superprocesses and semilinear
PDE going back to Dynkin \cite{Dyn91} (see also \cite[Ch.~VI]{LG99}), the
function $v(x):=\mathbb N_x(\Ran\not\subset\overline B_1)$, $|x|<1$, is the maximal
nonnegative solution of the boundary blow-up problem
$\tfrac12\Delta v=\tfrac\gamma2\,v^2$ in the open unit ball with
$v\to\infty$ at $\partial B_1$ --- a class of problems originating with
Keller \cite{Kel57} and Osserman \cite{Oss57} --- the distinction between
exiting the open and the closed unit ball being $\mathbb N_x$-null. Thus
$\theta_1=v(0)$ is characterized by a radial ODE, but is accessible only
numerically.
\end{remark}


\appendix
\section{The normalization dictionary for the $k$-point limit}\label{app:dictionary}

This appendix verifies that the convergence \eqref{eq:guiding-kp} is exactly \cite[Theorem~1]{CCHS26}, rewritten in the normalization of this paper.

The three constants entering \eqref{eq:guiding-kp} are pinned down by its
low-order cases, so that the identification $\gamma=\lambda/\mathfrak a=2d\,\mathfrak a\beta\rho$ is
not merely convention-dependent. The case $m=1$ fixes $\mathfrak a$ as the
\emph{two-point constant}, $n^{d-2}\tau_2(0,\fl{ny})\to\mathfrak a\,G(y)$ (recall
$G(y)=c_d|y|^{-(d-2)}$); $\beta=p_c/(1-p_c)$ is explicit. The case $m=2$ then
fixes $\rho$: since $|\T_2|=1$ with $I_T(y_1,y_2)=\int_{\R^d}G(w)G(y_1-w)G(y_2-w)\,dw$,
\eqref{eq:guiding-kp} reads
\[
n^{2(d-4)+2}\,\tau_3\bigl(0,\fl{ny_1},\fl{ny_2}\bigr)
\;\xrightarrow[n\to\infty]{}\;
\mathfrak a\,\lambda\int_{\R^d}G(w)G(y_1-w)G(y_2-w)\,dw,
\qquad \lambda=2d\,\mathfrak a^2\beta\rho,
\]
i.e.\ $\rho$ is the unique constant for which the rescaled three-point
function has limiting amplitude $\mathfrak a\lambda=2d\,\mathfrak a^3\beta\rho$ relative
to the normalized triangle integral $\int G(w)G(y_1-w)G(y_2-w)\,dw$.

We spell out the dictionary. Theorem~1 of \cite{CCHS26} states, for
$d>6$ and assuming the existence of the two-point limit
$\alpha_{\textsc{cchs}}:=\lim_n n^{d-2}\,\bP(0\leftrightarrow n\mathbf e_1)$,
that for \emph{distinct} $y_0,\dots,y_{k-1}\in\R^d$ (the value $0$ is a
permitted point, so pinning $y_0=0$ is a legitimate specialization),
\[
n^{(d-4)(k-1)+2}\,
\tau_k\bigl(\fl{ny_0},\dots,\fl{ny_{k-1}}\bigr)
\;\xrightarrow[n\to\infty]{}\;
\sum_{T}\alpha_{\textsc{cchs}}^{\,2k-3}\,(2d\beta\rho)^{k-2}\,
\widetilde I_T(y_0,\dots,y_{k-1}),
\]
the sum running over the binary trees on the $k$ labelled leaves
(Remark~\ref{rem:treeconvention}) and $\widetilde I_T$ denoting the tree
integral built from the \emph{bare} kernel $|u_a-u_b|^{-(d-2)}$. Their
two-point amplitude is normalized against $|x|^{-(d-2)}$, ours against
$G=c_d\,|x|^{-(d-2)}$; thus $\mathfrak a=\alpha_{\textsc{cchs}}/c_d$, while the
vertex factor $\rho$ and $\beta=p_c/(1-p_c)$ are the same in both papers.
Writing $m=k-1$, a tree in $\T_m$ has $2m-1$ edges, so
$I_T=c_d^{\,2m-1}\widetilde I_T$, and
\[
\mathfrak a\,\lambda^{m-1}I_T
=\bigl(c_d^{-1}\alpha_{\textsc{cchs}}\bigr)
\bigl(2d\,c_d^{-2}\alpha_{\textsc{cchs}}^{2}\,\beta\rho\bigr)^{m-1}
c_d^{\,2m-1}\,\widetilde I_T
=\alpha_{\textsc{cchs}}^{\,2m-1}\,(2d\beta\rho)^{m-1}\,\widetilde I_T ,
\]
which is exactly the displayed limit, since $2k-3=2m-1$ and $k-2=m-1$. In
particular no diffusion constant appears: the spatial normalization of
\cite{CCHS26} is already that of standard Brownian motion, as encoded in the
$m=1$ case above.

As a by-product, the cluster-tail constant of Remark~\ref{rem:tailknown} takes the fully explicit form
\[
\lim_{N\to\infty}\sqrt N\,\bP\bigl(|\mathcal C|\ge N\bigr)
=\sqrt{\frac{2\mathfrak a}{\pi\gamma}}=\frac{1}{\sqrt{\pi d\beta\rho}} ,
\]
using $\gamma=2d\,\mathfrak a\beta\rho$.

\section*{Acknowledgements}
This research was supported by JSPS Grant-in-Aid for Scientific Research (C) 24K06758, and the Research Institute for Mathematical Sciences, an International Joint Usage/Research Center located in Kyoto University. A.~Fribergh is supported by the Natural Sciences and Engineering Research Council of Canada (NSERC) through a Discovery Grant.

\bibliographystyle{plain}
\bibliography{references}

\end{document}